\documentclass[a4paper,11pt]{article}
\usepackage{hyperref}
\usepackage{mathtools,mathrsfs,amssymb,amsthm,array,booktabs,longtable}
\usepackage[margin=26mm]{geometry}
\usepackage{Article}
\usepackage{Symbols}
\usepackage[backend=biber,style=numeric,sorting=nyt,giveninits=true,url=false,eprint=false,isbn=false,maxbibnames=99]{biblatex}
\AtEveryBibitem{%
  \ifentrytype{article}{\clearfield{series}\iffieldundef{eid}{}{\clearfield{pages}}}{}%
  \iffieldundef{howpublished}{}{\clearfield{eprint}}%
}
\allowdisplaybreaks
\numberwithin{equation}{section}
\numberwithin{theorem}{section}
\hypersetup{hidelinks,pdftitle={Seven ordered points and the nonexistence of arithmetic ball quotient compactifications},pdfauthor={Ryunosuke Nakano and Hisato Matsukawa}}
\title{Seven ordered points and the nonexistence of arithmetic ball quotient compactifications}
\author{Ryunosuke Nakano \and Hisato Matsukawa}
\date{October 16, 2026}
\begin{document}
\maketitle
\begin{abstract}
	We prove that the ordered GIT quotient of seven points on \(\mathbb P^1\) with equal weights is not isomorphic to the Baily--Borel compactification of any arithmetic complex ball quotient.
	Combining this result with our companion paper and the known positive examples, we conclude that, for \(n\ge6\), the corresponding quotient of \(n\) points is isomorphic to the Baily--Borel compactification of an arithmetic complex ball quotient precisely when \(n=6,8\).
	We allow any arithmetic lattice and any algebraic isomorphism, as well as torsion and divisorial ramification.
	We combine arithmetic reductions with identities involving Chern classes of fixed subvarieties and intersection inequalities for nef divisors on the GIT quotient.
	In the final comparison, we include all fixed components and need no bound on the degree of an individual fixed curve.
\end{abstract}

\section{Introduction}
\label{split-second-intro}

Let
\[
	X_n=(\mathbb P^1)^n/\!/_{\mathcal O(1,\ldots,1)}\mathrm{SL}(2,\mathbb C)
\]
be the ordered GIT quotient, in which the symmetric group is not divided out, and put \(X=X_7\).
In this paper, we prove the following theorem.

\begin{theorem}
	\label{split-seven-main}
	There is no arithmetic lattice \(\Gamma\subset\mathrm{PU}(1,4)\) for which
	\[
		X\simeq(\mathbb B^4/\Gamma)^{\mathrm{BB}}
	\]
	as algebraic varieties, where \((\cdot)^{\mathrm{BB}}\) denotes the Baily--Borel compactification.
\end{theorem}

The theorem permits arbitrary torsion and divisorial ramification, and neither the branch divisor nor the action on labels is prescribed.
Our companion paper \cite{NakanoWeightedConfigurations2026}, cited in its public version arXiv:2610.23228v3, excludes such isomorphisms for every \(n\ge9\) and every vector \((\mu_1,\ldots,\mu_n)\) of rational weights with \(0<\mu_i<1\) and \(\sum_i\mu_i=2\).
Together with the Deligne--Mostow realizations at six and eight points \cite{DM,HKM,HulekMaedaEight}, this leaves seven points as the last case for equal weights.
Theorem~\ref{split-seven-main} therefore gives the following classification.
\begin{corollary}
	\label{split-equal-classification}
	For every \(n\ge6\), the variety \(X_n\) is the Baily--Borel compactification of an arithmetic ball quotient if and only if \(n=6,8\).
\end{corollary}
The result concerns abstract algebraic isomorphisms, rather than a prescribed period map; it makes no assertion about nonarithmetic lattices or other compactifications.

We now outline the proof; Sections~\ref{split-seven-geometry} and~\ref{seven-plane-section} fix the notation.
The geometry that controls the automorphic class is supplied by the 35 coincidence planes.
Each is a copy of \(\mathbb P^2\) with normal bundle \(\mathcal O(-1)^{\oplus2}\), so Taylor coefficients in the normal directions bound the restriction of that class, and this bound does not require lifting label permutations.
Sections~\ref{ew-seven-compact-section} and~\ref{ew-seven-fields-section} exclude cocompact realizations and nonuniform realizations outside the Gaussian field.
Section~\ref{ew-seven-gaussian-section} constructs a self-dual lattice preserved by the original group and proves that \(L-3H/2\) is nef.
For the group \(G(4,2,2)\), this last step uses a finite cover of degree three of a coincidence plane: its trace kernel is a twist of the normal bundle, and its cubic coefficients count the points of this plane with full stabilizer \(G_{31}\) (Lemma~\ref{ew-gaussian-g422-cover}).

Two results organize the final comparison.
Proposition~\ref{ew-chern-general} expresses the classes of fixed strata in terms of the automorphic class and the Chern classes of the coarse space, in arbitrary dimension under the hypotheses on finite volume in its statement.
Proposition~\ref{ew-quotient-models} gives linear relations that hold both for finite quotients of \(\mathbb P^3\) and for those of polarized abelian threefolds; these describe interior points and cusps, respectively.
We prove these identities here and use \cite[Propositions~2.2--2.4 and~3.1]{NakanoWeightedConfigurations2026} for the remaining inputs.
In Section~\ref{ew-seven-global-section}, the resulting linear equalities and nef inequalities have no common solution.
An explicit rational linear combination proves this incompatibility, retaining every fixed component, full normalizer, inertia character and cusp width, without bounding the degree of an individual curve.

\section{Geometry and preliminary comparisons}
\label{split-seven-geometry}

Put \(P_7=(\mathbb P^1)^7\), and let \(P_7^{\mathrm s}\) be the locus where at most three labels coincide at any point.
The effective diagonal action of \(\mathrm{PGL}(2,\mathbb C)\) is free on \(P_7^{\mathrm s}\), and stability equals semistability.
The principal bundle \(P_7^{\mathrm s}\to X\) defines
\[
	\pi_7:\mathcal P_7=(P_7^{\mathrm s}\times\mathbb P^1)/\mathrm{PGL}(2,\mathbb C)\longrightarrow X,
	\qquad[(p,z)]\longmapsto[p],
\]
with sections \(\sigma_i([p])=[(p,p_i)]\).
We put \(H=-K_X\); then
\[
	H=(2,\ldots,2),\qquad
	\operatorname{Pic}(X)=\{b\in\mathbb Z^7:\textstyle\sum_i b_i\in2\mathbb Z\}.
\]
For each of the 35 sets \(I\) of three labels, the coincidence plane \(Z_I\) of these labels is \(\mathbb P^2\), its normal bundle is \(\mathcal O_{\mathbb P^2}(-1)^{\oplus2}\), and \(H|_{Z_I}=\mathcal O_{\mathbb P^2}(1)\).
These assertions follow from the free quotient and its marked sections, as in \cite[Proposition~2.5, Equation~(2.2), and Proposition~2.6]{NakanoWeightedConfigurations2026}.

Suppose that \(X\simeq(\mathbb B^4/\Gamma)^{\mathrm{BB}}\) for an arithmetic lattice \(\Gamma\subset\mathrm{PU}(1,4)\) as in Theorem~\ref{split-seven-main}, and identify \(\mathbb B^4/\Gamma\) with the open part of \(X\).
We realize \(\mathbb B^4\) as the space of negative lines in a Hermitian space of signature \((1,4)\).
Let \(\mathcal L\) be the automorphic rational line bundle on \(\mathbb B^4/\Gamma\), normalized by \(K_{\mathbb B^4}=\mathcal L^5\), and put \(L=c_1(\mathcal L)\).
We write \(\chi_{\mathrm{orb}}=\chi(\mathbb B^4/\Gamma')/[\Gamma:\Gamma']\), where \(\Gamma'\subset\Gamma\) is a torsion-free subgroup of finite index; proportionality gives \(\chi_{\mathrm{orb}}=5\int_XL^4\).
The full normalizer of a fixed ball is its stabilizer in \(\Gamma\), and the effective image of the full normalizer is its quotient by the pointwise stabilizer~\cite[Lemma~2.12]{NakanoWeightedConfigurations2026}.
Let \(C\) be the smooth projective normalization of the Baily--Borel compactification of a fixed disk modulo the effective image of its full normalizer in \(\Gamma\), and let \(f:C\to X\) be the induced finite map.
We set
\[
	a=\deg f^*H,\qquad t=\deg f^*L,\qquad x=g(C)-1.
\]
For an invariant direction of polynomial degree \(e\) and dimension one preserved by the full normalizer, \(E_e\) denotes its equivariant bundle with all finite inertia characters, extended to the coarse curve \(C\) using the local invariant coordinates given below at each point with a larger full stabilizer and at each cusp.
Its orbifold degree is \(-et\); its degree on \(C\) adds the corrections at these points computed in these coordinates.
We write \(S\) for the reduced cusp divisor on \(C\), put \(s=\deg S\), and use subscripts for specified subsets of cusps.
When \(R\) counts the points of \(C\) with a specified type of larger stabilizer, \(D_R\) denotes the reduced divisor of these points.
We put \(\sigma=\sum_z(1-e_z^{-1})\), where \(z\) runs over the points of \(C\) with a stabilizer that is a direct product and \(e_z\) is the order of the effective inertia on \(C\) at \(z\).

We use \cite[Propositions~2.2--2.4]{NakanoWeightedConfigurations2026} for quotients of \(f^*T_X\).
Every homomorphism from \(f^*T_X\) to a line bundle of negative degree factors through differences of coincident sections.
We write the pullback of \(\mathcal P_7\) by \(f\) as \(\mathbb P_C(V)\), the projective bundle of quotients of dimension one, with quotient maps \(V\twoheadrightarrow L_i\).
We write \([ij]\) for the determinant of \(V\to L_i\oplus L_j\), a section of \((\det V)^{-1}\otimes L_i\otimes L_j\).
Suppose that the labels in \(I\), with \(|I|=m\), coincide and their common relative tangent line has degree \(-e<0\).
For \(j\notin I\), let \(D_j\) be the zero divisor of \([ij]\), where \(i\in I\), and put \(\delta=\sum_{j\notin I}\deg D_j\).
Then
\[
	2\le m\le3,\qquad a=(7-2m)e+2\delta,\qquad\delta\ge0.
\]
The remaining relative tangent lines generate the kernel of the quotient of \(f^*T_X\) onto differences of coincident sections, and \cite[Proposition~2.3]{NakanoWeightedConfigurations2026} bounds the exterior powers of this kernel.
A quotient of rank \(2\) is always required to be surjective at every point, not just at the generic point.
The nef criterion of \cite[Proposition~2.6]{NakanoWeightedConfigurations2026} must be checked on every \(Z_I\); it does not identify the fixed balls over different planes.

The original full arithmetic group and its normalized lift are those of \cite[Proposition~2.11]{NakanoWeightedConfigurations2026}.
Full stabilizers are taken in the original group itself, not in a subgroup of finite index.
We call an element, reflection, or local datum actual if it occurs in the original group, not only in the classification of finite reflection groups.
Let \(c_{\rm sc}\) be the order of the finite scalar kernel of the map from the normalized lift to \(\Gamma\).
For each fixed ball we use its full normalizer, the finite morphism of \cite[Lemma~2.12]{NakanoWeightedConfigurations2026}, and the restriction of \(\mathcal L\).
At a cusp, \(P\) denotes the full stabilizer, \(N\) its Heisenberg lattice, and \(F=P/N\) its finite horizontal group.
The horizontal torus \(A\) and its positive line bundle are defined with the least positive central period of \(N\), and the line bundle carries the action of \(F\) induced by \(P\).
For a smooth cusp, the horizontal holonomy is irreducible.
For its actual linearization, the invariant sections form a polynomial ring, whose projectivization is generally a weighted projective space.
The least positive central period gives its faithful linearization and the grading of its invariant section ring, as in \cite[Lemma~2.16]{NakanoWeightedConfigurations2026}.
The lift fixing the point in \cite[Lemma~2.17]{NakanoWeightedConfigurations2026} applies to the kernel of the actual fiber character.
We distinguish the degree of an invariant section in this grading from the polynomial degree \(e\) of an invariant of the normal representation.
At an end of a fixed curve at a cusp, let \(q\) be the coordinate of the cusp defined by the least positive central period of \(N\), and let \(Q\) be the coarse coordinate of the curve.
Then \(Q=q^{\kappa}\), where the actual relative width \(\kappa\) is the order of the image of the fiber character of the full stabilizer in \(F\) of the corresponding point of the horizontal torus~\cite[Lemma~2.17]{NakanoWeightedConfigurations2026}.
We use \cite[Proposition~3.4]{NakanoWeightedConfigurations2026} and \cite[Lemma~3.5]{NakanoWeightedConfigurations2026} only after specifying the larger stabilizers, the full horizontal stabilizers at the ends, the pointwise fixed groups lifted from them as kernels of finite characters, and the widths given by these characters.

The dependencies on the companion paper are summarized in Table~\ref{split-table-dependencies}.
All statement numbers in this table refer to arXiv:2610.23228v3; the hypotheses stated there remain in force.
\begin{table}[htbp]
	\centering
	\small
	\setlength{\tabcolsep}{4pt}
	\begin{tabular}{@{}p{.16\linewidth}p{.39\linewidth}p{.39\linewidth}@{}}
		\toprule
		Source                              & Hypotheses and conclusion used                                                                                                                                                                                                       & Use in this paper                                                                                          \\
		\midrule
		Propositions 2.2--2.4               & Complete curves in the stable quotient; homomorphisms to negative lines factor through coincident sections, with the stated kernel and determinant bounds.                                                                           & Tangent and normal quotients.                                                                              \\
		Propositions 2.5--2.6               & The marked family and the coincidence subvarieties; the nef criterion requires every relevant coincidence plane.                                                                                                                     & The classes of \(X_7\) and the estimate for \(L-3H/2\).                                                    \\
		Proposition 2.11                    & A hypothetical arithmetic realization; the integral model and scalar normalization retain the original full group.                                                                                                                   & Arithmetic reductions.                                                                                     \\
		Lemma 2.12                          & A fixed ball with its full generic stabilizer; its full normalizer has finite volume and a finite map, with the original automorphic bundle restricted.                                                                              & All fixed components; its additional product hypothesis is required for the stated ordinary normal bundle. \\
		Proposition 2.14 and Lemma 2.15     & The product hypothesis for persistence, and nonuniformity for embedding into the horizontal group of an actual cusp of the original lattice, preserving the reflection representation and orders, with the rank bounds stated there. & Reduction of the possible stabilizers.                                                                     \\
		Lemmas 2.16--2.17                   & A smooth cusp with its actual central period and fiber action; a finite lift is obtained for the kernel of the fiber character.                                                                                                      & Degrees of invariant theta sections, finite characters, and the relative width.                            \\
		Propositions 3.1, 3.4 and Lemma 3.5 & Full normalizers, all larger stabilizers, the full horizontal stabilizers at the ends of the curve, the pointwise fixed groups lifted from them as kernels of finite characters, and these characters, whose images give the widths. & Determinant and monomial comparisons, after their hypotheses are checked.                                  \\
		Lemma 3.8                           & At a fixed point of the horizontal torus, the width is the greatest common divisor of the degrees of the generating invariant theta sections that do not vanish there.                                                               & The width at a cusp in Lemma~\ref{ew-gaussian-g413}.                                                       \\
		Equation 3.13                       & The same restricted automorphic class on a fixed ball.                                                                                                                                                                               & The Chern polynomial used in Proposition~\ref{ew-chern-general}.                                           \\
		\bottomrule
	\end{tabular}
	\caption{Results of the companion paper used in this paper.}
	\label{split-table-dependencies}
\end{table}

\begin{lemma}
	\label{ew-git-seven-classes}
	For \(X=X_7\), let \(\mathcal A\) be the adjoint bundle of the principal bundle \(P_7^{\mathrm s}\to X\), with fiber \(\mathfrak{sl}_2(\C)\).
	Put \(\psi_i=c_1(\sigma_i^*T_{\mathcal P_7/X})\) and \(p=-c_2(\mathcal A)\).
	Then, in the rational Chow ring of \(X\),
	\[
		H=\sum_i\psi_i,\qquad \psi_i^2=p,\qquad
		c_2(T_X)=\frac{H^2-5p}{2},\qquad
		c_3(T_X)=\frac{H^3-13Hp}{6}.
	\]
\end{lemma}

\begin{proof}
	The section \(\sigma_1\) represents the universal family as \(\Ps_X(\mathcal V)\), where \(\mathcal V=\pi_{7*}\mathcal O_{\mathcal P_7}(\sigma_1(X))\) has rank \(2\).
	For each quotient \(\mathcal V\twoheadrightarrow\mathcal L_i\), the relative tangent line has first Chern class \(2c_1(\mathcal L_i)-c_1(\mathcal V)\), whose square equals \(c_1(\mathcal V)^2-4c_2(\mathcal V)=-c_2(\operatorname{End}_0(\mathcal V))=p\).
	The tangent sequence is \(0\to\mathcal A\to\bigoplus_i\sigma_i^*T_{\mathcal P_7/X}\to T_X\to0\).
	Since \(c_1(\mathcal A)=c_3(\mathcal A)=0\), expanding \(c(T_X)=\prod_i(1+\psi_i)/(1-p)\) through degree \(3\) proves the formulas.
\end{proof}

\section{The automorphic degree on the coincidence planes}
\label{seven-plane-section}
Throughout this section, \(X=X_7\), \(H=-K_X\), and \(\mathcal L\) is the automorphic line bundle.
For a curve \(f:C\to X\), we write \(a=\deg f^*H\) and \(t=\deg f^*\mathcal L\).
The degree of \(\mathcal L\) is understood on the orbifold until an indicated power descends to \(X\).
For a fixed curve, \(\mathcal N\) denotes \(f^*T_X/(df(T_C))^{\mathrm{sat}}\).
If the three labels of a set \(I\) coincide on a fixed curve, so that \(f\) factors through \(f_I:C\to Z_I\), then \(J\) denotes the line bundle \(f_I^*T_{Z_I}/T_f\), where \(T_f=(df(T_C))^{\mathrm{sat}}\).
For an invariant direction of polynomial degree \(e\) preserved by the full normalizer, \(E_e\) denotes its line bundle with the finite inertia characters included; its degree is stated in each calculation.
In the tables and the text, a string such as \(1112\) abbreviates the tuple \((1,1,1,2)\) of generator degrees; it does not by itself determine the horizontal group or its linearization.
The rank of a finite linear group is the codimension of its fixed subspace.
We use \(G_j\) for the Shephard--Todd group numbered \(j\), and \(A_j,B_j\) for the usual Weyl groups.
We write \(\mu_b=\{\zeta\in\mathbb C:\zeta^b=1\}\), and \(I_2(m)\) denotes the dihedral Coxeter group of order \(2m\).
For positive integers \(m,p,r\) with \(p\mid m\), \(G(m,p,r)\) consists of the monomial matrices of size \(r\) whose nonzero entries lie in \(\mu_m\) and have product in \(\mu_{m/p}\).
\begin{lemma}
	\label{seven-plane-gauss}
	Suppose that \(X\) is the Baily--Borel compactification of an arithmetic ball quotient, and let \(Z\simeq\mathbb P^2\) be a coincidence plane.
	Write \(L|_Z=\lambda c_1(\mathcal O_Z(1))\).
	If the nontrivial part of the normal representation of the full generic stabilizer has rank 0 or 1, then \(\lambda\ge1\).
	If it has rank 2 and its least positive invariant degree is \(q\), then \(q\lambda\ge1\).
\end{lemma}
\begin{proof}
	On a neat cover \(Y=\mathbb B^4/\Gamma'\), where \(\Gamma'\) is a neat normal subgroup of finite index in \(\Gamma\), resolve a component above \(Z\), with maps \(h:S\to Y\) and \(\tau:S\to Z\).
	Let \(G_0\) be the stabilizer of the normalized component in the deck group, and let \(K_0\) be the kernel of its action on that component.
	The quotient \(G_0/K_0\) acts faithfully and transitively on a general fiber, so the morphism from the normalized component to \(Z\) restricts over a suitable dense open subset of \(Z\) to a finite Galois cover with group \(G_0/K_0\).
	Let \(R\) be the common divisor of the maximal minors of \(dh\).
	After removing \(R\), the Gauss equation gives a positive current in \(c_1(K_S-R-3h^*\mathcal L)\), whose potential is locally bounded above.
	On the inverse image \(\tau^{-1}(\ell)\), of degree \(e\), of a general line \(\ell\subset Z\) avoiding cusps and exceptional images,
	\begin{equation}
		\label{seven-plane-gauss-current}(K_S-R)\cdot\tau^*\mathcal O_Z(1)\ge3\lambda e.
	\end{equation}
	For ranks \(0,1\), apply \cite[Lemma~3.10]{NakanoWeightedConfigurations2026} with \(k=c=2\).
	The normal bundle is \(\mathcal O_Z(-1)^{\oplus2}\), the set of cusps on \(Z\) is finite, and the full normalizer retains the restricted \(\mathcal L\) by \cite[Lemma~2.12]{NakanoWeightedConfigurations2026}.
	Thus \(\lambda\ge1\).

	For rank \(2\), the fixed ball maps finitely to the punctured plane.
	The full normalizer acts with finite image on the positive normal space \(U\), and \(\underline U\) denotes the associated flat bundle on the neat cover.
	The fixed ball has normal bundle \(\mathcal L^{-1}\otimes\underline U\)~\cite[Lemma~2.12]{NakanoWeightedConfigurations2026}.
	Let \(T\) be the full generic stabilizer, put \(R_T=\operatorname{Sym}(U^\vee)^T\), and let \(R_{T,+}\) be its ideal of positive degrees.
	We denote by \(\underline U_q\) the bundle associated with the dual of \((R_{T,+}/R_{T,+}^2)_q\), retaining the finite action of the full normalizer.
	If \(q\) is the least positive invariant degree and \(k=\operatorname{rank}\underline U_q\), the first normal Taylor coefficients of the plane ideal give a map \(\mathcal L^{-q}\otimes\underline U_q\to\tau^*\mathcal O_Z(-1)^{\oplus2}\) of generic rank \(k\).
	Kill the finite monodromy on the compact curve \(\tau^{-1}(\ell)\); a negative coefficient degree would make this map zero.
	Thus \(q\lambda\ge1\).
\end{proof}
The following two assertions concern a nonuniform arithmetic realization over an imaginary quadratic field \(K\).
We put \(w_K=|\mu(K)|\), where \(\mu(K)\) is the group of roots of unity in \(K\); thus \(w_K\in\{2,4,6\}\).
\begin{lemma}
	\label{ew-fields-plane-bound}
	Let \(Z\simeq\mathbb P^2\) be a coincidence plane defined by three labels, and put \(L|_Z=\lambda c_1(\mathcal O_Z(1))\).
	If the normal representation of its full generic stabilizer has rank 0, then \(\lambda\ge1\).
	If it has rank 1 and its image has order \(b\), then \((1-b^{-1})\lambda\ge1/3\).
	If it has rank 2 and the least positive invariant degree is \(q\), then \(q\lambda>1\).
	In all cases \(w_KL-2H\) is nef, and hence \(a\le w_Kt/2\) on every curve.
\end{lemma}
\begin{proof}
	\label{ew-finite-plane-bound}
	By Lemma~\ref{seven-plane-gauss}, only the case \(q\lambda=1\) in rank 2 remains.
	Retain its notation; since \(\mathcal L^{c_{\rm sc}}\) descends, \(q\mid c_{\rm sc}\).
	The bundle \(\mathcal S=\mathcal L^q\otimes\tau^*\mathcal O_Z(-1)\) is torsion, with an equivariant trivialization of its \(c_{\rm sc}\)th power.
	The bundle \(\mathcal F=\mathcal S^{-1}\otimes\underline U_q\) on \(S\) is flat with finite monodromy.
	The central scalar \(\zeta\) acts by \(\zeta^q\) on both factors before taking their quotient; thus \(\mathcal F\) and the map \(\mathcal F\to\mathcal O_S^{\oplus2}\) are equivariant for the actual stabilizer of the component, acting projectively.
	Constancy on the compact curves \(\tau^{-1}(\ell)\) after a finite cover gives rank \(k\) at every interior point; the source \(\mathcal F\) of this injective map then has trivial point inertia.
	It descends to the simply connected punctured plane as a trivial flat bundle, and Hartogs's theorem extends its constant rank \(k\) matrix to \(Z\).
	\cite[Lemma~2.15]{NakanoWeightedConfigurations2026} supplies an actual cusp of this fixed ball.
	In a unipotent cover with horizontal normal coordinates \(y\in U\), let \(q_{\rm cov}\) be the fiber coordinate defining the zero section on this auxiliary cover; the plane equations and their Taylor coefficients of degree \(q\) are divisible by \(q_{\rm cov}\), since the zero section contracts to the cusp.
	The trivialization of \(\mathcal S^{c_{\rm sc}}\) is pulled back from \((\mathcal L^{c_{\rm sc}}|_Z)^q\simeq\mathcal O_Z(c_{\rm sc})\) on the whole plane, using the Baily--Borel extension of \cite[Proposition~2.11]{NakanoWeightedConfigurations2026}.
	In the frame of \(\mathcal L\) on a Siegel domain at the cusp and a regular frame of \(\mathcal O_Z(-1)\), this trivialization is a unit at \(q_{\rm cov}=0\).
	Once the finite coefficient action on the same cover is killed, its local \(c_{\rm sc}\)th root is a unit, so the matrix of Taylor coefficients of degree \(q\) still vanishes at \(q_{\rm cov}=0\).
	The constant matrix of rank \(k>0\) therefore vanishes along a vertical ray, a contradiction.

	Finally \(q\le w_K\) in every possible group of rank 2.
	For monomial groups this follows from their degrees \(m,2m/p\) and their character fields; a dihedral group has \(q=2\).
	For primitive groups, \cite[Lemma~2.15]{NakanoWeightedConfigurations2026} excludes \(G_5,G_8,G_{12}\): the primitive groups of rank 3 over quadratic fields are \(G_{24},G_{25},G_{26}\) and their only primitive parabolic group of rank 2 is \(G_4\), of degrees \(4,6\)~\cite[Table 1]{GeckMalle}, \cite[Table 4]{KrishnasamyTaylor}.
	Thus the remaining primitive case has \(K=\mathbb Q(\omega)\) and \(q=4\le6\).
	The integer \(w_K\lambda\) is strictly greater than 1 in rank 2 and at least 2 in ranks 0 and 1.
	\cite[Proposition~2.6]{NakanoWeightedConfigurations2026} applied to the 35 planes gives the asserted nef inequality.

\end{proof}
\begin{lemma}
	\label{ew-fields-plane-equality}
	In the setting of Lemma~\ref{ew-fields-plane-bound}, if the full generic stabilizer has rank 0 or 1, then \(\lambda>1\).
\end{lemma}
\begin{proof}
	\label{ew-finite-plane-equality}
	Suppose \(\lambda=1\) in rank 0 or 1.
	The equality case of \cite[Lemma~3.10]{NakanoWeightedConfigurations2026}, with the same \(k=c=2\), makes the regular inverse image locally a totally geodesic ball of dimension 2, inside the generic fixed ball in rank 1.
	Let \(B=\mathbb B(W)\) be this ball in the universal cover.
	The identity theorem makes \(B\) an entire irreducible component of the closed inverse image of \(Z\).
	Local finiteness of these components proves properness of the quotient \(M_0\) of \(B\) by its stabilizer; its generically injective immersion thus normalizes the component and maps finitely to the punctured plane.
	Remove deep cusp neighborhoods with compact complement~\cite{GarlandRaghunathan}.
	Their intersections with \(B\) are horoballs if the isotropic vector lies in \(W\), and bounded geodesic balls or empty otherwise~\cite{Parker}.
	In the latter case write the isotropic vector of the cusp as \(\rho=\rho_W+\rho_U\in W\oplus W^\perp\); then the Hermitian form \(h\) satisfies \(h(\rho_W,\rho_W)=-h(\rho_U,\rho_U)<0\), and the height is \(u=-h(z,z)/|h(\rho_W,z)|^2\).
	Properness over a compact regular horosphere gives finitely many cuts, with compact horizontal cross sections for the horoballs.
	Finite volume follows from \(\int_{u_0}^\infty u^{-3}\,du<\infty\) on the horoballs, the boundedness of the other cuts, and the compact complement.

	Borel density makes the stabilizer on \(W\) irreducible~\cite{BorelDensity}.
	The \(K\)-algebra of its actual linear lifts is semisimple, since its complexification acts faithfully and completely reducibly on \(W\oplus W^\perp\).
	Its unique central idempotent of rank 3 projects onto \(W\); the positive complement has dimension 2, so Galois invariance of that rank defines both summands over \(K\).
	The preserved lattice makes the normal action of the full stabilizer finite, and the original full stabilizer gives a finite birational map onto the normal punctured plane, hence an isomorphism.

	For rank 1, let \(U_0\) be the line in the normal space fixed by the generic cyclic group; for rank 0 let it be the entire positive normal space.
	The first Taylor coefficients of the plane ideal give a generically injective map \(\mathcal L^{-1}\otimes\underline U_0\to\tau^*\mathcal O_Z(-1)^{\oplus2}\).
	Apply the argument excluding equality in Lemma~\ref{ew-fields-plane-bound} with \(q=1\) and \(\underline U_q=\underline U_0\).
	Put \(\mathcal S=\mathcal L\otimes\tau^*\mathcal O_Z(-1)\).
	The fixed ball is now defined over \(K\), its full normal action is finite, and the two factors have the same central scalar character.
	The first derivatives also exclude an additional inverse branch by the implicit function theorem.
	Thus their quotient descends to a nonzero constant matrix on the punctured plane.
	Meyer's theorem in six variables supplies an actual cusp, and the same trivialization of \(\mathcal S^{c_{\rm sc}}\) on the whole plane is a unit in the frame of \(\mathcal L\) on a Siegel domain at this cusp.
	The first Taylor coefficients are divisible by \(q_{\rm cov}\), giving the same contradiction.

\end{proof}
A coincidence plane with full generic stabilizer of rank 2 is the compactification of its fixed ball quotient: \cite[Lemma~2.12]{NakanoWeightedConfigurations2026} gives a finite birational map, and the plane is normal.

\section{Seven points and cocompact lattices}
\label{ew-seven-compact-section}
Throughout this section we assume that \(X=X_7\) is the quotient of \(\mathbb B^4\) by a cocompact arithmetic lattice \(\Gamma\).
We retain \(H=-K_X\) and \(L=c_1(\mathcal L)\) with \(K_{\mathbb B^4}=\mathcal L^5\).
The branch divisor, its indices, and the isomorphism with \(X\) are arbitrary.
\cite[Proposition~2.11]{NakanoWeightedConfigurations2026} gives a normalized lift generated by reflections over a CM field \(K\); let \(F\) be its maximal real subfield and \(D_K,D_F\) their absolute discriminants.
Every stabilizer below is the full stabilizer, and every quotient of a fixed ball uses the full normalizer.
On each fixed plane, \(d_k\) is the total degree of the branch components of ramification index \(k\).
Lemma~\ref{ew-compact-groups} first forces a factor of rank \(4\).
The normal quotient comparisons eliminate the remaining groups except \(F_4\); Lemmas~\ref{ew-compact-primitive}, \ref{ew-compact-hfour}, and~\ref{ew-compact-g29} then force a point with full stabilizer \(F_4\) and exclude \(G_{29}\).
Lemma~\ref{ew-compact-bthree} yields \(L\equiv H\), contradicting the arithmetic volume calculation in Lemma~\ref{ew-compact-770}.

For a fixed curve \(f:C\to X\) we put \(x=g(C)-1\), \(a=\deg f^*H\), and \(t=\deg f^*L\).
When \(df\) vanishes on a divisor \(D\), we use \(\mathcal N=f^*T_X/T_C(D)\).

\subsection{Finite stabilizers}
\label{ew-compactapp-section}
On a fixed curve or plane with full generic stabilizer \(T\), an upper point is a point whose full stabilizer, an upper group, properly contains \(T\); these points form the upper locus.
The parabolic classifications \cite[Table 1]{GeckMalle}, \cite[Theorem 3.6 and Tables 5--9]{KrishnasamyTaylor} give the irreducible upper groups in Table~\ref{ew-table-upper-groups}; products \(T\times\mu_b\) retain every positive integer \(b\) before reflection orders are restricted.
\begin{table}[htbp]
	\centering
	\(\begin{array}{c|l|c}
		T&\text{irreducible upper groups}&\text{effective inertia}\\\hline
		G(m,p,3),\ m\ge6&G(m,p,4)&m\\
		G(m,1,3),\ 3\le m\le5&G(m,1,4)&m\\
		G(5,5,3)&G(5,5,4)&5\\
		G(3,3,3)&G(3,3,4),G(3,1,4)&3,3\\
		B_3&B_4,F_4,G_{29}&2,2,4\\
		G(4,2,3),G(4,4,3)&G(4,2,4),G_{31};\ G(4,4,4),G_{29}&4\\
		G_{25},H_3&G_{32};\ H_4&6;\ 2\\
		A_3&A_4,D_4,F_4,H_4,G_{29},G_{31}&
	\end{array}\)
	\caption{Irreducible upper groups and orders of effective inertia.}
	\label{ew-table-upper-groups}
\end{table}
The full list of irreducible groups of rank \(4\) is \(A_4\), \(G(m,p,4)\), \(F_4,H_4,G_{29},G_{31},G_{32}\).
By Lemma~\ref{ew-compact-groups}, for a full generic stabilizer \(T\) of rank \(3\), the product upper groups are \(T\times\mu_2\) and, when \(T=G_{25}\), also \(G_{25}\times\mu_3\); the rank \(2\) factors are products, \(A_2,B_2,I_2(5),G(4,2,2),G_4\).
A factor with reflection orders \(2,3\) cannot enter rank \(4\); on a fixed plane it is \(G(6,2,2)\).
For \(G(m,p,2)\), the invariants \(S=x^m+y^m\), \(P=(xy)^{m/p}\) give the branch equation \(S^2-4P^p=0\) and, for \(p<m\), \(P=0\) of index \(m/p\).
For the full effective local group \(\Gamma_x\) on the plane at a point \(x\), not in the ambient quotient, a mirror with \(\Gamma_x\)-orbit of size \(\rho\) and pointwise stabilizer of order \(b\) has effective normalized order \(|\Gamma_x|/(b\rho)\), separately on every branch of a self-intersection \cite[Table 1]{Tan}.
The effective normalized orders for \(A_2,I_2(5),B_2,G(4,2,2),G(6,2,2),G_4\) are \(1,1,2,4,6,2\).
The groups \(A_2,I_2(5),G_4\) give cusps \((2,3),(2,5),(2,3)\); \(B_2\) gives two branches of contact \(2\), and \(G(4,2,2),G(6,2,2)\) give ordinary triple points.
A product \(\mu_b\times\mu_c\) gives a node of effective normalized orders \(c,b\).
For planes of degree below \(1/2\), the only pairs \((T,G)\) with irreducible upper group \(G\) are \((G(4,2,2),G(4,2,3))\), \((G_4,G_{25})\), \((\mu_3^2,G_{25})\).
For the first and third pairs, the degree corrections and normal determinant orders are \((1/2,2)\) and \((1/2,3)\), respectively; for the second pair, the image \((v^2,v^3)\) of the line fixed by \(G_4\) under the invariants of \(G_{25}\) is nonimmersed.
For \(T=\mu_3^2\), the line fixed by \(T\) has the same nonimmersed image under the invariants of the further product \(G_4\times\mu_3\).
At \(G_{32}\) the effective local group \(G(6,1,2)\) on the plane has degrees \(6,12\); ambient degrees \(12,18,24,30\) omit its degree \(6\) cotangent direction.
Thus this point is excluded from a smooth fixed plane, and a cyclic factor contained in a larger irreducible factor meets a point with an excluded upper group.
This proves the assertions on products of groups used in Lemma~\ref{ew-compact-halfbound}.

\subsubsection{Monomial groups}
\label{ew-compactapp-initial}
On a fixed curve with full generic stabilizer \(G(m,p,3)\), put \(k=mt\) and \(D=3m+3m/p\).
The degree \(m\) direction is unique except for \(p=3\).
For \(m\ge6,p=3\), coordinate mirrors have squared normalized inner products in \(\{0,1/2\}\), but diagonal mirrors have a pair with value \(\cos^2(\pi/m)\notin\{0,1/2\}\).
Thus the full normalizer preserves the coordinate mirrors and their elementary symmetric invariant of degree \(m\).
The character at every larger monomial group and product is trivial, so \(k\in\mathbb Z_{>0}\).
The equations of \cite[Proposition~3.4]{NakanoWeightedConfigurations2026} are \(2x+(1-m^{-1})R+\sigma=2t\) and \(a=-2x-Dt+(2+p^{-1})R\), where \(\sigma=0\) or \(\sigma\ge1/2\).
For \(m\ge6\), and for \(3\le m\le5,p=1\), the inequality \(D>2(2+p^{-1})/(1-m^{-1})\) gives \(g=0\).
The finite test uses Table~\ref{ew-table-imprimitive-tests}; \(h\) is the degree of the selected target line bundle, associated with the invariant of greatest degree for \(p=1\) and with the product invariant for \(G(5,5,3)\).
\begin{table}[htbp]
	\centering
	\begin{tabular}{p{0.21\linewidth}|p{0.16\linewidth}|p{0.13\linewidth}|p{0.38\linewidth}}
		Group                        & Larger group & \(h\)       & Finite inequalities or surviving row                                                          \\\hline
		\(G(m,1,3)\), \(m\ge6\)      & \(G(m,1,4)\) & \(R-3k\)    & \(2k\le R\le2(k+m)/(m-1)\); \(k(m-2)\le m\).                                                  \\
		\(G(m,p,3)\), \(m\ge6,p>1\)  & \(G(m,p,4)\) &             & \(R>(3k-2)/2\); \(k<(6m-2)/(3m-7)\).                                                          \\
		\(G(5,5,3)\)                 & \(G(5,5,4)\) & \(-3t+R/5\) & \((k,R,a,t)=(1,3,5,1/5)\).                                                                    \\
		\(G(m,1,3)\), \(3\le m\le5\) & \(G(m,1,4)\) & \(R-3k\)    & \(1\le k\le\lfloor m/(m-2)\rfloor\); \(2k\le R\le\lfloor2(k+m)/(m-1)\rfloor\); \(a=2-6k+3R\).
	\end{tabular}
	\caption{Finite tests for the imprimitive groups of rank \(3\).}
	\label{ew-table-imprimitive-tests}
\end{table}
In the first row \((k,R,a,h)=(1,2,2,-1)\) contradicts \(a\ge3\).
In the second row \(k\le2\), except \(m=6,k=3\) where \(a\le0\); substituting \(k=1,2\) leaves no positive integral \(a\) with \(\sigma=0\) or \(\sigma\ge1/2\).
For the other rows, \cite[Lemma~3.5]{NakanoWeightedConfigurations2026} gives Hom dimension \(1\) to the target of degree \(h\).
For \(p=1\), we apply it to coefficient degrees \(R-2k,R-k,R\) and the zeros at all \(R\) upper points, and the Hom space is attained by the projection of the inverse differential to the coordinate of greatest invariant degree.
At \(G(5,5,4)\), the other coefficients, of degrees \(1,2<R\), vanish at the \(R\) upper points, and the Hom space is attained by the projection of the inverse differential to the coordinate of the product invariant.
The tangent sequence in the proof of Lemma~\ref{ew-git-seven-classes} excludes degree \(0\) for odd \(a\) and degree \(1\) for even \(a\): negative relative tangent degrees contribute at least \(2,4\), respectively, and otherwise, for even \(a\), evaluation along the constant sections is a constant matrix, so the Hom space is its kernel tensored with \(H^0(\mathcal O_{\mathbb P^1}(1))\), of even dimension.
Together with \(a\ge-3h\) for \(h<0\), this excludes \(G(5,5,4)\) and leaves only \((m,k,R,a,t,\sigma)=(3,2,5,5,2/3,0)\), of relative tangent degrees \((-1,-1,3,1,1,1,1)\).

\begin{lemma}
	\label{ew-compact-git-volume}
	For an ample GIT class \(D\) of multidegree \((b_i)\) and half sum \(q\), one has
	\[
		D^4=3q^4-2\sum_i(q-b_i)^4+\sum_{i<j}(q-b_i-b_j)^4.
	\]
	If its degree on a curve of relative tangent degrees \((-1,-1,3,1,1,1,1)\) is \(2/3\), then \(D^4\le28/81\).
\end{lemma}
\begin{proof}
	Extract the coefficient of \(k^4\) in \(\sum_{|I|\le3}(-1)^{|I|}\binom{k(q-\sum_{i\in I}b_i)+5-|I|}{5}\).
	Concavity of \((D^4)^{1/4}\) permits averaging to \((x,x,z,y,y,y,y)\), where \(r_0=x-y\), \(y=2/3+r_0-3z/2\), \(x=2/3+2r_0-3z/2\), \(0<z<1/3\), \(|r_0|<z/2\).
	For \(u=3z\), \(v=r_0/z+1/2\), the coefficients of \(648D^4\) in \(B_i(u)B_j(v)\), \(B_j(T)=\binom4jT^j(1-T)^{4-j}\), have maximum \(224\) at \((i,j)=(2,4)\).
	This nonnegative basis sums to \(1\), giving \(D^4\le28/81\).
\end{proof}

\begin{lemma}
	\label{ew-compactapp-volumeinitial}
	For any cocompact arithmetic lattice defined over \(K/F\), put \(h_{K,5}=|\operatorname{Cl}(K)[5]|\).
	Then
	\[
		\chi_{\mathrm{orb}}\ge\frac5{h_{K,5}}D_F^{12}
		(D_K/D_F^2)^7\left(\frac{288}{(2\pi)^{14}}\right)^{[F:\mathbb Q]}.
	\]
\end{lemma}
\begin{proof}
	For split nonhyperspecial places \(S_{\mathrm{split}}\), the maximal normalizer index is at most \(5^{1+|S_{\mathrm{split}}|}h_{K,5}\) and each local factor exceeds \(5\).
	Cancel these factors in Prasad's formula, retaining factors \(5\) from the compact dual and central covering \cite[Sections 2.3 and 2.5--2.9, equations (3), (5), (6), (8)]{PrasadYeungSpaces}.
\end{proof}

The integral ample class \(c_{\rm sc}L\), for the scalar order \(c_{\rm sc}\) in \cite[Proposition~2.11]{NakanoWeightedConfigurations2026}, has positive integral degree on every coincidence line.
Complementary triples therefore make the multidegrees of \(c_{\rm sc}L-H\) and its degrees on all \(Z_I\) nonnegative; \cite[Proposition~2.6]{NakanoWeightedConfigurations2026} yields \(a\le c_{\rm sc}t\).
For the remaining \(G(3,1,4)\) row, the fixed curve has signature \((0;3,3,3,3,3)\) and area \(4/3\).
The generators of order \(3\) lie in an Eichler norm one group of genus zero because the normalizer quotient is an elementary abelian \(2\)-group \cite[Theorem 2.1 and Section 3]{MaclachlanRosenberger}.
Voight's exhaustive list of \(1209\) groups of genus at most \(2\) gives \([F:\mathbb Q]\le7\) \cite[Theorem in the Introduction and Section 4]{VoightSmallGenus}.
Here \(c_{\rm sc}\ge8\), \(\zeta_3,\zeta_{c_{\rm sc}}\in K\), and \(m_0=\operatorname{lcm}(3,c_{\rm sc})\) lies in \(\{3,6,9,12,15,18,21,24,30,36,42\}\).
For each row corresponding to a compact orbifold curve of genus zero and area \(\mu\), require integral \(n=(4/3)/\mu\), an embedding \(\mathbb Q(\zeta_{m_0}+\zeta_{m_0}^{-1})\hookrightarrow F\), and cycles \(n=ex_e+(e/3)y_e\) at each cone point, with \(x_e,y_e\ge0\), \(y_e=0\) if \(3\nmid e\), and \(\sum y_e=5\).
Testing embeddings by defining polynomials gives the following rows.
The rows \((D_F,\text{signature},n)\) are \((12,(0;3,3,6),8)\), \((81,(0;2,3,9),24)\), \((81,(0;3,3,9),6)\) and \((1125,(0;3,3,15),5)\).
The defining polynomials are \(T^2-3\), \(T^3-3T+1\), \(T^4-T^3-4T^2+4T+1\), giving \(K=\mathbb Q(\zeta_{12}),\mathbb Q(\zeta_9),\mathbb Q(\zeta_{15})\).
In the first field \(c_{\rm sc}=12\); traces contain every reflection root of unity, but an order \(4\) or \(12\) reflection would remain in a factor of rank \(2\) with equal degrees \(q\mid12\), where no such reflection exists.
All orders therefore divide \(6\), and the generated determinant lies in \(\mu_6\), giving \(c_{\rm sc}\mid30\), a contradiction.
The other two fields have class number \(1\); Lemma~\ref{ew-compactapp-volumeinitial} gives lower bounds greater than \(6\) and \(270\), contrary to \(\chi_{\mathrm{orb}}=5L^4\le140/81\).

\subsubsection{The group \texorpdfstring{\(G(3,3,4)\)}{G(3,3,4)}}
\label{ew-compactapp-g334}
Once \(G(3,1,4)\) is excluded, the fixed curve of the disk fixed by a parabolic subgroup \(G(3,3,3)\) of the full stabilizer \(G(3,3,4)\) of a point has only \(G(3,3,4)\) and product upper groups.
The normalizer preserves the direction of degree \(6\) but may mix the directions of degree \(3\).
\cite[Proposition~3.4]{NakanoWeightedConfigurations2026} gives \(k=6t\in\mathbb Z_{>0}\), \(g=0\), \(k\le10\), \(a=2-2k+7R/3\), \(\sigma=2+k/3-2R/3\), leaving exactly
\[
	(k,R,a,\sigma,h)=(2,3,5,2/3,1),(3,3,3,1,0),(4,3,1,4/3,-1),(6,6,4,0,0),\quad h=R-k.
\]
The projection of the inverse differential to the coordinate of greatest invariant degree is everywhere surjective onto \(\mathcal O(h)\).
The invariants \(g_1=v+S_1\), \(g_2=vS_1+S_2\), \(g_3=vS_2+P^3\), \(g_4=uP\), with \(v=u^3\), \(S_1=\sum y_j^3\), \(S_2=\sum_{i<j}y_i^3y_j^3\), \(P=y_1y_2y_3\), give \(dS_1=dg_2/v\), \(dP=dg_4/u\) on the kernel of this projection.
Hence, once finite monodromy is killed, every line subbundle of the kernel has degree at most \(R-k/2\); no splitting on \(C\) is assumed.
For \(k=3\), odd relative tangent degrees of sum \(3\) force labels of negative relative tangent degree: three such labels give degree \(7>3/2\) in the kernel, two prevent surjection to \(\mathcal O\), and one requires \(a\ge5e\).

Apply the preceding conditions on \(n,x_e,y_e\) to the compact genus zero rows in Voight's table, with area \(2/3\) and four order \(3\) points, or area \(2\) and six order \(3\) points; the area of an orbifold curve of genus \(g\) with cone orders \(e_i\) is \(-\chi_{\mathrm{orb}}=2g-2+\sum_i(1-1/e_i)\).
The second test gives the fields below; each pair \(([F:\mathbb Q],D_F)\) identifies a unique defining polynomial in that table.
\[
	\begin{array}{c|l}
		[F:\mathbb Q]&D_F\\\hline
		1&1\\
		2&5,8,12,13,17,24,28,29,33,60\\
		3&49,81,148,169,229,257,316,321,361,469\\
		4&725,1957,2000,2304,2777,3981,4205,4352\\
		5&24217,36497,38569
	\end{array}.
\]
The first test deletes \(24,29,33,60,321,361,469,2304,3981,4205\), leaving \(22\) fields.
For either test, \(n\) is the required area divided by the area of the tabulated row, and each cone order \(e\) permits precisely the integers \(0\le y_e\le4\) or \(0\le y_e\le6\) satisfying \(n-(e/3)y_e\in e\mathbb Z_{\ge0}\) and \(y_e=0\) if \(3\nmid e\); the sum of the \(y_e\) is \(4\) or \(6\).
These bounded conditions specify every row used in the test.
For \(k=2\), use the first list.
The conditions \(c_{\rm sc}\ge15\), \(K=F(\zeta_3)\), \(F\ne\mathbb Q\), \(|\mu(K)|\ge15\) leave only \(K=\mathbb Q(\zeta_9)\), \(D_F=81\).
Three labels of negative relative tangent degree violate the kernel bound \(13>2\); otherwise Lemma~\ref{ew-compact-git-volume} for \(2L\), or averaging with \(y=1/9+z/6\), \(0<z<1/3\), gives \(L^4\le7/324\) or \(31/1944\).
Every ideal class of \(\mathbb Q(\zeta_9)\) has a representative of norm at most \(5\) by Minkowski's theorem.
The primes \(2,5\) have residue degree \(6\), and the unique prime above \(3\) is generated by \(1-\zeta_9\); hence the class number is \(1\).
Both contradict the arithmetic bound greater than \(6\).
For \(k=6\) the kernel bound excludes negative degrees by \(10>3\); immersion then gives \(\mathcal N=\mathcal O^{\oplus2}\oplus\mathcal O(2)\), \(t=1\).
Average to \((z,z,z,z,z,1/2,1/2)\), \(1/3<z<1\), giving \(L^4\le27/8\).
The same conditions with area \(2\) and six order \(3\) points give \(32\) polynomials.
For \(K=F(\zeta_3)\) put \(d=[F:\mathbb Q]\), \(\Delta=D_K/D_F^2\), \(\varpi=288/(44/7)^{14}\), and
\[
	M_K=\left\lfloor\left(\frac43\right)^d\frac{(2d)!}{(2d)^{2d}}(\lfloor\sqrt{D_K}\rfloor+1)\right\rfloor,
	\qquad U_K=\sum_{1\le m\le M_K}\prod_{p^a\parallel m}\binom{2d+a-1}{a}.
\]
Minkowski's theorem and the coefficients of \(\zeta_K\) give \(h_{K,5}\le h_K\le U_K\).
Every ideal class has a representative of norm at most \(M_K\), and the number of ideals of norm \(p^a\) is at most \(\binom{2d+a-1}{a}\).
For the listed fields, \(\Delta=3^d\) except that \(\Delta=1\) at \(D_F=12,24,33,60,2304\), \(\Delta=3\) at \(81,321\), and \(\Delta=9\) at \(3981\).
The lower bound \(5D_F^{12}\Delta^7\varpi^d/U_K\) exceeds \(135/8\) in the following cases; the last column bounds its integer part from below for every indicated field.
\[
	\begin{array}{c|l|r}
		d&D_F&\text{lower bound}\\\hline
		2&13,17,28,29,60&52\\
		3&49,148,169,229,257,316,321,361,469&443\\
		4&725,1957,2000,2304,2777,3981,4205,4352&1343\\
		5&24217,36497,38569&10^{19}
	\end{array}.
\]
For \(D_F=33\), one has \(K=\mathbb Q(\sqrt{-3},\sqrt{-11})\) and \(M_K=5\).
The prime \(2\) has two primes of residue degree \(2\), the prime \(3\) has two primes of ramification index \(2\) and residue degree \(1\), and the prime \(5\) has residue degree \(2\).
Thus the prime ideals of norm at most \(5\) have norm \(3\) or \(4\).
Each prime of norm \(4\) is principal, generated by a conjugate of \((\sqrt{-3}+\sqrt{-11})/2\), whose norm is \(4\).
The square of each prime of norm \(3\) is principal, generated by a conjugate of \((1+\sqrt{-11})/2\).
Thus \(h_{K,5}=1\), and the volume lower bound exceeds \(30>135/8\).
Consequently \(F\ne\mathbb Q\), \(c_{\rm sc}\ge4\) and the volume bound leave the following fields, where \(\rho_{K/F}\) is the rational number of Lemma~\ref{ew-compact-volume}:
\[
	\begin{array}{c|c|c|c}
		D_F&D_K&|\mu(K)|&\rho_{K/F}\\\hline
		5&225&6&31/6075\\8&576&6&201641/124416\\12&144&12&23/497664\\24&576&6&551/2560\\81&19683&18&18628987/12597120
	\end{array}.
\]
The numerator divides \((c_{\rm sc}L)^4\) and is coprime to \(5c_{\rm sc}\) because the maximal normalizer index is a power of \(5\).
For \(D_F=81\) it exceeds \(18^4(27/8)=354294\).
Otherwise \(c_{\rm sc}=6\), also using \(c_{\rm sc}\mid30\) for \(D_F=12\); positive degrees of \(c_{\rm sc}L\) on the planes \(Z_I\) and moving weight sum \(6\) leave \((2,3,3,3,3;3,3)\), \((3,3,3,3,4;3,3)\), \((4,4,4,4,4;2,4)\), \((4,4,4,4,4;3,3)\), of volumes \(586\), \(880\), \(1450\), and \(1684\).
None has the required divisibility, with special values from Section~\ref{ew-compactapp-770}.

For \(k=4\) the negative target gives a line \(C\subset Z\), \(\mathcal N=\mathcal O(1)\oplus\mathcal O(-1)^{\oplus2}\), and kernel of the projection to the coordinate of greatest invariant degree \(\mathcal O(1)\oplus\mathcal O(-1)\).
Let \(\rho\) be the map from the orbifold curve to the coarse curve \(C\), let \(\mathcal E\) be the orbifold bundle in polynomial degree \(3\), and put \(\mathcal U=\mathcal E\otimes\rho^*\mathcal O_C(2)\).
Then \(\mathcal U\) is unitary flat, and the direct image of its invariant sections is \(\rho_*\mathcal U\simeq\mathcal O_C\oplus\mathcal O_C(-2)\); Bochner makes the trivial summand parallel.
The characters \((1,\zeta_3^2)\), the trivial product action, and the parallel normalizer-invariant lines \(P(S_1^3-27P^3)=0\) select a plane fixed by \(\mathfrak S_3\).
Its compact quotient \(Y\) has maps \((U,V)\mapsto(U,V,0,0)\), \(U=u^3+z^3\), \(V=u^3z^3\), at upper points, and \((u,z)\mapsto(u^b,z^3)\) at stabilizers \(T\times\mu_b\).
Hence \(N_{C/Y}=\mathcal O(-2)(D_R)=\mathcal O(1)\); the vanishing maps \(\mathcal O_C(1)^{\oplus2}\to\mathcal O_C(-j)\) and Krull intersection give \(Y=Z\), \(\lambda_Z=2/3\).
The product contribution \(4/3\) and three index \(2\) contacts with \(C\) leave \(\Delta_Z=2C/3+D/2+2E/3\), \((\deg D,\deg E)=(6,2)\), or \(\Delta_Z=2C/3+D/2+5A/6\), \((\deg D,\deg A)=(7,1)\).
The general groups are \(G(3,3,3)\) on \(C\) and \(\mathfrak S_3\times\mu_b\) elsewhere.
\cite[Theorem 3.12 and Tables 5--9]{KrishnasamyTaylor}, applied to the full normalizers, excludes \(A_4,B_4\) by effective degrees \((1,2),(2,2)\) and \(D_4,G(4,2,4),G(4,4,4)\) by normalizer/reflection orders \((2,1),(8,4),(4,1)\); the same defect excludes \(G_{29},G_{31}\).
For \(H_4\), degree \(6\) of \(I_2(6)\) is absent from \(2,12,20,30\); \(G_{32}\) lacks order \(2\) reflections.
The \(\mathfrak S_3\) types in \(G(3,3,4)\) have degrees \((1,3)\), excluded, and \((3,6)\), giving the three points on \(C\); in \(F_4\), containment in \(B_3\) contradicts general group \(\mathfrak S_3\times\mu_2\).
Any rank \(3\) factor absorbing \(\mathfrak S_3\) must act on its fixed line, since otherwise immersion requires degree \(1\); it therefore occurs only on \(C\) as \(G(3,3,3)\).
For \(x\in Z\setminus C\), write \(\mathfrak S_3\times G'_x\) for the full stabilizer of \(x\) in the original group; its irreducible factors with reflections of two orders are \(G(6,2,2),G_7\) in the first divisor and \(G(12,2,2)\) in the second, of effective normalized orders \(6,12,12\).
On \(A\), on an irreducible \(E\), and on a line component of \(E\), let \(n_1\) be the number of ordinary contacts with \(D\) off \(C\), and let \(n_2\) be the number of points of the intersection with \(D\) off \(C\) at which \(G'_x\) is \(G(12,2,2)\) on \(A\) and \(G(6,2,2)\) on \(E\), both excluding the crossing of two line components of \(E\); B\'ezout counts them with weights \(1,2\), and their orbifold contributions are \(1/2\) for \(n_1\) and \(11/12\) on \(A\), \(5/6\) on \(E\) for \(n_2\).
For \(A\), B\'ezout and proportionality give \(n_1+2n_2=7\), \(-2+2/3+n_1/2+11n_2/12=4/3\), hence \((n_2,n_1)=(10,-13)\).
For irreducible \(E\) they give \(n_1+2n_2=12\), \(-2+4/3+n_1/2+5n_2/6=8/3\), hence \((n_2,n_1)=(16,-20)\).
If \(E\) splits, its crossing is off \(C\); a line gives \(n_1+2n_2=6\), \(-2+4/3+n_1/2+5n_2/6=4/3\), hence \(n_1=-6\), off \(D\).
On \(D\), \(G'_x=G_7\) gives \(n_1+2n_2=5\), \(-2+2/3+11/12+n_1/2+5n_2/6=4/3\), hence \(n_1=-4\).
These negative values of \(n_1\) exclude \(k=4\).

\begin{lemma}
	\label{ew-compact-groups}
	An irreducible factor of rank \(4\) occurs in some point stabilizer.
	The possible irreducible groups of rank \(4\) are \(A_4,B_4,D_4,G(4,2,4),G(4,4,4),F_4,H_4,G_{29},G_{31},G_{32}\).
	Every smaller irreducible factor enters one of these groups as a parabolic subgroup, so every reflection has order \(2\) or \(3\).
\end{lemma}
\begin{proof}
	Suppose first that no point stabilizer has an irreducible factor of rank \(4\).
	Follow an irreducible factor through successive absorptions, that is, points at which it becomes a parabolic subgroup of a larger irreducible factor; enlargements of the same rank already lie in its full stabilizer.
	If the rank stops below \(4\), \cite[Lemma~2.12; Proposition~2.14]{NakanoWeightedConfigurations2026} exclude ranks \(3,1\) and, in rank \(2\), give the fixed plane \(Y=\mathbb P^2\) of this factor, with \(N_{Y/X}=\mathcal O(-1)^{\oplus2}\).
	Equal degrees \(q\) occur only for \(G(m,2,2)\), even \(m\ge4\), and \(G_7,G_{11},G_{19}\), with \(q=m,12,24,60\).
	At a point \(x\) of \(Y\), this factor is a direct factor of the full stabilizer of \(x\), since an absorption would give an irreducible factor of rank at least \(3\).
	The effective local groups on the plane are therefore products of cyclic groups or the irreducible groups of rank \(2\) with equal basic degrees listed above.
	Their branch germs are nodes or ordinary triple points of types \([2,2,m/2]\), \([2,3,3]\), \([2,3,4]\), and \([2,3,5]\), by \cite[Table~1]{Tan}.
	Several local branches may belong to the same irreducible curve.
	Write \(\Delta_Y=\sum_i\beta_iD_i\), where \(\beta_i=1-b_i^{-1}\), and let \(d_i\), \(g_i\), and \(\delta_i\) be the degree, the genus of the normalization, and the sum of the delta invariants of \(D_i\).
	In \cite[equations~(5), (7), and~(8)]{Tan}, identifying \(r\) local branches of index \(b_i\) changes \(\theta_x\) by \(-3\beta_i\binom r2\).
	Here \(\Gamma_x\) is the effective local group on the plane as above, and \(\rho_1,\rho_2\) count the mirrors above two of these branches.
	The intersection multiplicity of these two branches is one, so pulling back gives \(b_i^2\rho_1\rho_2=|\Gamma_x|\).
	Use \(\mathcal P(Y,\Delta_Y)=\operatorname{Prop}(Y,\Delta_Y)-\frac12\sum_i\beta_i\operatorname{prop}_{(Y,\Delta_Y)}(D_i)\) with the surface and curve invariants of \cite[Definition~3.8 and Remark~3.10]{Tan}; both vanish for a compact ball quotient orbifold.
	Thus \cite[Propositions~3.7 and~3.9]{Tan} gives
	\[
		0=\mathcal P(Y,\Delta_Y)
		=\frac32\sum_i\beta_i\bigl(d_i(d_i-1)-2\delta_i\bigr)
		=3\sum_i\beta_i(d_i-1+g_i).
	\]
	Every component is consequently a line, so the hypotheses of \cite[Theorem~1.1]{Tan} hold.
	Its two arrangements have automorphic degrees \(1/6\) and \(1/2\).
	Since \(q\lambda=1\) and this factor is irreducible, \(q=6\) and the arrangement is the complete quadrilateral.
	A point of type \([2,3,3]\) supplies a factor \(G_7\) that is never absorbed, whose fixed plane has degree \(1/12\), contradicting the same conclusion.
	This proves the first assertion.
	Sections~\ref{ew-compactapp-initial}--\ref{ew-compactapp-g334} exclude every other rank \(4\) group.
	For a smaller factor never entering rank \(4\), the maximal \(2\)-adic reflection order and Section~\ref{ew-compactapp-section} exclude \(G_7,G_{11},G_{19}\) and leave \(q=2s\) with odd \(s\ge3\).
	An order divisible by \(4\) gives \(3/q=\deg(K_Y+\Delta_Y)\) a denominator of larger \(2\)-adic valuation than its branch coefficients have.
	For index \(b\ge5\), branch degree \(d\) and \(n_{\rm node}\) nodes, proportionality and B\'ezout give
	\[
		6d-4+2b\lambda d\le2d^2+2b\lambda d-4n_{\rm node}\le d(3+3\lambda-(1-b^{-1})d),\quad\lambda=1/(2s),
	\]
	forcing \(d=1\), \(b/s\in\mathbb Z_{>0}\), \(b(2b-3)\le2s\), a contradiction.
	Thus \(3d_2+4d_3=18+9/s\), \(s=3,9\).
	Since the degree of each component of ramification index \(3\) is divisible by \(3\), we exclude \(q=18\) and \((q,d_2,d_3)=(6,7,0)\).
	For \((d_2,d_3)=(3,3)\), a component of degree \(1\le d\le3\), with \(n_{2,3}\) cusps, satisfies \(2d^2+d-9p_a(d)\le2d^2+d-9n_{2,3}-4n_{\rm node}\le3d/2\), where \(p_a(d)=(d-1)(d-2)/2\); the lower bounds \(3,10,12\) exceed \(3/2,3,9/2\).
\end{proof}
\begin{lemma}
	\label{ew-compact-halfbound}
	For every coincidence plane \(Z_I\simeq\mathbb P^2\), \(|I|=3\), write \(L|_{Z_I}=\lambda_I c_1(\mathcal O_{Z_I}(1))\).
	Then \(\lambda_I\ge1/2\) and \(2L-H\) is semiample.
	In particular, \(a\le2t\) for every complete curve.
\end{lemma}
\begin{proof}
	Lemma~\ref{seven-plane-gauss} gives \(\lambda\ge1\) in ranks \(0,1\) and \(q\lambda\ge1\) for the least positive invariant degree \(q\) in rank \(2\), with regular maps also at larger stabilizers.
	For \(\lambda<1/2\), Section~\ref{ew-compactapp-section} leaves the following possibilities, with \(d\) the total degree of the upper locus:
	\[
		\begin{array}{c|c|l}
			\text{generic stabilizer}&\text{upper group}&\text{degree equation}\\\hline
			G(4,2,2)&G(4,2,3)&8\lambda=2+3d/2,\quad d=0,1\\
			G_4&G_{25}&10\lambda=2\quad\text{if no upper locus}\\
			\mu_3^2&G_{25}&6\lambda=2+5d/2,\quad d=0
		\end{array}.
	\]
	In the first row the upper locus has index \(4\), and with the degrees \(d'_2,d'_3\) of the other branch components of indices \(2,3\), \(3\lambda+3=d'_2/2+2d'_3/3+3d/4\); the invariants \(v+S,vS+P^2,u^2P\), \(v=u^4\), \(S=y_1^4+y_2^4\), \(P=(y_1y_2)^2\), give \(d'_2/2+2d'_3/3=(60-3d)/16\), impossible after multiplication by \(6\).
	For \(G_4\), at an upper point the image \((v^2,v^3,0)\) of the line fixed by \(G_4\) under the invariants of \(G_{25}\) is nonimmersed; an empty upper locus contradicts \(4\lambda\ge1\).
	For \(\mu_3^2\), the normal differential \(\left(\begin{smallmatrix}0&-v\\24v^2&0\end{smallmatrix}\right)\) at an upper point gives \(\lambda=1/3\) and \((d_2,d_3)=(8,0),(4,3),(0,6)\).
	An index \(3\) component of degree \(d\le6\) requires \(9n_{2,3}+4n_{\rm node}=2d(d+1)\) and \(n_{2,3}+n_{\rm node}\le(d-1)(d-2)/2\).
	For \(d\le5\) these are incompatible; for \(d=6\), the only pairs \((n_{2,3},n_{\rm node})=(0,21),(4,12),(8,3)\) still violate the bound.
	Hence \(d_3=0\), but the effective normalized orders \(1,2,4\) make each index \(2\) degree divisible by \(3\), contrary to \(d_2=8\).
	For the multidegrees \((w_i)\) of \(2L-H\), the degrees \(d_I=2\lambda_I-1\ge0\) on the planes \(Z_I\) give \(w_i=d_I+d_{I'}\ge0\) for complementary triples \(I,I'\) not containing \(i\); \cite[Proposition~2.6]{NakanoWeightedConfigurations2026} proves semiampleness.
\end{proof}

\subsection{Normal differentials at larger stabilizers}
\label{ew-compactapp-matrices}
The following finite recipe uses invariant coordinates on the quotient germs.
Expand \(\Phi_j(uv_0+y)=\Phi_j(uv_0)+\sum_i c_{ji}(u)g_i(y)\) modulo \((g_1,\ldots,g_s)^2\), for ambient and normal basic invariants \(\Phi_j,g_i\).
Adjoin the tangent column, remove its common factor, and eliminate the saturated tangent direction; the resulting matrix \(\mathsf{M}(v)\) is the normal differential, with columns indexed by normal directions adapted to the finite characters and rows by the normal components of the targets from the \(\Phi_j\).
Its entries are taken in the coarse frames specified below, local generators of the invariant sections: if \(u\mapsto\zeta_eu\) and \(e_i\mapsto\zeta_e^{a_i}e_i\) with \(0\le a_i<e\), the frame over \(v=u^e\) is \(u^{a_i}e_i\), with degree correction \(-a_i/e\)~\cite[Lemma~3.3]{NakanoWeightedConfigurations2026}.
Cofactors then give the rows of \(\mathsf{M}(v)^{-1}\) and their restrictions to successive kernels; each row is a rational homomorphism extracting one normal component, regular after twisting the target by its poles, and its surjectivity is checked at every point.
Every remaining operation is holomorphic and invertible, so the conclusions hold under arbitrary analytic isomorphisms of quotient germs.
Table~\ref{ew-table-normal-inputs} gives all inputs for the larger groups; its last column gives the determinant order, the pole order of the inverse row associated with the invariant of greatest degree, and the pole order of the preceding inverse row on the kernel of that row.
For \(G_{25}\subset G_{32}\) use \(r_jz=z-(Cz)_j\varepsilon_j\), with
\[
	C=\begin{pmatrix}1-\omega&\omega^2&0&0\\-\omega^2&1-\omega&-\omega^2&0\\0&\omega^2&1-\omega&\omega^2\\0&0&-\omega^2&1-\omega\end{pmatrix},\qquad \omega^3=1,\quad\omega\ne1.
\]
The first three reflections fix \(v_0=((2\omega+1)/3,1,(4\omega+2)/3,1)\) and act normally on \(z_4=0\).
Use the \(240\) forms in \(e_4^*G_{32}\) and \((1,2,3)G_{25}\) for the respective power sums; separate the square of degree \(6\) in degree \(12\).
Normalizing the first coefficient row to \((1,1,1,1)\) gives determinant \(-19489164500000/1162261467\).
On the line \(\mathbb Cv_0\), the first invariants are \((56/9)v^2,(1456/243)v^3\), \(v=u^6\), and the frame in degree \(9\) is \(u^3\partial_{f_9}\).
Thus \(\mathsf{M}(v)=\operatorname{diag}(v^2,v^3,v^4)A\operatorname{diag}(1,1,v^{-1})\), \(\det A\ne0\), with a simple tangent zero.
\begin{table}[htbp]
	\centering
	\(\begin{array}{c|c|c|c}
		T\subset G&\text{degrees of }G&\text{degrees of }T&\text{orders}\\\hline
		G_{25}\subset G_{32}&12,18,24,30&6,9,12&(8,3,\text{---})\\
		A_3\subset A_4&2,3,4,5&2,3,4&(3,1,1)\\
		A_3\subset D_4&2,4,4,6&2,3,4&(3,1,1)\\
		B_3\subset B_4&2,4,6,8&2,4,6&(3,1,1)\\
		B_3\subset F_4&2,6,8,12&2,4,6&(7,3,2)\\
		H_3\subset H_4&2,12,20,30&2,6,10&(22,10,7)\\
		G(4,4,3)\subset G_{29}&4,8,12,20&4,3,8&(7,3,3)
	\end{array}\)
	\caption{Degrees and orders for the pairs \(T\subset G\).}
	\label{ew-table-normal-inputs}
\end{table}

For \(A_3\subset A_4\) use \((u,u,u,u,-4u)\) and \(\sum y_i=0\); the power sums have values \((20u^2,-60u^3,260u^4,-1020u^5)\).
For \(A_3\subset D_4\) use \((u,u,u,u)\), \(v=u^2\), and character frame \(u\partial_{q_3}\).
The matrices are
\[
	\mathsf{M}_A=\begin{pmatrix}15u/2&1&0\\-20u^2&4u&1\\275u^3/2&10u^2&5u\end{pmatrix},\quad
	\mathsf{M}_D=\begin{pmatrix}4v&4v&1\\0&4v/3&0\\0&8v^2&12v\end{pmatrix}.
\]
The determinants are \(\det \mathsf{M}_A=625u^3/2\) and \(\det \mathsf{M}_D=64v^3\).
The row associated with the invariant of greatest degree and the next row restricted to its kernel remain nonzero after pole correction at every stalk; the last zero has order \(1\), and the \(A_4\) tangent zero is simple.
All three \(D_4\) types are conjugate by the root preserving transformations \(\frac12\left(\begin{smallmatrix}1&1&1&1\\1&1&-1&-1\\1&-1&1&-1\\1&-1&-1&1\end{smallmatrix}\right)\) and \(\operatorname{diag}(-1,1,1,1)\).

For \(B_3\subset F_4\) take the short roots \(\{\pm\varepsilon_j\}\cup\{\frac12\sum\epsilon_j\varepsilon_j:\epsilon_j=\pm1\}\) and their power sums.
On \(u\varepsilon_1\), \(v=u^2\), elimination gives rows \((-3v^2,15v,18)\), \((0,-15v^2,45v)\), \((0,0,-405v^3/8)\).
For \(u(\varepsilon_1+\varepsilon_2)\) multiply the rows by \(-8,16,-64\); the values \(6v,12v\) of invariants \(\Phi_j\) on the lines \(\mathbb C\varepsilon_1,\mathbb C(\varepsilon_1+\varepsilon_2)\), respectively, prove immersion, and the last zero order is \(2\).
For \(B_3\subset B_4\) use \(\left(\begin{smallmatrix}v&1&0\\0&v&1\\0&0&v\end{smallmatrix}\right)\), of last zero order \(1\).
For \(H_3\subset H_4\) use the Gram matrix of diagonal \(2\) and successive off-diagonal entries \(-\phi,-1,-1\), \(\phi=(1+\sqrt5)/2\), and the fixed vector \((1+\sqrt5/2,(3+\sqrt5)/2,(5+\sqrt5)/4,1)\) of \(H_3\).
The \(120\) ambient and \(30\) normal root sums give coefficients of exponent \((d-e)/2\) in \(v=u^2\).
The nonzero pivots are \(-(21\sqrt5+63)v^5/512\), \((3740\sqrt5+8415)v^7/1024\), \((22297275\sqrt5+49864815)v^{10}/1048576\).

For \(G(4,4,3)\subset G_{29}\) use \(r_jz=z-(Bz)_j\varepsilon_j\), with
\[
	B=\begin{pmatrix}2&-1&-i&0\\-1&2&i-1&0\\i&-i-1&2&-1\\0&0&-1&2\end{pmatrix}.
\]
The first three reflections fix \(((1+i)/2,1-i/2,3/2,1)\); the \(160\) forms in \(z_1G_{29}\) give tangent and normal determinant \(128957475465i/524288\).
For \(v=u^4\), elimination gives rows \((-81v/8,-9iv^2/4,-81/4)\), \((0,-1183iv^3/384,-3549v/16)\), \((0,0,-4037365v^3/4096)\).
The character in degree \(3\) is \(-i\), with frame \(u^3\partial_P\), and the last zero has order \(1\).
At \(G(4,4,4)\) use \(\left(\begin{smallmatrix}v&0&1\\0&v&0\\0&0&v\end{smallmatrix}\right)\).
The unique parabolic class and nonzero value, proportional to \(v\), of the first invariant on the line fixed by \(G(4,4,3)\) give a single smooth branch.
Distinct basic degrees make the full normalizer preserve the flag.

\subsection{The groups \texorpdfstring{\(B_4\)}{B4} and \texorpdfstring{\(G_{32}\)}{G32}}
\begin{lemma}
	\label{ew-compact-bfour}
	Let \(B_3\) be a parabolic subgroup of the full stabilizer \(B_4\) of a point, and let \(C\) be the fixed curve of the disk fixed by this \(B_3\), which passes through the point.
	Then \(C\) contains a point with full stabilizer \(F_4\) or \(G_{29}\).
	Similarly, for \(p=2,4\), let \(C\) be the fixed curve of the disk fixed by a parabolic subgroup \(G(4,p,3)\) of the full stabilizer \(G(4,p,4)\) of a point; then \(C\) contains a point with full stabilizer \(G_{31}\) if \(p=2\) and \(G_{29}\) if \(p=4\).
\end{lemma}
\begin{proof}
	The full generic stabilizer of \(C\) is this \(B_3\); if \(C\) contains no point with full stabilizer \(F_4\) or \(G_{29}\), its only upper groups are \(B_4\) and products.
	\cite[Proposition~3.4]{NakanoWeightedConfigurations2026} gives \(2x+R/2+\sigma=2t\) and \(a=-2x-12t+3R\).
	These equations imply \(a\le-14x\), so \(C=\mathbb P^1\).
	Put \(k=2t\in\mathbb Z_{>0}\) and \(j=R-2k\).
	Then \(a=2+3j\), \(\sigma=2-j/2\), and the quotient associated with the invariant of greatest degree has degree \(h=j-k\).
	Lemma~\ref{ew-compact-halfbound} gives \(k\ge2+3j\), hence \(j\ge0\) and \(h<0\).
	The monomial calculation in Section~\ref{ew-compactapp-initial} gives Hom dimension \(1\), attained by the surjective projection to the coordinate of greatest invariant degree.
	\cite[Proposition~2.2]{NakanoWeightedConfigurations2026} therefore gives two coincident labels, a common relative tangent line of degree \(-(k-j)\), and \(a\ge3(k-j)\).
	This contradicts \(2+3j\le k\le2j+2/3\).

	In the second assertion, assume that \(C\) contains no point with full stabilizer \(G_{31}\) or \(G_{29}\), respectively.
	The upper groups are then \(G(4,p,4)\) and products, and \(k=4t\) is a positive integer.
	\cite[Proposition~3.4]{NakanoWeightedConfigurations2026} gives \(g=0\), \(k\le3\), and the three rows
	\[
		(p,k,R,a,\sigma)=(2,1,3,5,1/4),(2,2,4,3,0),(4,1,3,5,1/4).
	\]
	The first and third have \(0<\sigma<1/2\); the second contradicts \(a\le2t\).
\end{proof}

\begin{lemma}
	\label{ew-compact-gthirtytwo}
	The group \(G_{32}\) does not occur, and every reflection has order \(2\).
\end{lemma}
\begin{proof}
	Take a parabolic subgroup \(G_{25}\) of the full stabilizer \(G_{32}\) of a point and the fixed curve \(C\) of the disk fixed by it.
	Let \(R\) count the points of \(C\) with full stabilizer \(G_{32}\) and \(s_b\) count those with full stabilizer \(G_{25}\times\mu_b\), for \(b=2,3\).
	The calculation in Section~\ref{ew-compactapp-matrices} gives effective inertia of order \(6\), characters \((1,-1,1)\) on degrees \(6,9,12\), a simple tangent zero, normal determinant order \(8\), and pole order \(3\) for the inverse row associated with the invariant of greatest degree.
	Consequently, with \(k=6t\in\mathbb Z_{>0}\),
	\[
		2t=2x+5R/6+s_2/2+2s_3/3,\qquad
		a=-2x-27t+17R/2,\qquad h=-12t+3R.
	\]
	The first equation gives \(t=x+5R/12+s_2/4+s_3/3\).
	Substitution in the equation for \(a\) gives \(a=-29x-11R/4-27s_2/4-9s_3\), hence \(g=0\).
	The conditions \(a\in\mathbb Z_{>0}\) and \(a\le k/3\) leave exactly
	\[
		(R,s_2,s_3,t,a,h)=(2,2,1,2/3,1,-2),(7,1,0,13/6,3,-5).
	\]
	In both rows, the nonzero inverse map has target degree less than \(-a\), contrary to \cite[Proposition~2.2]{NakanoWeightedConfigurations2026}.
	Lemma~\ref{ew-compact-groups} now excludes every reflection of order \(3\).
\end{proof}

\subsection{The plane of degree \texorpdfstring{\(1/2\)}{1/2}}
\label{ew-compactapp-halfplane}
A compact plane quotient of degree \(1/2\) has branch indices \(2\) and a reduced branch curve \(B\) of degree \(9\).
For nodes, cusps \((2,3)\), contacts \(2\), cusps \((2,5)\), and triple points, let the counts be \(n_{\rm node},n_{2,3},n_{\rm tan},n_{2,5},n_{\rm tr}\).
Their local groups \(\mu_2^2,A_2,B_2,I_2(5),G(4,2,2)\) have orders \(4,6,8,10,16\), \(\delta\) invariants \(1,1,2,2,3\), Milnor numbers \(1,2,3,4,4\), and effective normalized orders \((2,2),(1),(2,2),(1),(4,4,4)\).
Let \(r_B\) be the number of irreducible components of \(B\) and \(g_i\) the genera of their normalizations; then \(\sum g_i=27+r_B-\delta_B\), \(\delta_B=n_{\rm node}+n_{2,3}+2n_{\rm tan}+2n_{2,5}+3n_{\rm tr}\le36\).
Proportionality and \(\chi_{\mathrm{orb}}=3/4\) give
\[
	\begin{gathered}4n_{\rm node}+8n_{2,3}+12n_{\rm tan}+16n_{2,5}+15n_{\rm tr}=180,\\180n_{\rm node}+320n_{2,3}+450n_{\rm tan}+576n_{2,5}+585n_{\rm tr}=7020.
	\end{gathered}
\]
Thus \(n_{2,3}=9k_{2,3}\), \(n_{2,5}=5k_{2,5}\), \(n_{\rm tr}=4k_{\rm tr}\), and
\[
	n_{\rm tan}=12-4k_{2,3}-8k_{2,5}-4k_{\rm tr},\quad n_{\rm node}=9-6k_{2,3}+4k_{2,5}-3k_{\rm tr},\quad\delta_B=33-5k_{2,3}-2k_{2,5}+k_{\rm tr}.
\]
The conditions \(k_{2,3}+2k_{2,5}+k_{\rm tr}\le3\), \(n_{\rm node}\ge0\), \(\delta_B\le36\) leave exactly the nine rows in Table~\ref{ew-table-halfplane}.
If the \(i\)th component is a line or smooth conic of degree \(d\), then its numbers \(n_i,t_i,r_i\) of nodes, contacts, and triple points satisfy \(n_i+2t_i+2r_i=d(9-d)\), \(n_i/2+t_i/2+3r_i/4=d+2\).
Consequently lines have \((n_i,t_i,r_i)=(0,0,4),(2,1,2),(4,2,0)\) and conics need \(t_i\ge5\).
For the union \(B^*\) of duals of nonlinear components and \(l\) lines, \(D=\deg B^*=72-2n_{\rm node}-3n_{2,3}-4n_{\rm tan}-5n_{2,5}-6n_{\rm tr}\) and \(\delta(B^*)=(D-1)(D-2)/2+\delta_B-28-l\).
A contact not involving a line, or a cusp \((2,5)\), contributes at least \(2\) to \(\delta(B^*)\), additively even at a common dual point.
For a cusp write \(x=u^2\), \(y=a_4u^4+a_5u^5+\cdots\), \(a_5\ne0\); its dual has type \((2,5)\) after subtracting the quadratic term if \(a_4\ne0\), and \((3,5)\) otherwise.
\begin{table}[htbp]
	\centering
	\begin{tabular}{c|c|p{0.69\linewidth}}
		\text{row} & \((k_{2,3},k_{2,5},k_{\rm tr})\) & \text{exclusion or equality}                                                                                               \\\hline
		1          & \((0,0,3)\)                      & \(\delta_B=36\): nine lines and twelve ordinary triple points.                                                             \\
		2          & \((1,0,1)\)                      & No nodes, hence no line contact; dual upper/lower bounds \(7,8\).                                                          \\
		3          & \((0,0,2)\)                      & \(r_B\ge8\): one conic and seven lines; five required contacts exceed four.                                                \\
		4          & \((1,0,0)\)                      & Too few nodes for a line; dual bounds \(15,16\).                                                                           \\
		5          & \((0,0,1)\)                      & \(r_B=7\): a cubic supplies fewer than eight contacts; the partition into two conics has impossible tangent distributions. \\
		6          & \((1,1,0)\)                      & No contacts; dual bounds \(8,10\).                                                                                         \\
		7          & \((0,0,0)\)                      & \(r_B=7\): at most six contacts; \(r_B=6\), partitions \((4,1^5),(3,2,1^4)\): at most eleven.                              \\
		8          & \((0,1,1)\)                      & No contacts; dual bounds \(5,10\).                                                                                         \\
		9          & \((0,1,0)\)                      & \(r_B\ge4\): five cusps \((2,5)\) force \((6,1,1,1)\); the three lines need six contacts rather than four.
	\end{tabular}
	\caption{The nine rows for a plane of degree \(1/2\) in the compact case.}
	\label{ew-table-halfplane}
\end{table}
For row \(5\), the conics need two mutual contacts, six line contacts, and \((t_i,n_i,r_i)=(5,0,2)\).
Their four triple points are distinct; distributions of lines touching two, one, or no conic are \((1,4,0),(2,2,1),(3,0,2)\).
A line touching one conic joins the other's two triple points, so there are at most two; a line touching neither joins both pairs, excluding the last two distributions as well.
The remaining partition \((2,2,2,1,1,1)\) in row \(7\) requires two contacts per conic pair.
The intersections of the dual conics of each pair are exhausted by the common tangents; a line touching two conics would give three tangent branches, an excluded local type.
Only row \(1\) remains.

\begin{lemma}
	\label{ew-compact-halfplanes}
	If \(\lambda_I=1/2\), the full generic stabilizer of \(Z_I\) is \(\mu_2^2\), and every stabilizer on the fixed plane has the form \(\mu_2^2\times G'\).
	The induced branch divisor consists of nine lines with twelve ordinary triple points and no other singularities.
	Two disjoint triples cannot both have degree \(1/2\).
	The integral semiample bundle \(5(2L-H)\) has strictly positive multidegrees \(v_i\).
	A curve fixed by an irreducible group of rank \(3\) avoids every such plane.
\end{lemma}
\begin{proof}
	Lemma~\ref{seven-plane-gauss} excludes ranks \(0,1\), and \(G(4,2,2)\) would require \(8\lambda=2+3d/2\), \(d=4/3\).
	For the remaining groups the quadratic Taylor map to \(\mathcal O(-1)^{\oplus2}\) has numerically trivial coefficients, so its nonzero coefficient or determinant never vanishes.
	If an irreducible factor of the full stabilizer at a general point of the plane were contained in a larger irreducible factor at a special point, the quadratic Taylor map would lose rank there: nonreal factors have no quadratic invariant, and a real irreducible quadratic form cannot lie in the ideal of the plane.
	Thus the factors remain direct factors, and Lemma~\ref{ew-compact-groups} leaves only \(\mu_2^2\).
	Section~\ref{ew-compactapp-halfplane} gives the branch divisor.
	The planes of two disjoint triples meet transversely at one point and share no reflection, since otherwise they would lie in a smooth threefold.
	Their common point would have full stabilizer \(\mu_2^4\), inducing an excluded node.
	With \(q=\sum v_i/2\) and \(d_I=q-\sum_{i\in I}v_i\), a zero \(v_i\) gives two complementary triples \(I\) with \(d_I=0\), which form this excluded pair; hence \(v_i\ge1\).
	Every factor of \(\mu_2^2\times G'\) has rank at most \(2\), proving the final assertion.
\end{proof}

\subsection{The group \texorpdfstring{\(G_{31}\)}{G31}}

\begin{lemma}
	\label{ew-gaussian-g423-local}
	Let \(T=G(4,2,3)\) be the full stabilizer of a generic point of a fixed curve in a smooth finite reflection quotient of dimension \(4\).
	Assume that every upper group of the curve is \(G(4,2,4)\), \(G_{31}\), or \(T\times A_1\).
	The coarse frames corresponding to the basic polynomial invariants have degrees \(4,6,8\).
	At the first two types their inertia characters are \((1,-1,1)\), and the degree correction of the coarse frame in polynomial degree six is \(-1/2\).
	The following orders hold, independently of arithmeticity or compactness.
	In the second table, the poles belong to the inverse rows in degrees eight and six, with the latter restricted to the kernel of the former.
	The last column of the second table records the order of vanishing of the column of the normal differential in degree four:
	\[
		\begin{gathered}
			\begin{array}{c|rr}
				&\operatorname{ord}(df)&\operatorname{ord}(\det\mathsf{M})\\ \hline
				G(4,2,4)&0&3\\
				G_{31}&1&10\\
				T\times A_1&0&0
			\end{array}\\[1ex]
			\begin{array}[b]{c|rrr}
				&\substack{\text{pole in}\\\text{degree eight}}&\substack{\text{pole in degree six}\\\text{on the kernel}}&\substack{\text{zero in}\\\text{degree four}}\\ \hline
				G(4,2,4)&1&1&1\\
				G_{31}&4&4&2\\
				T\times A_1&0&0&0
			\end{array}.
		\end{gathered}
	\]
	At \(G_{31}\), the inverse differential has a component of pole order \(5\) in polynomial degree six before restriction to the kernel specified in the table.
\end{lemma}

\begin{proof}
	Take \(f_4=\sum y_j^4,f_6=(y_1y_2y_3)^2,f_8=\sum_{j<k}y_j^4y_k^4\).
	For \(G(4,2,4)\), use invariants \((v+f_4,vf_4+f_8,u^2f_6,vf_8+f_6^2)\), \(v=u^4\), and frame \((\partial_{f_4},u^2\partial_{f_6},\partial_{f_8})\); removing the tangent gives \(\mathsf{M}_U\) below.
	For \(G_{31}\), the full normalizer is \(T\times\mu_4\) \cite[Table 8 and proof of Lemma 5.13]{KrishnasamyTaylor}.
	From the roots and coroots in \cite[Table 2]{KrishnasamyTaylor}, generate \(I-\Transpose{\mathsf A_j}\mathsf B_j\), \(1\le j\le5\), and take the power sums of degrees \(8,12,20,24\) of the \(240\) dual root forms and of degrees \(4,6,8\) of the \(192\) forms in \((1,2,3)T\), subtracting a multiple of \(f_4^2\) in degree \(8\).
	The basis, coefficient matrix, and normal matrices are
	\begin{equation}\label{ew:compact-normal-matrices}
		\begin{gathered}
			\begin{pmatrix}1&1&0&0\\0&0&1&0\\0&-1-i&(1+i)/2&0\\0&0&0&1\end{pmatrix},\quad
			C=\begin{pmatrix}
				1       & 1         & 1      & 1           \\21/2&-33/2&33&231\\
				10795/2 & -41021/14 & 9690/7 & 226746      \\
				105165  & 129789/5  & 4807   & -19530841/5
			\end{pmatrix},\\
			\mathsf{M}_U=\begin{pmatrix}v&0&1\\0&v&0\\0&0&v\end{pmatrix},\quad
			\mathsf{M}_R=\operatorname{diag}(v^2,v^4,v^5)A\operatorname{diag}(1,1,v^{-1}).
		\end{gathered}
	\end{equation}
	Here \(A\) is the lower Schur complement of \(C\), with \(A_{11}=-27\), leading minor \(2070090/7\), determinant \(-2066330716560\), and last cofactor column \(\Transpose{(3625040628/7,-4491828,2070090/7)}\).
	With the automorphy factor, \(u\mapsto iu\) gives characters \((1,-1,1)\); at \(G_{31}\), the image of the line fixed by \(T\) under the basic invariants is \((v^2,v^3,v^5,v^6)\) up to nonzero rescaling.
	For \(B=A^{-1}\), the kernel of the projection to the coordinate of greatest invariant degree is \(c_3=-(B_{31}v^3c_1+B_{32}vc_2)/B_{33}\), where the middle row becomes \(v^{-2}(B_{21}-B_{23}B_{31}/B_{33})c_1+v^{-4}(B_{22}-B_{23}B_{32}/B_{33})c_2\).
	The last coefficient of this middle row is \(A_{11}\) divided by the leading minor of \(A\), which is nonzero, giving poles \(4,4\), rather than pole order \(5\) before restriction to this kernel; the saturated first column and the derivative of this image have orders \(2,1\), giving \(T_C(D_R)\).
	Products give the identity matrix; clearing the poles in polynomial degrees eight and six computed from \eqref{ew:compact-normal-matrices}, with the row in degree six restricted to the kernel of the row in degree eight, leaves primitive rows on every stalk, proving the table.
\end{proof}

\begin{lemma}
	\label{ew-compact-gthirtyone}
	The group \(G_{31}\) does not occur.
\end{lemma}
\begin{proof}
	Take a parabolic subgroup \(T=G(4,2,3)\) of the full stabilizer \(G_{31}\) of a point and the fixed curve \(C\) of the disk fixed by \(T\); let \(U,R,s\) count \(G(4,2,4),G_{31},T\times A_1\) on \(C\), with \(R\ge1\).
	Lemma~\ref{ew-gaussian-g423-local} gives saturated tangent \(T_C(D_R)\), \(\operatorname{ord}(\det\mathsf{M})=10\), and pole order \(4\) for the inverse row in degree eight.
	The inverse row in degree six has pole order \(4\) on the kernel of the row in degree eight, but \(5\) on the whole bundle.
	The everywhere surjective inverse rows give line quotients of degrees \(h_8=-8t+U+4R\), \(h_6=-6t+U/2+7R/2\), \(h_4=-4t+U+2R\), with \(2t=2x+3(U+R)/4+s/2\) and \(a+2x-R=-18t+5U/2+19R/2\).
	For \(V=2t-a\ge0\), \(11h_8=-4V-2U-R-2s<0\); put \(e=-h_8\).
	For three coincident labels, \(J\) has degree \(3a+2x-R\ge a\) by \(R\le2a+2x\).
	Both successive maps kill \(J\), since \(\deg J-h_6=U/2-3e/2+2a>0\), and \(e\le a\).
	With \(b=a-e\), \(\delta=e-R\), the degrees give \(11\delta\ge20b\), \(x=-e/2-\delta+b/2<0\), \(b=e+2\delta-2\), \(20e+29\delta\le40\).
	The sole row \((a,e,U,R,s)=(2,2,6,2,0)\) gives a smooth conic with two tangent zeros, which is impossible.
	Thus two labels coincide, namely those of \(I=\{1,2\}\); the quotient associated with the invariant of greatest degree is the common relative tangent line bundle \(M=M_I\) of their sections \cite[Proposition~2.2]{NakanoWeightedConfigurations2026}, and \cite[Proposition~2.3]{NakanoWeightedConfigurations2026} gives \(h_6\ge e\), equivalently
	\begin{equation}
		\label{ew-compact-g31degrees}
		\begin{aligned}
			t&=(e+U+4R)/8,&a&=(6R+U-9e)/4-2x,\\
			s&=(e-2U+R)/2-4x,&2R&\ge U+7e.
		\end{aligned}
	\end{equation}
	Let \(D\) be the divisor where labels \(1,2\) coincide, the blow-up of \(\mathbb P^3\) at five points, and let \(\pi:D\to\mathbb P^3\) be its blow-down.
	For \(j=3,\ldots,7\), let \(\beta_j\) be the degree of the zero divisor of the bracket between the common section of labels \(1,2\) and section \(j\).
	For the factorization \(f_D:C\to D\) of \(f\), put \(\gamma=\pi\circ f_D\), \(d=\deg\gamma^*\mathcal O(1)\), \(B_0=\sum_{j=3}^7\beta_j\) for the five brackets; then \(e=d-B_0\), \(a=2d+e\).
	The image is not a line, by \(R\ge1\) tangent zeros, so a general plane projection has positive dual degree at most \(m=2d+2x-R\).
	Since \(4m=6x-5V-s\), \(m,x>0\).
	For the multidegrees \(v_i\) of \(5(2L-H)\) in Lemma~\ref{ew-compact-halfplanes}, put \(q=\sum v_i/2\), \(\rho=q-v_1-v_2\); then \(1\le v_j\le\rho\) for \(j\ge3\), \(5V=e\rho+\sum_{j=3}^7v_j\beta_j\), and \(v_j=\rho\) implies \(\beta_j=0\).
	Every triple \(\{j,k,l\}\subset\{3,\ldots,7\}\) of weight sum \(q\) gives a cross-ratio \([1j][kl]:[1k][jl]\) with no base point on \(C\), by avoidance of the planes of degree \(1/2\).
	The two products are sections of the same line bundle, so their ratio is independent of local frames.
	Its degree is \(e+\beta_j+\beta_k+\beta_l>0\), and all points of \(C\) with full stabilizer \(G_{31}\) ramify, giving \(R\le2(e+\beta_j+\beta_k+\beta_l)+2x\).
	Four vanishing \(\beta_j\) give \(R\le3e+3x\) by three cross-ratios with disjoint base loci; fewer than three nonconstant ones force rational normalization.
	Two identical sections \(j,k\ge3\) give \(R\le2e+2x\), and two identical pairs would make a cross-ratio of positive degree constant.
	The exclusions before the finite test are
	\[
		\begin{array}{c|c|c}
			\text{condition}&\text{inequalities}&\text{contradiction}\\\hline
			R\le3e+3x&16m+20e\le14x,\quad4m+15e\ge16x&32m\le-22x\\
			\rho\ge5&53m+65e\le62x&107m\le-22x\\
			\rho\ge3,\ v_j\ge2&33m+30e\le32x&25m\le0
		\end{array}.
	\]
	The first row excludes \(a=3e\) and \(\rho=1\).
	The eleven remaining multidegrees are parametrized below; exponents mean repetition and the first pair is ordered.
	\[
		\begin{array}{c|c|c|c}
			\rho&\nu&(v_3,\ldots,v_7)&(v_1,v_2)\\\hline
			2&1,3,4&(1^{5-\nu},2^\nu)&(\lfloor(1+\nu)/2\rfloor,\lceil(1+\nu)/2\rceil)\\
			3,4&0,1,3,4&(1,(\rho-1)^{4-\nu},\rho^\nu)&(\lfloor(2\rho-3+\nu)/2\rfloor,\lceil(2\rho-3+\nu)/2\rceil)
		\end{array}
	\]
	This follows from \(v_1+v_2=\sum_{j=3}^7v_j-2\rho\) and the inequalities \(\sum_{i\in I}v_i\le q\); the two rows at \(\rho=2\) with all \(v_j=2\), \(j\ge3\), have \(B_0=0\) and are already excluded.
	For \(\rho=2,\nu=1\), six cross-ratios give \(R\le2e+B_0+2x\).
	Set \(S=12s\), \(Q=6(2e+B_0+2x-R)\), and \(E=21e+5U+80x-52B_0=0\).
	Substitution of \eqref{ew-compact-g31degrees} and \(a=3e+2B_0\) gives \(828e+7S+106Q-3E=-96x\).
	For \(\nu=3,4\) use \(R\le2e+2x\), \(R\le3e+3x\).
	For \(\rho=3,4\), \(\nu=0,1,3\), let \(l\) be the bracket degree for weight \(1\) and \(H_0\) their sum for weight \(\rho-1\).
	Averaging gives \(R\le2e+3H_0/2+2x\).
	Put \(Q=6(2e+3H_0/2+2x-R)\) and \(E_\rho=(45-12\rho)e+5U+80x-52l-(28+12\rho)H_0=0\).
	Substitution of \eqref{ew-compact-g31degrees} and \(a=3e+2(l+H_0)\) gives
	\[
		(180\rho-150)U+16\rho S+(8\rho+115)6m+(20\rho-40)Q+15E_\rho
		=-(528\rho-1500)x.
	\]
	The left side is nonnegative and the right side negative; \(\nu=4\) has four vanishing \(\beta_j\) and uses the first exclusion.
\end{proof}

\subsubsection{Determinant valuations along a fixed curve with full generic stabilizer \texorpdfstring{\(A_3\)}{A3}}
\label{ew-compactapp-a3}
Suppose that on a fixed curve \(C\) with full generic stabilizer \(A_3\), the kernel corresponding to the invariant of degree \(2\) is the quotient of the pulled back tangent bundle of a coincidence plane \(Z\) by the saturated tangent line.
Let \(\eta:S\to Z\) resolve an inverse image component of degree \(e\), with \(L|_Z=\lambda\mathcal O_Z(1)\).
For \(\mathfrak S_4\) invariants on \(\sum y_i=0\), the plane equations start as \(\Phi_3=q_3+O(q^2)\), \(\Phi_4=q_4+c(u)q_2^2+O_{>4}(y)\).
The form \(d\Phi_3\wedge d\Phi_4\) has degree at least \(5\) for general rank \(0\).
For rank \(1\) write \(y=(v+r,v-r,z-v,-z-v)\); then \(q_2=4v^2+2r^2+2z^2\), \(q_3=6v(r^2-z^2)\), \(q_4=4v^4+12v^2(r^2+z^2)+2r^4+2z^4\).
The \(r^2\) coefficients in \(\Phi_3,\Phi_4\) have degrees at least \(1,2\), so their determinant with the tangential differential has degree at least \(4\).
At a divisor with minimum normal valuation \(k\), the leading vector satisfies \(q_3(y^0)=q_4(y^0)+c(u)q_2(y^0)^2=0\) and \(q_2(y^0)\ne0\), since otherwise all invariants vanish.
Thus \(\eta^*C\) has valuation \(2k\), and the two determinants have valuations at least \(5k,4k\).
Their saturated kernels are direct summands over the discrete valuation ring, including a critical divisor after removal of the differential's common factor.
Exceptional curves have zero intersection with a general pulled back line; hence the determinant divisors contain \(5\eta^*C/2\), \(2\eta^*C\).
Inequality~\eqref{seven-plane-gauss-current} bounds their degrees by \((2\lambda-2)e\), \((3\lambda-2)e\), as used in Lemma~\ref{ew-compact-primitive}.

\subsection{Semiampleness of \texorpdfstring{\(L-H\)}{L-H} and the groups \texorpdfstring{\(A_4,D_4\)}{A4,D4}}
\begin{lemma}
	\label{ew-compact-nef}
	Every \(\lambda_I\ge1\), and \(L-H\) is semiample.
\end{lemma}
\begin{proof}
	Lemmas~\ref{ew-compact-gthirtyone} and~\ref{ew-compact-gthirtytwo} exclude \(G_{31},G_{32}\), and Lemma~\ref{ew-compact-bfour} then excludes \(G(4,2,4)\).
	The remaining seven groups lack the \(G(4,2,2)\) factor required at the twelve triple points of a plane of degree \(1/2\) by Lemma~\ref{ew-compact-halfplanes}.
	By Lemma~\ref{ew-compact-gthirtytwo}, all reflection orders are \(2\), so the reduced reflection divisor \(\mathscr B\), the sum of the distinct images of the mirrors, satisfies \(5L=K_X+\mathscr B/2\), and \(10L=\mathscr B-2H\) is integral.
	In general rank \(2\), the induced indices \(2,4\) give \(3\lambda=-3+d_2/2+3d_4/4\), hence \(12\lambda,10\lambda\in\mathbb Z\) and \(2\lambda\in\mathbb Z\).
	The exclusion of \(\lambda=1/2\) and Lemma~\ref{seven-plane-gauss} for ranks \(0,1\) yield \(\lambda\ge1\).
	For the multidegrees \(v_i\) of \(5(2L-H)\), \(d_I=10\lambda_I-5\ge5\) and \(v_i=d_I+d_{I'}\ge10\) for complementary triples \(I,I'\) not containing \(i\); the multidegrees \(v_i-10\) and the degrees \(d_I-5\) on the planes \(Z_I\) prove semiampleness of \(10(L-H)\) by \cite[Proposition~2.6]{NakanoWeightedConfigurations2026}.
\end{proof}

\begin{lemma}
	\label{ew-compact-primitive}
	At least one point has stabilizer \(F_4,H_4\), or \(G_{29}\).
\end{lemma}
\begin{proof}
	Otherwise Lemma~\ref{ew-compact-bfour} leaves \(A_4,D_4\).
	On a fixed curve \(C\) with full generic stabilizer \(A_3\), count \(A_4,D_4,A_3\times\mu_2\) by \(R_A,R_D,s\), and let \(\Delta_A\) be the reduced divisor on \(C\) of the points with full stabilizer \(A_4\).
	The matrices of Section~\ref{ew-compactapp-matrices} give tangent \(T_C(\Delta_A)\) and a surjective quotient \(Q\) of rank \(2\), with \(h_3=-3t+R_A+R_D/2\), \(h_4=-4t+R_A+R_D\).
	The degree calculation gives \(2t=2x+(R_D+s)/2\), \(a=-2x-9t+4R_A+5R_D/2\), \(h_3+h_4=a/2-3t/2-s/4<0\).
	By \cite[Proposition~2.4]{NakanoWeightedConfigurations2026}, \(Q\simeq M^{\oplus2}\), \(\deg M=-a\), with kernel given by the quotient of the pulled back tangent bundle of the coincidence plane by the saturated tangent line; hence \(10a=6t+s\), \(a+6R_A=18x+4s\), positive even \(a\) and \(t/a\le5/3\).
	If \(R_A\ge1\), generic rank \(2\) is excluded by the degrees \(1,2\) of the effective images for \(A_2,2A_1\subset A_4\).
	For ranks \(0,1\), Section~\ref{ew-compactapp-a3} and \eqref{seven-plane-gauss-current} give \(5a/2\le2\lambda-2\) or \(2a\le3\lambda-2\), contrary to \(a\ge2\), \(\lambda=t/a\le5/3\).
	If \(R_A=0\), the surjective map whose target has degree \(h_3\) forces \(h_3\ge-a\), giving \(12x+5s/2\ge0\), \(14x+3s\le0\).
	Thus \(x=s=0\), \(a=t=R_D/4\), but \(\deg\det Q=-a\), contrary to \(Q\simeq M^{\oplus2}\).
\end{proof}

\subsection{The arithmetic volume}
\begin{lemma}
	\label{ew-compact-volume}
	For a cocompact arithmetic lattice generated by reflections of order \(2\), put \(d_F=[F:\mathbb Q]\), \(\Delta=D_K/D_F^2\), \(w=|\mu_5(K)|\), and \(\varpi=288/(44/7)^{14}\).
	Then
	\[
		\frac{\chi_{\mathrm{orb}}}{5w\rho_{K/F}}\in\mathbb Z_{>0},\qquad
		\rho_{K/F}=2^{-4d_F}\zeta_F(-1)L_{K/F}(-2)\zeta_F(-3)L_{K/F}(-4),
	\]
	and \(\chi_{\mathrm{orb}}>5wD_F^{12}\Delta^7\varpi^{d_F}\).
\end{lemma}
\begin{proof}
	For \(p:\mathrm{SU}(1,4)\to\mathrm{PU}(1,4)\), choose a maximal arithmetic group above \(p^{-1}\Gamma\), with principal subgroup \(\Lambda\).
	The quotient of this maximal group by \(\Lambda\) is a \(5\)-group \cite[Sections 2.4 and 2.9]{PrasadYeungSpaces}; every normalized order \(2\) reflection has \(-r\in\Lambda\), so \(\Gamma\subset p(\Lambda)\).
	Prasad's formula gives \(\chi_{\mathrm{orb}}=5wj\rho_{K/F}\prod_{v<\infty}e'(P_v)\), with \(j=[p(\Lambda):\Gamma]\) and \(e'(P_v)\in\mathbb Z_{>0}\) \cite[Sections 2.3 and 2.6--2.9, equations (7)--(8)]{PrasadYeungSpaces}.
	The functional equations give \(\rho_{K/F}=D_F^{12}\Delta^7(288/(2\pi)^{14})^{d_F}\) times an Euler product greater than \(1\): at inert places the denominators \((1-q^{-2})(1+q^{-3})\), \((1-q^{-4})(1+q^{-5})\) are less than \(1\), as are those at split and ramified places.
	Use \(2\pi<44/7\).
\end{proof}

\begin{lemma}
	\label{ew-compact-770}
	Under the hypotheses of Lemma~\ref{ew-compact-volume}, \(\chi_{\mathrm{orb}}\ne770\).
\end{lemma}
\begin{proof}
	Section~\ref{ew-compactapp-770} proves this by enumerating the fields and showing that \(770/(5w\rho_{K/F})\) is never an integer.
\end{proof}

\subsubsection{The value \texorpdfstring{\(770\)}{770}}
\label{ew-compactapp-770}
Suppose \(\chi_{\mathrm{orb}}=770\) in Lemma~\ref{ew-compact-volume}.
For \(F=\mathbb Q\), Meyer's theorem makes the indefinite rational form in ten variables isotropic, contradicting compactness \cite[Chapter IV, Section 3.2, Corollary 2]{Meyer}.
For \(d_F\ge5\), use \(D_F^{1/d_F}>34/5\) \cite[Proposition 2(d)]{PrasadYeungSpaces}, \((34/5)^{12}\varpi>1\), and \(5((34/5)^{12}\varpi)^5>770\).
For \(d_F=2,3,4\), the discriminant bounds are \(44,230,1226\).
The following filters use the proved exhaustive small discriminant tables \cite[Section 3.4, proof of Theorem 1(iv)(1)--(3)]{PrasadYeungSpaces}.
For \(d_F=4\), \(D_F=725,1125\) and \(\Delta\le2,\Delta=1\) leave only \(\mathbb Q(\zeta_{15})\), excluded by \(25\rho_{K/F}=45814652/30375>770\).
For \(d_F=3,2\), substituting \(\Delta_{\max}+1\) into the lower bound of Lemma~\ref{ew-compact-volume} gives
\[
	\begin{gathered}
		\begin{array}{c|rrrrr}
			D_F&49&81&148&169&229\\\hline
			\Delta_{\max}&14&5&2&1&1
		\end{array},\\
		\begin{array}{c|rrrrrrrrrrrrr}
			D_F&5&8&12&13&17&21&24&28&29&33&37&40&41\\\hline
			\Delta_{\max}&40&17&8&7&4&3&2&2&1&1&1&1&1
		\end{array}.
	\end{gathered}
\]
The cubic candidates are \(\mathbb Q(\zeta_7),\mathbb Q(\zeta_9)\).
In degree \(2\), the tables exclude \(13,17,29,37,41\) and leave, at \(21,24,33\), \(\mathbb Q(\sqrt{-3},\sqrt{-7})\), \(\mathbb Q(\sqrt6,\sqrt{-3})\), \(\mathbb Q(\sqrt{-3},\sqrt{-11})\).
At \(40\), the only totally complex quartic of absolute discriminant \(1600\) is \(\mathbb Q(\sqrt5,\sqrt{-2})\), with real discriminant \(5\).

For \(D_F=5,8,12,28\), Minkowski gives real class number \(1\); at \(28\), the only ideal of norm \(2\) is generated by \(3+\sqrt7\).
Write \(K=F(\sqrt\alpha)\) with totally negative \(\alpha\), whose square class is a product of units and principal prime generators.
An odd valuation away from \(2\) contributes its prime norm to \(\Delta\), and above \(2\) contributes exponent \(2v_{\mathfrak p}(2)+1\).
The permitted odd prime norms are \(9,5,11,19,29,31\) at \(D_F=5\) and \(9,7,17\) at \(8\), so at most one occurs.
At \(12\) the prime above \(3\) has mixed signs; at \(28\) no odd prime is allowed, and odd valuations above \(2\) exceed each bound.
Signs at both embeddings leave the unique totally negative unit multiple up to squares and conjugation in each class:
\begin{equation}\label{ew:relative-discriminants}
	\begin{array}{c|l|l}
		F&\alpha&\Delta\\\hline
		\mathbb Q(\sqrt5)&-1,-3,-2-\phi&16,9,5\\
		\mathbb Q(\sqrt2)&-1,-3,-5-2\sqrt2&4,9,17\\
		\mathbb Q(\sqrt3)&-1&1\\
		\mathbb Q(\sqrt7)&-1&1
	\end{array}.
\end{equation}
Here \(\phi=(1+\sqrt5)/2\); the square residues modulo \(4\) in \(\mathbb Z[\phi]\) have pairs \((1,0),(1,1),(2,3)\), giving the relative discriminants in \eqref{ew:relative-discriminants}.
For \(\alpha=-5-2\sqrt2\), \(\beta=1+3\sqrt2\) gives \(\beta^2-\alpha=8(3+\sqrt2)\), of valuation \(6\) above \(2\), versus \(2\) for \(2\beta\).
Hensel makes \(\alpha\) locally square; the split Euler factor gives \(\chi_{\mathrm{orb}}>5\cdot8^{12}\cdot17^7\varpi^2(16384/9765)>868\).
The complete twelve remaining fields are the nine biquadratic fields \(K=\mathbb Q(\sqrt u,\sqrt v)\) for
\[
	(u,v)=(5,-3),(5,-1),(2,-1),(2,-3),(3,-1),(-3,-7),(6,-3),(7,-1),(-3,-11),
\]
and \(\mathbb Q(\zeta_5),\mathbb Q(\zeta_7),\mathbb Q(\zeta_9)\).
Only \(\mathbb Q(\zeta_5)\) has \(w=5\), and every other field has \(w=1\).
Use the exact Bernoulli sum \(L(1-j,\psi)=-(f_\psi^{j-1}/j)\sum_{u=1}^{f_\psi}\psi(u)B_j(u/f_\psi)\) for primitive characters, taking the three quadratic characters or the primitive even and odd cyclotomic characters.
In the listed order the reduced denominators of \(770/(5w\rho_{K/F})\) are \(31,587,19,18331,23,2,551,113,3995,293,130429,18628987\).
None is \(1\), contrary to Lemma~\ref{ew-compact-volume}.

\subsection{The last three groups}
\begin{lemma}
	\label{ew-compact-hfour}
	The group \(H_4\) does not occur.
\end{lemma}
\begin{proof}
	Take a parabolic subgroup \(H_3\) of the full stabilizer \(H_4\) of a point and the fixed curve \(C\) of the disk fixed by it; count \(H_4,H_3\times\mu_2\) on \(C\) by \(R,s\).
	Section~\ref{ew-compactapp-matrices} gives \(2t=2x+(R+s)/2\), \(a=-2x-18t+22R\), and \(h_6=-6t+7R\), \(h_{10}=-10t+10R<0\) by \(44t\ge45R+s\).
	For three labels, \(3(\deg J-h_6)=8a+4x+R>0\) makes both quotient maps vanish on \(J\).
	Their determinant and the map associated with the invariant of greatest degree give \(26a=23R-16x\), \(2x-11R\ge0\), a contradiction.
	For two labels, \cite[Proposition~2.3]{NakanoWeightedConfigurations2026} gives \(h_6\ge e=-h_{10}\); integrality gives \(e=5b\), \(b\in\mathbb Z_{>0}\).
	The complete degree conditions are \(t=R+b/2\), \(a=4R-2x-9b\), \(s=3R-4x+2b\), \(R\ge8b\), but \(a\le t\) gives \(R\le7b\).
\end{proof}

\begin{lemma}
	\label{ew-compact-zero}
	Let \(f:C\to X\) be an immersion from a curve of positive genus.
	Suppose that a semiample GIT class \(B\) has nonnegative multidegrees \(b_i\) with \(\sum_{i\in I}b_i\le\sum_ib_i/2\) for every triple \(I\), and that \(f^*B\) has degree zero.
	If the normal bundle has no nonzero map to a line bundle of negative degree, then \(B\equiv0\).
\end{lemma}
\begin{proof}
	For nonzero \((b_i)\), these inequalities give \(b_i<q=\sum b_i/2\) and at least six positive weights; if one weight vanishes the others are equal.
	Two or three coincident labels of negative relative tangent degree give a line quotient of negative degree, and one such label gives \(\deg f^*B=e(q-b_i)+\sum_{j\ne i}b_j\beta_j>0\).
	Hence all relative tangent degrees are nonnegative and at least six vanish.
	Three sections are distinct, since a triple coincidence and nonvanishing remaining brackets imply \(a=0\).
	These sections trivialize the projective line bundle, and a nonconstant label ramifies in positive genus, contrary to immersion.
\end{proof}

\begin{lemma}
	\label{ew-compact-bthree}
	If some point has full stabilizer \(F_4\) and no point has full stabilizer \(G_{29}\), then \(L\equiv H\).
\end{lemma}
\begin{proof}
	Take a parabolic subgroup \(B_3\) of the full stabilizer \(F_4\) of a point and the fixed curve \(C\) of the disk fixed by it; count \(B_4,F_4,B_3\times\mu_2\) on \(C\) by \(R_B,R_F,s\).
	Section~\ref{ew-compactapp-matrices} and \(a\le t\) give \(a=-14x+4(R_F-s)+s\), \(h_4=R_F-s-4x\), \(2h_6=3(R_F-s)-R_B-12x\), \(R_B\ge15(R_F-s)+2s-60x\).
	For \(h_6=-e<0\) and two labels, \cite[Proposition~2.3]{NakanoWeightedConfigurations2026} gives \(h_4\ge e\), hence \(5e+s\le0\).
	Three labels have \(\deg J-h_4=(8a+4x+R_F)/3>0\); the quotient of rank \(2\) then gives \(3R_B+5R_F=17a+10x\), \(R_F\ge4a+2x\), hence \(-3R_B\ge3a\).
	For \(h_6\ge0\) the only solution is \(R_F=4x\), \(s=R_B=0\), \(a=t=2x\), \((h_2,h_4,h_6)=(4x,0,0)\).
	Now \(g\ge2\) and the flag has no line quotient of negative degree; apply Lemma~\ref{ew-compact-zero} to \(L-H\).
\end{proof}

\begin{lemma}
	\label{ew-compact-g29}
	The group \(G_{29}\) does not occur.
\end{lemma}
\begin{proof}
	Take a parabolic subgroup \(G(4,4,3)\) of the full stabilizer \(G_{29}\) of a point and the fixed curve \(C\) of the disk fixed by it; let \(U,R,s\) count \(G(4,4,4),G_{29},G(4,4,3)\times\mu_2\) on \(C\), with \(R\ge1\).
	Section~\ref{ew-compactapp-matrices} gives a single immersed branch and successive surjections, with \(h_3=-3t+U/4+9R/4\).
	The degree correction \(-3(U+R)/4\) in degree \(3\) is included; the degree equations are \(t=x+3(U+R)/8+s/4\), \(a=-17x-27U/8+5R/8-15s/4\), \(h_8=-8x-2U-2s\).
	The inequality \(a\le t\) forces \(h_8<0\): the contrary for \(x=0\) gives \(U=s=0\), \(a>t\); for \(x=-1\), it gives \(U+s\le4\), \(a-t\ge3-s/4+R/4>0\).
	Put \(e=-h_8=8x+2U+2s\in2\mathbb Z_{>0}\).
	For three labels the image is a smooth plane curve, \(x=(a-1)(a-2)/2-1\), and \(e\le a\).
	The genus bound \(8x-a=a(4a-13)>0\) excludes \(a\ge4\); integrality leaves only \(a=2\) among \(a\le3\).
	For \(a=2\), the sole rows \((U,R,s,t)=(0,6,5,5/2),(5,3,0,2)\) contradict respectively \(10(t/a)\in\mathbb Z\) and the next surjection, since \(\deg J=4>h_3=2\) makes that surjection factor through degree \(-2\).
	Thus exactly two labels coincide and \(a\ge3e\).
	For the multidegrees \(w_i\) of the class \(10(L-H)\), nonzero by Lemma~\ref{ew-compact-770}, let \(q=\sum w_i/2\), \(\rho=q-w_1-w_2\), \(\delta=t-a\).
	The inequalities \(\sum_{i\in I}w_i\le q\) give \(\rho\ge1\), \(0\le w_j\le\rho\) for \(j\ge3\), and
	\[
		10\delta=e\rho+\sum_{j=3}^7w_j\beta_j,\qquad\delta\le8x/5+s/10,\qquad\delta/e\le1/5.
	\]
	Hence \(1\le\rho\le2\), \(w_1+w_2\le6\), \(q\le8\), \(w_i\le6\), \(d_I\le8\).
	The multidegrees of \(10(2H-L)\) and its degrees on the planes \(Z_I\) are at least \(14,2\), so \(2H-L\) is nef and \(\chi_{\mathrm{orb}}=5L^4\le5(2H)^4=12320\).
	The scalar order \(c_{\rm sc}\mid10\) cannot be \(1,2\), since integrality of \(c_{\rm sc}L\) would force \(5\mid\rho\).
	Thus \(c_{\rm sc}=5,10\), \(\zeta_5\in K\); the normalized central \(G_{29}\) element has trace \(1+4i\), giving \(\mathbb Q(\zeta_{20})\subset K\) and \(D_F\ge2000^{[F:\mathbb Q]/4}\).
	Lemma~\ref{ew-compact-volume} gives \(\chi_{\mathrm{orb}}>25(2000^3\varpi)^{[F:\mathbb Q]}\ge25(2000^3\varpi)^4>12320\), since \(2000^3\varpi=11921889952423828125/777727659178477568>1\).
\end{proof}

\begin{theorem}
	\label{ew-seven-compact}
	No cocompact arithmetic lattice \(\Gamma\subset\mathrm{PU}(1,4)\) satisfies \(X_7\simeq\mathbb B^4/\Gamma\).
\end{theorem}
\begin{proof}
	Lemmas~\ref{ew-compact-primitive}, \ref{ew-compact-hfour}, and \ref{ew-compact-g29} force a point with full stabilizer \(F_4\) and exclude \(G_{29}\).
	Lemma~\ref{ew-compact-bthree} gives \(L\equiv H\).
	Since \(H^4=154\), proportionality gives \(\chi_{\mathrm{orb}}=5H^4=770\), contrary to Lemma~\ref{ew-compact-770}.
\end{proof}

\section{Nonuniform realizations outside the Gaussian field}
\label{ew-seven-fields-section}
We retain the notation of Section~\ref{seven-plane-section}.
\begin{proposition}
	\label{ew-seven-fields}
	If \(X\) is the Baily--Borel compactification of a nonuniform arithmetic ball quotient, then the imaginary quadratic field \(K\) defining the lattice is \(\mathbb Q(i)\).
\end{proposition}
We apply Lemmas~\ref{ew-fields-plane-bound} and~\ref{ew-fields-plane-equality} to the 35 coincidence planes.
Lemma~\ref{ew-fields-coxeter} excludes the Coxeter possibilities once the other horizontal groups have been removed.
For fields with two units the remaining primitive group is \(G_{24}\); over the Eisenstein field the following subsections treat the primitive groups, the groups with sixth roots of unity, and finally the groups \(G(3,1,3),G(3,3,3),A_3,B_3\).
Thus the degree bound used at each stage is established before it is applied.

\subsection{The Coxeter groups}
\begin{lemma}
	\label{ew-fields-coxeter}
	Suppose that the full horizontal groups are \(A_3,B_3\), the irreducible interior groups are among \(A_4,B_4,D_4,F_4\), and the groups fixing a ball of dimension 2 are \(A_1^2,A_2,B_2\).
	Suppose also that \(\mathcal L^{\otimes2}\) descends to \(X\), \(L-H\) is nef, and the finite characters on the directions of indecomposable invariants have trivial inertia except for the usual real reflection characters.
	Then \(X\) has no cusps.
\end{lemma}
\begin{proof}
	For a coincidence plane \(Z\) of generic rank 2, put \(L|_Z=\lambda c_1(\mathcal O_Z(1))\), \(\lambda\ge1\), \(2\lambda\in\mathbb Z\).
	All inclusions use actual full parabolic stabilizers.
	The local determinants below give \(10\lambda=4+5\deg D_B\) for \(A_2\), and \(6\lambda-2=2\deg D_B\), \(\deg D_B\le2\lambda-1\) for \(B_2\), both impossible.
	For \(A_1^2\), with \(D_A,D_B,D_P\) the reduced curves in \(Z\) with full generic stabilizers \(A_3,B_3,A_1\times B_2\), respectively, the local determinants give
	\begin{equation}
		\label{ew-fields-plane}
		4\lambda-2=\tfrac32\deg D_A+3\deg D_B+\deg D_P.
	\end{equation}

	For a fixed curve \(C\) with full generic stabilizer \(B_3\), count \(B_4,F_4,B_3\times A_1\) by \(R_B,R_F,p\), and cusps \(1111,1112,1122,1222\) by \(s_i\).
	Put \(x=g-1\), \(s=\sum s_i\), \((\alpha,\beta,\gamma)=(s_1+s_2+s_3+2s_4,s_1+s_2+2s_3+2s_4,s_1+2s_2+2s_3+2s_4)\).
	Let \(\mathcal K_{24}\) be the kernel of the quotient associated with the invariant of degree \(6\), and let \(\mathcal K_2\subset\mathcal K_{24}\) be the kernel of its successive quotient associated with degree \(4\).
	The bundles \(\mathcal K_2,\mathcal K_{24}\) have ranks \(1,2\), respectively.
	We write \(h_2=\deg\mathcal K_2\), \(h_4=\deg(\mathcal K_{24}/\mathcal K_2)\), and \(h_6=\deg(\mathcal N/\mathcal K_{24})\).
	The inverse differential gives everywhere surjective maps onto these successive line quotients, with
	\begin{equation}
		\label{ew-fields-Bflag}
		(h_2,h_4,h_6)=(-2t+R_B+2R_F+\alpha,-4t+R_B+2R_F+\beta,-6t+R_B+3R_F+\gamma).
	\end{equation}
	The full characters and widths give
	\begin{equation}
		\label{ew-fields-Bdegrees}
		2t=2x+s+(R_B+R_F+p)/2,\quad a=-2x-12t+3R_B+7R_F+\alpha+\beta+\gamma.
	\end{equation}
	Section~\ref{ew-finite-Belimination} leaves, with cusps, only
	\begin{equation}
		\label{ew-fields-Bresidue}
		(a,t,p,R_B,R_F;s_1,s_2,s_3,s_4;h_4,h_6)=(1,3/2,2,0,0;0,3,1,0;-1,-1).
	\end{equation}
	The plane of \(C\) has full generic stabilizer \(A_1^2\), \(\lambda=3/2\), \((\deg D_A,\deg D_B,\deg D_P)=(0,1,1)\).
	Without cusps, \(R_B=p=0\), which also excludes \(B_4\) through its \(B_3\) parabolics.

	At a cusp with full horizontal group \(B_3\), use the model \(x_1^2=U\), \(x_2^2=T\), \(x_3^2=T-U\), \(x_4^2=T-c_0U\), \(c_0\ne0,1\), of the elliptic curve in its homogeneous coordinates \(x_1,\ldots,x_4\).
	Write \(X_i=\prod_{\alpha=1}^3x_i^{(\alpha)}\), where \(x_i^{(\alpha)}\) is the coordinate \(x_i\) on the \(\alpha\)th of the three copies of the elliptic curve, and let \(I_{\rm sq}\) be the set of indices \(i\) for which \(X_i^2\) is used as a generator; in the grading before the scalar kernel is removed, \(|I_{\rm sq}|\) is the number of entries \(2\) in the cusp type.
	A short reflection is actual exactly at an index in \(I_{\rm sq}\); the other lift has nontrivial fiber character.
	Consider the connected component of the \(A_1^2\) fixed locus that corresponds to the cusp of the chosen plane \(Z\).
	For a path in this component starting at a general nonzero fiber point, \cite[Lemma~2.17]{NakanoWeightedConfigurations2026} lifts the generic stabilizer \(A_1^2\) at that point, and by uniqueness of path lifting the same lifted group, which fixes the starting point of the lifted path, also fixes the endpoint.
	Hence the curve germs used next lie on \(Z\).
	Two long reflections give an excluded branch of \(D_A\), and two short reflections give two branches of \(D_P\); hence their lengths differ.
	The invariant rings on the plane for \(1112,1122,1222,1111\) are \(\mathbb C[u,v]\), \(\mathbb C[u,v]\), \(\mathbb C[u,uv,v^2]\), \(\mathbb C[u^2,uv,v^2]\).
	Embedding dimension leaves the three \(1112\) cusps and the one \(1122\) cusp in \eqref{ew-fields-Bresidue}, with horizontal group of order 2, boundary degree 2, and width 1.
	Let \(D_Q\) be the reduced curve in \(Z\) with full generic stabilizer \(A_1^3\).
	Section~\ref{ew-finite-coxeter} excludes nontransverse interior \(D_P,D_Q\) intersections: their only possible ambient group \(B_3\times A_1\) gives another coincidence plane intersecting \(Z\) in a curve.
	The numbers of branches of \(D_B,D_P,D_Q\) at the cusp are \((1,0,3)\) at \(1112\) and \((1,1,2)\) at \(1122\), with distinct directions \(U,T,T-U,T-c_0U\).
	Thus \(K_Z+\tfrac12(D_A+D_B+D_P+D_Q)=3L|_Z\) gives \(\deg D_Q=13\).
	The cusp of \(D_P\) contributes multiplicity 2, leaving 11 transverse elliptic points of order 2, contrary to \(3=2L\cdot D_P\ge-2+1+11/2=9/2\).
	The actual \(\mathcal L^{\otimes2}\) has trivial cusp character.
	The curvature current of its invariant metric has zero Lelong number by the logarithmic cusp potential in~\cite[Section 2.2, equations (1)--(8)]{MokCompactification}.

	For \(A_3\), count \(A_4,D_4,A_3\times A_1\) by \(R_A,R_D,p\); all cusps are \(1111\).
	Let \(\Delta_A\) be the reduced divisor on \(C\) of the points with full stabilizer \(A_4\), so \(\deg\Delta_A=R_A\), and saturate the tangent bundle to \(T_C(\Delta_A)\).
	The normal quotients have degrees \((h_4,h_3,h_2)=(-4t+R_A+R_D+s,-3t+R_A+R_D/2+s,-2t+R_A+R_D+s)\).
	The comparison gives \(2t=2x+s+(R_D+p)/2\), \(a=-2x-9t+4R_A+5R_D/2+3s\).
	The last two directions give a surjective quotient \(Q\), \(\deg\det Q=a/2-3t/2-p/4<0\).
	By \cite[Proposition~2.4]{NakanoWeightedConfigurations2026}, \(Q=M^{\oplus2}\) on a three point coincidence plane, \(\deg M=-a\).
	Hence \(10a=6t+p\), \(a+6R_A=18x+4p+3s\), \(1\le\lambda\le5/3\).
	Section~\ref{ew-finite-gauss} excludes rank 0 by \(5a/2\le2\lambda-2\); rank 1 forces \(a=1\), \(x=-1\), \(\lambda=3/2\), then \(R_D=-4R_A-1\).
	The plane therefore has full generic stabilizer \(A_1^2\).
	At a cusp with full horizontal group \(A_3\), let \(E\) be the elliptic curve in the description of its horizontal torus, with origin \(0\), and put \(A=\ker((z_i)\mapsto\sum_i z_i:E^4\to E)\).
	Coordinate permutations give \(\pi:A\to A/\mathfrak S_4\simeq|4[0]|\simeq\mathbb P^3\), and the polarization determined by the least positive central period is \(\pi^*\mathcal O(1)\) with its induced linearization~\cite[Section 5.1]{Rains}.
	The fixed components for \(\langle(12),(34)\rangle\) are \((z,z,u-z,u-z)\), \(u\in E[2]\); they map through \(|[0]+[u]|\) to noncollinear Veronese conics, incompatible, at a smooth cusp \(y\) of \(Z\), with the boundary line \(\mathbb P(T_yZ)\) and the exceptional bundle \(\mathcal O(-1)\).

	\subsubsection{Coxeter groups and the local normal directions}
	\label{ew-finite-coxeter}
	For \(A_2\subset B_3\), expand the power sums at \(x_i=u+y_i\), \(\sum y_i=0\), with \(q_k=\sum y_i^k\), \(v=u^2\) and coarse frame \(ue_3\).
	Removing the tangent direction gives \(\left(\begin{smallmatrix}4v&4v\\12v^2&20v^2\end{smallmatrix}\right)\), of determinant \(32v^3\).
	At \(A_3\), the image \((12u^2,-24u^3,84u^4)\) of the line fixed by \(A_2\) is not immersed.
	For \(B_2\) the matrix is \(\left(\begin{smallmatrix}v&1\\0&v\end{smallmatrix}\right)\).
	For \(A_1^2\) the determinant character contributes \(-\deg D_A/2\), and local orders at \(A_3,B_3,A_1\times B_2,A_1^3\) are \(2,3,1,0\), proving \eqref{ew-fields-plane} without splitting.

	The matrix at \(B_4\) and power sums over the short roots on both lines fixed by parabolic subgroups \(B_3\) of \(F_4\) give corrections \((1,1,1)\) and \((2,2,3)\) with nonzero minors~(Section~\ref{ew-compactapp-matrices}).
	The first three even Taylor coefficients of the products of sections along the branches at a cusp have a Vandermonde determinant.
	Thus the corrections in directions \((2,4,6)\) at \(1111,1112,1122,1222\) are \((1,1,1),(1,1,2),(1,2,2),(2,2,2)\), with width 1 at every common fixed point of trivial fiber character.
	Removing the kernel of the fiber character from \(2222\) gives \(1111\), replaces \(q\) by \(q^2\), and halves the period.
	For \(A_3\subset A_4,D_4\), Section~\ref{ew-compactapp-matrices} gives \(\det \mathsf{M}_A=625u^3/2\), \(\det \mathsf{M}_D=64v^3\).
	The last two rows of \(\mathsf{M}_D^{-1}\) in the coarse frame \(ue_3\) are \((0,3/(4v),0)\), \((0,-1/2,1/(12v))\).
	These give independent coefficients with simple poles for the inverse row in degree four and the row in degree three restricted to the kernel of the row in degree four; triality, tangent saturation at \(A_4\), and the \(1111\) theta coefficients complete the surjectivity and character calculation.

	For the full normalizers of fixed planes, enumerate orthogonal root pairs of pointwise stabilizer \(A_1^2\) in \(A_4,B_4,D_4,F_4\) and all direct products of rank 4 of \(A_1,A_2,B_2,A_3,B_3\).
	Restrict the other roots to the fixed plane and take invariants of \(N_F(A_1^2)/A_1^2\)~\cite[Theorems 2.1, 3.9, 3.12, Table 5]{KrishnasamyTaylor}.
	Roots orthogonal to the fixed space generate every stabilizer, proving exhaustiveness.
	Without branches of \(D_A\) or points with full stabilizer \(B_4\), only \(B_3\times A_1\) gives a nontransverse intersection of \(D_P\) and \(D_Q\); in \(F_4\), equal root lengths give \(B_2\), so \(A_1^2\) uses both lengths.

	\subsubsection{Weighted vanishing on a surface}
	\label{ew-finite-gauss}
	Let \(\Sigma\subset X\) be a smooth coincidence plane or the strict transform in \(\operatorname{Bl}_5\mathbb P^3\) of a plane in \(\mathbb P^3\), containing \(C\), such that the projections of the inverse differential to the two coordinates of greatest invariant degree vanish along \(\Sigma\).
	Let \(\ell\) be a general line in the plane, or its strict transform, avoiding the cusp and blow-up centers, and put \(\lambda=L\cdot\ell\), \(c=\det N_{\Sigma/X}\cdot\ell\), \(d=C\cdot\ell\).
	Thus \(K_\Sigma\cdot\ell=-3\), \(\ell^2=1\), and \(c=-3+H\cdot\ell\) by adjunction.
	As in the proof of Lemma~\ref{seven-plane-gauss}, resolve a component above \(\Sigma\) in a neat cover \(Y\), with maps \(h:S\to Y\) and \(\tau:S\to\Sigma\) of degree \(e\), and let \(R\) be the common divisor on \(S\) of the maximal minors of \(dh\); then \((K_S-R)\cdot\tau^*\ell\ge3\lambda e\) as in \eqref{seven-plane-gauss-current}.
	In generic rank 0, the conormal sheaf \(N_h^*=\ker(h^*\Omega_Y\to\Omega_S)\) has determinant \(5h^*L-K_S+R\), and the determinant of the map \(\tau^*N_{\Sigma/X}^*\to N_h^*\) given by the differential of \(Y\to X\) is a section of \(5h^*L-K_S+R+\tau^*\det N_{\Sigma/X}\).
	In generic rank 1, \(h(S)\) lies in the component \(Y_1\) of dimension 3 of the fixed locus of a reflection \(\sigma\) in the generic stabilizer, and \(N_{h,1}^*=\ker(h^*\Omega_{Y_1}\to\Omega_S)\) has determinant \(4h^*L-K_S+R\).
	For a \(\sigma\)-equivariant local retraction onto \(Y_1\) and a local equation \(r\) of \(Y_1\) with \(\sigma^*r=-r\), the pullback \(g\) to \(Y\) of a local equation of \(\Sigma\) has the form \(g|_{Y_1}+r^2g_2+O(r^4)\).
	For the pullbacks \(g,g'\) of two local equations of \(\Sigma\), the section \(d(g|_{Y_1})\,g'_2-d(g'|_{Y_1})\,g_2\) on \(S\) takes values in \(N_{h,1}^*\otimes h^*N_{Y_1/Y}^{*\otimes2}\), and \(N_{Y_1/Y}^{*\otimes2}\) is \(2L\) twisted by a finite character~\cite[Lemma~2.12]{NakanoWeightedConfigurations2026}.
	A change of the normal coordinate alters \(g_2\) by \(a(g|_{Y_1})\) for a vector field \(a\) on \(Y_1\), and the added terms cancel because \(d(g|_{Y_1})\) and \(d(g'|_{Y_1})\) lie in the line \(N_{h,1}^*\) on \(S\).
	Thus this section is independent of choices, and its local formula is regular also at points with larger stabilizers, where \(Y_1\) is smooth and \(\sigma\) is linearizable; it is a section of \(6h^*L-K_S+R+\tau^*\det N_{\Sigma/X}\).
	These two classes have degrees at most \((2\lambda+c)e\) and \((3\lambda+c)e\) on \(\tau^*\ell\).
	For transverse order \(k\), the vanishing orders of the curve equation and of the conormal determinant, a section of the class above for the generic rank, are at most and at least, respectively, \((3k,6k),(2k,7k)\) for \(B_3\) and \((2k,5k),(2k,4k)\) for \(A_3\).
	For \(B_3\), let \(f_2,f_4,f_6\) be the basic invariants of the normal representation, of degrees \(2,4,6\).
	In rank \(1\), the coefficient of \(f_2^2\) in the equation of degree six of \(\Sigma\) in these invariants vanishes on the mirror component of \(R\).
	These orders are taken in the image of \(h^*\Omega_Y\) in \(\Omega_S\), not in its saturation, and so include the multiplicity of \(R\).
	Hence the zero divisors of the determinants contain \(2\tau^*C\), \(\tfrac72\tau^*C\) for \(B_3\) and \(\tfrac52\tau^*C\), \(2\tau^*C\) for \(A_3\).
	Since the general line \(\ell\) avoids the images of exceptional curves, these have degree zero on \(\tau^*\ell\), and intersecting with \(\tau^*\ell\) and dividing by \(e\) gives
	\[
		\begin{array}{c|cc}
			&\operatorname{rank}=0&\operatorname{rank}=1\\\hline
			B_3&2d\le2\lambda+c&7d/2\le3\lambda+c\\
			A_3&5d/2\le2\lambda+c&2d\le3\lambda+c.
		\end{array}
	\]

	\subsubsection{The equality cases for a fixed curve with full generic stabilizer \texorpdfstring{\(B_3\)}{B3}}
	\label{ew-finite-Belimination}
	Put \(D_0=3s_1+2s_2+s_3\), \(m_1=s_1+s_2\), \(m_2=2s_1+s_2+s_3+s_4\).
	The comparison gives \(a=-14x+4(R_F-p)+p-D_0\), \(t=x+(R_B+R_F+p+2s)/4\), \(h_4=R_F-p-4x-m_1\), \(h_6=(3(R_F-p)-R_B-12x)/2-m_2\).
	For \(h_6\ge0\), \(a\le t\) forces \(R_B=p=s_1=s_3=s_4=0\).
	Writing \(b=s_2/3\) leaves \(a=t=2x+2b\), \(R_F=4x+2b\), \((h_2,h_4,h_6)=(4x+3b,-b,0)\).
	For \(b>0\), \cite[Proposition~2.4]{NakanoWeightedConfigurations2026} gives \(b=2a\), \(x=-3a/2<-1\); without cusps it gives \(R_B=p=0\).

	For \(h_6<0\) and two coincident labels of relative tangent degree \(-e\), \cite[Proposition~2.3]{NakanoWeightedConfigurations2026} gives \(h_4\ge e\).
	The equality \(R_B=-3(h_4-e)-2(t-a)-2s_1-s_2-p-2s_4\) forces \(t=a\), \(h_4=e\), \(s_3=5e/2\), \(4\mid e\), \(a=3e+2m\), \(g=1+3e/4+m\), where \(m=R_F/4-e\).
	Nonnegative bracket multidegrees force \(L\equiv H\); the inverse row of the next degree on \(\operatorname{Bl}_5\mathbb P^3\) gives a constant row \((M^{-1})^{\oplus4}\to M^{-1}\), \(M=\mathcal O(-h+\mathscr E)|_C\), \(H=\mathcal O(3h-\mathscr E)\).
	The Euler relation of this constant row shows that \(C\) lies in a plane transform with \(c=0\), \(\lambda=3\), \(d=e+m\ge4\).
	Section~\ref{ew-finite-gauss} excludes ranks 0,1; at \(F_4\), degrees \(2,6,8,12\) exclude immersed planes with full generic stabilizer \(A_1^2\) or \(B_2\).
	At \(1122\), the remaining plane, with full generic stabilizer \(A_2\), has coordinates \(qa^3,qb^3,q^2(b^2-a^2)^3,q^2(b^2-c_0a^2)^3\); the determinant \(9c_0(c_0-1)\ne0\) modulo the first three products gives tangent dimension 4.

	For three coincident labels, \(\mathcal N=J\oplus M^{\oplus2}\), \(\deg J=3a+2x>0\), \(\deg M=-a\), \(-a\le h_6<0\).
	The projections to the two coordinates of greatest invariant degree kill \(J\).
	If \(\deg J\le h_4\), their coefficient inequalities force \(a=1\), \(h_6=-1\), \(h_2=-h_4\), then \(h_4=1\), \(s_3=4\), \(t=1\), contrary to \(h_6=2\) in \eqref{ew-fields-Bflag}.
	Thus their quotient of rank 2 is \(M^{\oplus2}\), and
	\[
		\begin{gathered}
			h_2=3a+2x,\quad h_4+h_6=-2a,\quad h_4\le-a,\\
			8t=11a+10x+p+s_1+2s_2+3s_3,\quad
			4R_F=a+14x+3p+3s_1+2s_2+s_3,\\
			4R_B=21a-10x-5p-9s_1-6s_2-3s_3-8s_4,\quad
			4h_4=a-2x-p-s_1-2s_2+s_3.
		\end{gathered}
	\]
	Without cusps these give \(5a-2x\le p\le21a/5-2x\), impossible.
	For \(R_F>0\), the lowest invariant degrees give generic rank at most 1; Section~\ref{ew-finite-gauss}, with \(c=-2\), \(\lambda\le7/2\), forces \(a\le2\), \(x=-1\).
	Putting \(h_4=-a-r\), \(h_6=-a+r\), \(0\le r<a\), gives \(t=2a-1+r/2+s_3/2\), \(R_B=3a/2+1-3r-p/2-3s_1/2-3s_3/2-2s_4\).
	For \((a,r)=(1,0),(2,0),(2,1)\) the bounds \(t\le3/2,4,7/2\) contradict \(t\ge11/6,6,6\).
	With \(R_F=0\), \(x=-1\), \(a\le2\), the plane has full generic stabilizer \(A_1^2\).
	For \(a=2\), \(r=1\) is inconsistent and \(r=0\) gives \(s_2=6\), \(t=3\), violating \eqref{ew-fields-plane}; for \(a=1\) that equality leaves only \eqref{ew-fields-Bresidue}.
\end{proof}

\subsection{Fields with two units and the group \texorpdfstring{\(G_{24}\)}{G24}}
For \(K\ne\mathbb Q(i),\mathbb Q(\sqrt{-3})\), the full horizontal groups are \(A_3,B_3\), with \(G_{24}\) possible only over \(\mathbb Q(\sqrt{-7})\)~\cite{GeckMalle,KrishnasamyTaylor}.
The character fields of \(G_{23},\ldots,G_{27}\) are \(\mathbb Q(\sqrt5)\), \(\mathbb Q(\sqrt{-7})\), \(\mathbb Q(\sqrt{-3})\), \(\mathbb Q(\sqrt{-3})\), and \(\mathbb Q(\sqrt{-3},\sqrt5)\).
After \(G_{24}\) is excluded, the interior groups are \(A_4,B_4,D_4,F_4\) and \cite[Lemma~2.15]{NakanoWeightedConfigurations2026} restricts the groups of rank 2 to \(A_1^2,A_2,B_2\).
Their least positive invariant degree \(2\), descent of \(\mathcal L^{\otimes2}\), Lemma~\ref{ew-fields-plane-bound}, and \cite[Proposition~2.6]{NakanoWeightedConfigurations2026} give every hypothesis of Lemma~\ref{ew-fields-coxeter}.

For \(G_{24}\), the horizontal torus is the Jacobian of the Klein quartic with its principal polarization, and its ring of invariant sections has generators of degrees \(1,2,4,7\)~\cite{Rains}.
The common fixed point of \(G_{24}\) on this torus is unique by \(\det(-g-I)=-1\) for an element \(g\) of order 7~\cite[Proposition 2.1(v) in arXiv:2107.03745v2]{MarkushevichMoreau}.
The section of degree 1 is nonzero there because the quartic is nonhyperelliptic and the tangent representation has no invariant linear form.
Let \(\Gamma_{00}\supset\Gamma_{000}\) be the subspaces of \(H^0(A,\mathcal O_A(2\Theta_0))\) of sections vanishing at the origin of this Jacobian \(A\) to order at least \(4\) and \(6\), respectively~\cite[Section~1, (1.2)--(1.5)]{PaulyPreviato}.
For this polarization, \(\dim\Gamma_{00}=1\) and \(\Gamma_{000}=0\)~\cite[Section~6.2.1]{PaulyPreviato}; this gives a nonzero projection from theta sections of degree 2 in the polynomial invariant of degree 4~\cite[equations~(1.3) and~(1.5), Theorems~1.1(2) and~1.3]{PaulyPreviato}.
A determinant character needs all 21 mirrors, so it cannot occur in order 4 or 5; subtract the square of the section of degree 1.
For a fixed curve with full generic stabilizer \(G_{24}\), let \(R_2\) count points of effective inertia \(2\) and \(s=s_4+s_7\) count cusps according as the direction of degree 14 appears in the derivatives of invariant sections of degree 4 or 7.
Put \(x=g-1\), \(j=2t-s\).
The comparison is \(2t=2x+s+R_2/2\), \(a=-24t+13s-2x\), \(j=2x+R_2/2\).
Hence \(a\le t\) gives \(j>0\), \(s\le25j+4x\le27j\), \(a\le14j\).
For the invariant of greatest degree, the associated target \(Q=E_{14}(4S_4+7S_7)\) has degree \(-7j-3s_4<0\).
Its lower coefficient degrees are \(-5j-3s_4\) and \(-s-4j-3s_4\), so \cite[Lemma~3.5]{NakanoWeightedConfigurations2026} gives \(\dim\operatorname{Hom}(f^*T_X,Q)=1\).
Then \cite[Proposition~2.2]{NakanoWeightedConfigurations2026} forces two coincident labels and \(a\ge21j\), a contradiction.
Without cusps, \(a+2g-2=-24t\) gives \(25a\le2-2g\), also impossible.

\subsection{The primitive Eisenstein groups}
Put \(K=\mathbb Q(\omega)\), \(\omega^2+\omega+1=0\), \(\mathcal O_K=\mathbb Z[\omega]\), \(\delta=1+2\omega\).
Initially \(\mathcal L^{\otimes6}\) descends by~\cite[Proposition~2.11]{NakanoWeightedConfigurations2026}, and \(a\le3t\) by Lemma~\ref{ew-fields-plane-bound}.
The possible full horizontal groups are \(A_3,B_3,G(3,1,3),G(3,3,3),G(6,p,3)\) for \(p=1,2,3,6\), and \(G_{25},G_{26}\).

\label{ew-finite-primitive}
For \(G_{25}\), use the \(E_6\) deformation \(x^3+y^3+z^3+a_1yz+a_2y+b_2z+a_3\) at \((a_1,a_2,b_2,a_3)=(1,-1/4,-1/4,5/108)\).
With \(u=(y-z)/2\), the Kodaira--Spencer map on the slice \(a_1=1\) is an isomorphism onto \(\langle1,u,u^2\rangle=(\mathbb C[x,u]/(x^2,u^3))^{\mu_3}\), where \(\mu_3\) acts by \(x\mapsto\omega x\), \(u\mapsto u\)~\cite{ShepherdBarron}.
Put \(\alpha=(a_2+b_2)/2+1/4\), \(\beta=(a_2-b_2)/2\), \(\gamma=a_3-5/108\).
The regular element of order 3 has fixed weights \(6,9,12\), centralizer \(G_{25}\), and linear terms \((P_6,P_9,P_{12})=(-3\alpha,2\beta,\gamma+\alpha/3)\)~\cite{Springer}.
Reversing weight 9 gives \(G_{26}\), of degrees \(6,12,18\).
For the two models of \(G_{26}\), with generators of degrees \(1,2,3,4\) and \(1,3,4,6\), let \(\theta_0,\theta'_0\) be the generators of degree \(1\) and \(\theta_d,\theta'_d\) those of degree \(d>1\).
Written in \(dP_6,dP_{12},dP_{18}\), \(P_{18}=P_9^2\), the differentials of \(\theta_2/\theta_0^2,\theta_3/\theta_0^3\) and of \(\theta'_3/{\theta'_0}^3,\theta'_4/{\theta'_0}^4\) have first two columns \(\left(\begin{smallmatrix}-2/3&0\\1/9&1\end{smallmatrix}\right)\) and \(\left(\begin{smallmatrix}1/9&1\\1/6&0\end{smallmatrix}\right)\).
For \(G_{25}\) the normal differential is \(\operatorname{diag}(q^2,q^2,q^3)\left(\begin{smallmatrix}b_6&0&*\\0&b_9&0\\0&0&b_{12}\end{smallmatrix}\right)\), \(b_6b_9b_{12}\ne0\).
Weight preserving equivariant changes preserve these two columns.
The section of degree 1 is nonzero at the three common fixed points of \(G_{25}\) on the horizontal torus, whose images in the coordinates \((a_2,b_2,a_3)\) of the slice \(a_1=1\) are \((-\zeta^{-1}/4,-\zeta/4,5/108)\), \(\zeta^3=1\).
In each of the two models of \(G_{26}\) obtained by the involution and the isogeny of degree 3, it is nonzero at the origin, the only common fixed point of the full \(G_{26}\).
Thus the theta projections have the stated orders, trivial fiber characters, and width 1.

In the following calculations, \(n_2,n_3,n_6\) count the points with a stabilizer that is a direct product whose effective inertia on the curve has order \(2,3,6\), respectively, so \(\sigma=n_2/2+2n_3/3+5n_6/6\).
For \(G_{26}\), \(s_1,s_2\) count \(1234,1346\) cusps and \(j=6t\).
There is no irreducible larger group, and
\[
	j/3=2g-2+s_1+s_2+\sigma,\quad
	a=10s_1+14s_2+\sigma-19j/3,\quad h=-3j+4s_1+6s_2.
\]
Since \(j\in\mathbb Z\), the first equation gives \(3\sigma\in\mathbb Z\).
These equations force \(g=0\), \(h<0\), \(a=s_1+s_2+2h+2\).
The lower residual degrees are at most \((h-j)/2,h/3\); for the target associated with the invariant of greatest degree, the Hom space therefore has dimension 1 and \(a\ge-3h\)~\cite[Lemma~3.5]{NakanoWeightedConfigurations2026}.
Equivalently, \(21s_1+13s_2+39\sigma\ge82\) and \(24s_1+14s_2+45\sigma\le92\).
Eliminating \(s_2\) gives \(6s_1+13\sigma\le16\); since \(3\sigma\in\mathbb Z\) and \(\sigma=n_2/2+2n_3/3+5n_6/6\), this leaves \(\sigma\in\{0,2/3,1\}\) and the six pairs \((s_1,\sigma)\) below.
In all these cases, the bounds on \(s_2\) are
\[
	\begin{array}{c|rrrrrr}
		(s_1,\sigma)&(0,0)&(1,0)&(2,0)&(0,2/3)&(1,2/3)&(0,1)\\\hline
		\text{lower bound on }s_2&7&5&4&5&3&4\\
		\text{upper bound on }s_2&6&4&3&4&2&3.
	\end{array}
\]

For \(G_{25}\), let \(R\) count its only irreducible larger group \(G_{32}\).
The comparison, retaining the actual finite character, is
\[
	\begin{gathered}
		2t=2g-2+s+5R/6+\sigma,\quad a=8s-29t+28R/3+\sigma,\\
		\deg(E_6,E_9,E_{12})=(-6t,-9t-R/2,-12t),\\
		h=\deg E_{12}(3S+3D_R)=6-6g-6t+R/2-3\sigma.
	\end{gathered}
\]
It gives \(g=0\), \(h\le0\).
For the matrix at \(G_{32}\), extract \(B\) as in Section~\ref{ew-compactapp-matrices}, divide each column by its nonzero first entry, and set \(A_{ij}=B_{i+1,j+1}-B_{i+1,1}B_{1,j+1}\).
Let \(D\) be the Schur complement of \(A_{11}\) in \(A\); the matrix \(\mathsf{M}=\operatorname{diag}(v^2,v^3,v^4)A\operatorname{diag}(1,1,v^{-1})\) has Schur complement \(\operatorname{diag}(v^3,v^4)D\operatorname{diag}(1,v^{-1})\), and \(A_{11}=-650/297\), \(D_{11}=(-19661200/649539)/A_{11}\).
The determinants are \(\det A=-19489164500000/1162261467\) and \(\det D=\det A/A_{11}\).
The coefficient associated with the invariant of degree 6 has residual degree \(-4t+2+R/6-\sigma<0\).
The other coefficients have zeros \((2,3)\) at cusps and \((3,2)\) at \(G_{32}\), with \(D_{11}(b_{12}/v^2)(0)=D_{12}(b_9/v^3)(0)\).
Surjective evaluation at the \(R\) distinct points gives \(\dim\operatorname{Hom}(\mathcal N,E_{12}(3S+3D_R))=\max(\eta+1,0)+1\), \(\eta=-t+2-R/3-\sigma\).
Exhaustive substitution in the preceding equations, under the conditions in Table~\ref{ew-table-gtwentyfive}, gives the outputs listed there.
\begin{table}[htbp]
	\centering
	\begin{tabular}{p{.12\linewidth}|p{.45\linewidth}|p{.32\linewidth}}
		                                                                             & integer inputs and bounds & \((a,h,\dim\operatorname{Hom})\) \\\hline
		\(\eta\ge0\)                                                                 &
		\(0\le R\le10\), \(0\le s\le4\), \(0<6t<20\), \(6t+R\) even, \(1\le a\le3t\) &
		\((4,-1,2)\), \((2,-1,2)\), \((3,0,3)\)                                                                                                     \\
		\(\eta<0\)                                                                   &
		The same bounds, \(a\ge-3h\), \mbox{\(R+39\sigma/16\le3\)}                   &
		No solution.
	\end{tabular}
	\caption{Finite conditions and outputs for fixed curves with full generic stabilizer \(G_{25}\).}
	\label{ew-table-gtwentyfive}
\end{table}
The outputs \((4,-1,2),(2,-1,2)\) contradict \cite[Proposition~2.2]{NakanoWeightedConfigurations2026}.
For \((3,0,3)\), the relative tangent line bundles along the labeled sections are \(M_i=L_i^2\otimes\det(V)^{-1}\)~\cite[Equation~(2.1)]{NakanoWeightedConfigurations2026}; this identity and \(a\equiv\deg V\pmod2\) make the contribution \((m-1)(e+1)\) of each component of coincident labels of negative relative tangent degree even when \(a\) is odd.
This parity excludes the output \((3,0,3)\), every cusp with full horizontal group \(G_{25}\), and the points with full stabilizer \(G_{32}\) through their actual parabolics.

\subsection{The sixth roots of unity}
Rains's 16 monomial actions come from \(E:z^2=y^3+x^6\), with \(d_2\mid2\), \(d_3\mid3\), \(d_6\mid6\); no translation isogeny occurs because \(1-\zeta_6\) is a unit~\cite[Theorem 3.6, Example 3.11]{Rains}.
Put \(d_0=\gcd(2d_3,3d_2)\), \(g_0=\gcd(d_6,d_0)\), \(q=6/\operatorname{lcm}(d_2,d_3,d_6)\), \(p=6/d_6\).
The full group is \(G(6,q,3)\), its subgroup with trivial fiber character is \(T=G(6,p,3)\), the generator degrees are \((d_6,2d_3,3d_2,6)/g_0\), and the width is \(\kappa=d_0/g_0\).
For \(p>1\) put \(B=(d_6+2d_3+3d_2+6)/d_0\), \(\nu=6/d_0\), \(\eta=d_6/d_0\), and \(\ell=3\) for \(d_0=1\), \(\ell=\nu\) otherwise.
Let \(R\) count \(G(6,p,4)\), let \(\sigma\) be the contribution of product groups, and let \(s_0\) count \(d_0=1\) ends, where the curve has a semicubical cusp and saturated tangent bundle \(T_C(S_0)\).
We project the inverse differential to the product invariant and to the elementary symmetric invariant of greatest degree, obtaining simultaneous everywhere surjective quotients with \(h_P=-18t/p+R/p+\sum\eta_c\), \(h=-12t+R+\sum\nu_c\), and
\begin{equation}
	\label{ew-fields-six}
	\begin{gathered}
		2t=2g-2+s+5R/6+\sigma,\quad
		a=\sum B_c+(17/6+1/p)R+\sigma-(20+18/p)t,\\
		ph_P=h-6t,\qquad a=(2p+1)h_P+18t+2-2g-2s_0.
	\end{gathered}
\end{equation}
The other coefficient degrees for the first of these two targets are \(h_P+6t-R-\sum\ell_c\) and \(h_P-h\).

\label{ew-finite-six}
To compute \(B_c,\nu_c,\eta_c\), use \(S_1=\sum w_i^6\), \(S_2=\sum_{i<j}w_i^6w_j^6\), \(P=(w_1w_2w_3)^{6/p}\).
The full normalizer preserves their lines because coordinate axes contain more mirrors than other intersection lines.
Primitive alternatives are excluded by the 18 reflections of order 2 that exchange coordinates, compared with zero and nine in \(G_{25},G_{26}\).
The two theta projections onto \((dS_1,dS_2)\) have matrix \(\left(\begin{smallmatrix}1&1\\1&0\end{smallmatrix}\right)\); an equivariant change \(\xi_i=bw_i+O(w_i^7)\) preserves the required zeros.
At upper points, \(v=u^6\), \(dU_2=v\,dS_1+dS_2\), \(dU_3=v\,dS_2\), \(dU_0=u^{6/p}dP\).
A character \(\zeta^{-r}\) on a coarse frame contributes \(r/6\) after the allowed pole, giving the fractional \(\eta\) in \eqref{ew-fields-six}.

The cases \(p=1,2\) reduce to the following inequalities.
For \(p=1\), the pairs \((A_c,b_c)\), defined as in Section~\ref{ew-finite-cubic}, are \((4,1),(7,2),(10,3),(17,6)\), all widths are 1, and at \(2366\) the degree correction of the tangent cancels the removed determinant order.
Put \(j=6t\), \(A_{\rm cusp}=\sum_c(A_c-1-2b_c)\ge2\sum b_c/3\).
Then \(g=0\), \(h=-3j+\sum b_c+R<0\), \(a=A_{\rm cusp}+2h+R+2\), \(8h\le3a-6j-R-6\), hence \(11a\ge6j+R+6\), contradicting \(a\le j/2\).
For \(p=2\), \((B,\nu,\ell)=(14,6,3),(17/2,3,3),(6,2,2),(7/2,1,1)\).
Here \(g=0\), \(a=4h_P+9t+R/2+A_{{\rm cusp},2}+2\), \(A_{{\rm cusp},2}=\sum_c(B_c-1-2\nu_c)>0\), so \(h_P<-3t/2\), \(h=2h_P+6t\ge0\).
Both residual degrees are negative; the condition that the Hom space has dimension one gives \(a\ge-3h_P>9t/2\).

For \(p=6\), \(13h_P+11t+2g+5R/3+2\sigma\le2\) gives \(h_P<0\).
If \(g\ge1\), then \(\eta_c\le1\) and the logarithmic identity give \(h_P\le-t+2-2g-2R/3-\sigma<0\) and \(h=6h_P+6t\le0\).
The quotient onto both targets has negative determinant, so \cite[Proposition~2.4]{NakanoWeightedConfigurations2026} gives \(h_P=h=-a\), \(t=5a/6\), and \(g+s_0=1+a/2\).
Now \(\sum_c\nu_c\le3s+3s_0\) and \eqref{ew-fields-six} give \(5a/2+3R/2+3\sigma+9g\le9\), a contradiction.
Thus \(g=0\).
If \(h<0\), the quotient of rank 2 onto both targets gives \(h_P=h=-a\), \(t=5a/6\), \(s_0=1+a/2\).
The bound \(\sum\nu\le3s+3s_0\) gives \(5a/2+3R/2+3\sigma\le9\), hence \(a=2\), \(s_0=2\), impossible for a line or immersed conic.
If \(h\ge0\), put \(h_P=-b\).
Since \(h_P=-b<0\), \cite[Proposition~2.2]{NakanoWeightedConfigurations2026} gives coincident labels whose common relative tangent line has degree \(-e\le-b\); let \(\mathcal K\) be the kernel of the quotient of the saturated normal bundle onto the differences of their sections, as in \cite[Proposition~2.3]{NakanoWeightedConfigurations2026}.
The kernel \(\mathcal K\) and the negative residual degrees give \(h\ge b\), \(a=3b+2v\), \(0\le v\le h/4\), \(s_0=b+1+3h/2-v\).
The five remaining sections give a base point free map \(\gamma:\mathbb P^1\to\mathbb P^3\) of degree \(b+v\) whose differential vanishes at least \(s_0-v\) times.
If \(\gamma(\mathbb P^1)\) spans a linear space of dimension \(\rho\), the Wronskian gives \(\rho(s_0-v)\le(\rho+1)(b+v-\rho)\), \((3\rho-1)h/4\le b-\rho(\rho+2)\).
This contradicts \(h\ge b\) for \(\rho=2,3\); a line germ cannot be semicubical.

For \(p=3\), \(a\le26(1-g)-23R/3-12\sigma\) gives \(g=0\).
The identity \(a=6h_P+12t+R/3+\sum(\nu-2\mathbf1_{d_0=1})+2\) gives \(h_P<-3t/2\).
The residual degrees are \(-2R/3+\sum(\nu/3-\ell)<0\) and \(-2h_P-6t\le0\).
Equality in the second gives \(h_P=h=-a\), \(a=3t\), three coincident labels, and a kernel line of degree \(3a-2-s_0\) mapping to a target of degree \(a-3s_0\), hence \(a+s_0\le1\).
The remaining triples \((d_2,d_3,d_6)=(2,1,2),(1,3,2),(2,3,2)\) give the following pairs of full groups and cusp widths: \((G(6,3,3),1),(G(6,1,3),3),(G(6,1,3),3)\).
The first has no reflection of order 6; in the others the fiber character of such a reflection is nontrivial on the connected fixed torus.
\cite[Lemma~2.15]{NakanoWeightedConfigurations2026} excludes every actual interior reflection of order 6.
Thus \(a=1\), \(t=1/3\), \(s_0=0\), and the plane has full generic stabilizer \(G(6,3,2)\): Lemma~\ref{ew-fields-plane-bound} excludes ranks 0,1 and the other groups of rank 2 by their least invariant degree \(2\).

\subsection{The fixed surface and its full lattice}
Let \(Z\simeq\mathbb P^2\) have full generic stabilizer \(T=G(6,3,2)\).
The normal representation invariants have degrees 4,6; along the curve \(D_6\) in \(Z\) with full generic stabilizer \(G(6,3,3)\), their inverse coarse frames are \(u^{-2}e_4,v^{-1}e_6\), \(v=u^6\).
The projections of the inverse differential to these two invariants, extended across \(D_6\) in these frames, give the actual quotient line bundles \(Q_4,Q_6\) with \(Q_4^{\otimes3}=\mathcal L^{-12}|_Z\otimes\mathcal O_Z(D_6)\) and \(Q_6=\mathcal L^{-6}|_Z\otimes\mathcal O_Z(D_6)\), including finite characters because the punctured plane has unchanged Picard group and trivial fundamental group.
Hartogs extension of the nonzero inverse determinant gives \(Q_4=Q_6=\mathcal O_Z(-1)\), \(\lambda=1/3\), \(\deg D_6=1\), and orbifold Euler number \(1/3\).
Every cusp lies on \(D_6\): the subtorus fixed by \(T\) is connected, contains the origin, and has constant fiber character, so lifting the entire group gives \(G(6,3,3)\) and actual width 1 or 3.
The normalizer of \(T\) in \(G(6,3,3)\) preserves the line in the tangent space at the origin fixed by \(T\), and its effective action on this line has order 6: \(u\mapsto\zeta_6u\), with quotient coordinate \(v=u^6\).

\begin{lemma}
	\label{ew-fields-surface-lattice}
	The effective image of the full normalizer of \(Z\) is conjugate to the image \(\overline H_2\) in \(\mathrm{PU}(h_3)\) of the linear group \(H_2\) defined in Subsubsection~\ref{ew-finite-Htwo}.
	Its mirror of order 6 has signature \((0;6,6;1\text{ cusp})\) and coarse degree 1.
\end{lemma}
\begin{proof}
	\label{ew-finite-surface}
	Put \(\pi=1-\omega\), \(h_0=-\left(\begin{smallmatrix}0&0&1\\0&-1&0\\1&0&0\end{smallmatrix}\right)\).
	The four possibilities are the projective full unitary groups \(M_{\epsilon\eta}\) of \(L_{\epsilon\eta}=\mathcal O_K(2^\epsilon\pi^\eta)^{-1}e_1\oplus\mathcal O_Ke_2\oplus\mathcal O_Ke_3\), \(\epsilon,\eta\in\{0,1\}\)~\cite{StoverVolumes,PrasadYeungPlanes}.
	For \(M=M_{\epsilon\eta}\), let \(\Lambda_{\rm pr}\) be the principal arithmetic subgroup of \(\mathrm{SU}(h_0)\) normalized by the full inverse image of \(M\), and let \(\overline\Lambda_{\rm pr}\) be its projective image.
	Put \(n=[M:\overline\Lambda_{\rm pr}]\), \(j=[M:\Lambda]\) for this effective image \(\Lambda\subset M\), \(\mathcal I=\{p:K_p\text{ is split Iwahori}\}\), and \(A=\prod_p\lambda_p\), where \(K_p\) are the defining parahoric subgroups and \(\lambda_p\) are the factors in \cite[Theorem 4.1]{StoverVolumes}.
	Since \(\mu_3\subset\Lambda_{\rm pr}\), the group \(\overline\Lambda_{\rm pr}\) has Euler number \(A/24\), giving \(jA=8n\).
	By the bound \(n\le3^{1+|\mathcal I|}\), the split Iwahori factors \(\ge456\), and the inert factors \(p^2-p+1=3,21,\ge111\) at \(2,5,\ge11\), we have \(\mathcal I=\varnothing\), \(n\le3\), and \(A\le24\); only the vertices at 3 and the choices at 2 remain.
	The determinant of \(\operatorname{diag}(1,\omega,1)\) is not a cube, giving \(n=3\).
	The Euler numbers are \(1/72,1/72,1/24,1/24\), hence \(j=24,24,8,8\).
	Each group has one cusp since the class number and its 3 torsion are 1 and \(\mathcal I=\varnothing\)~\cite[Theorem 4.8]{StoverCusps}.

	The subgroup calculation uses the presentations and cusp generators in Subsection~\ref{ew-eisenstein-presentations}.
	Its inputs are the four indices \(24,24,8,8\), the prescribed cusp reflection, and the exclusion of isolated stabilizers not generated by reflections.
	Table~\ref{ew-table-eisenstein-actions} gives every case passing these tests and its final obstruction; paired orders refer to a point stabilizer and its reflection subgroup.
	\begin{table}[htbp]
		\centering
		\begin{tabular}{c|c|l}
			group      & cases & exclusion or surviving case                                      \\\hline
			\(M_{00}\) & 4     & Point order 36; orders \((12,4),(6,1)\); \(\overline H_2\).      \\
			\(M_{01}\) & 4     & Mirror degrees \(1/2,3/2,3/2,3/2\).                              \\
			\(M_{10}\) & 2     & Mirror degree 2; point order 36 on \(\operatorname{diag}(1,2)\). \\
			\(M_{11}\) & 2     & Mirror degrees \(2,1/2\).
		\end{tabular}
		\caption{Outcomes for the four Eisenstein lattices.}
		\label{ew-table-eisenstein-actions}
	\end{table}
	The search consists of transitive permutation actions of the prescribed degrees, so every generator has finitely many possible images.
	The complete relators, all extensions through the common subgroup, and all cusp orbits are tested; the construction in Subsection~\ref{ew-eisenstein-presentations} proves exhaustiveness.
	The remaining tests concern actual point stabilizers and mirror degrees, as follows.

	Here \(J,R,\sigma\) are the matrices of the presentation of \(M_{00}\) in Subsection~\ref{ew-eisenstein-presentations}, where \(J=RP\) and \(\sigma=QP^{-1}\).
	To compute full stabilizers, put \(D=\operatorname{diag}(1,1,-1)\), \(\sigma_+=D^{-1}\sigma D\), \(J_+=D^{-1}JD\), \(R_+=D^{-1}RD\), \(r_2=J_+\sigma_+J_+^{-1}\), \(r_3=J_+^{-1}\sigma_+J_+\), \(b=r_2r_3r_2^{-1}\), \(d=r_2J_+^2\).
	The full stabilizers of \([-1,0,1]\), \([\omega^2,0,1]\) are \(\langle\sigma_+,R_+\rangle\), \(\langle\sigma_+,b,d\rangle\), of orders 12,72.
	Enumerate their projective matrices fixing each coset and test local reflections by \(\operatorname{rank}(\gamma-\alpha I)=1\), \(\gamma z=\alpha z\).
	For mirror orbit length \(\ell\) and pointwise kernel of order \(k\mid6\), the effective image of the stabilizer of a mirror has index \(\ell k/6\) in that of its full normalizer.
	The full normalizer \(\langle\sigma_+,R_+,b\rangle\) of a mirror with root of norm 1 has pointwise kernel \(\mu_6\) and effective image \(\Delta(2,6,\infty)\)~\cite[Section 3]{PaupertWells}.
	The full mirror stabilizers in \(M_{01},M_{11}\) have pointwise kernel \(\mu_6\), and their effective images contain \(\mathrm{PSL}(2,\mathbb Z)\).
	The trace condition permits only effective elliptic orders \(2,3,6\) before the group is determined, since order 4 fails the valuation at 2 of \((\operatorname{tr}\gamma)^2=2\det\gamma\).
	Let \(C^\circ\) be the quotient of the mirror by this effective image, an orbifold curve with elliptic orders \(e_z\), and let \(C\) be its smooth compactification by \(s\) cusps; thus \(\chi_{\mathrm{orb}}(C^\circ)=2-2g(C)-s-\sum_z(1-e_z^{-1})\), and hence \(-\chi_{\mathrm{orb}}(C^\circ)\ge1/6\).
	Equality follows from the contained modular group, so the effective images are \(\mathrm{PSL}(2,\mathbb Z)\), of signature \((0;2,3;1\text{ cusp})\).
	Thus \(-\chi_{\mathrm{orb}}(C^\circ)=2L\cdot C=2\deg C/3\) gives degrees \(\ell k/12,\ell k/24\) respectively.

	The remaining point stabilizers in \(M_{00}\) violate the required order 72 or the Chevalley--Shephard--Todd theorem.
	For the binary lattice, \(|z_1|^2+2|z_2|^2=1\) gives the full group \(\mu_6^2\), excluding the required effective \(G(6,1,2)\) of order 72.
	All fractional or larger degrees contradict \(\deg D_6=1\).
	For \(\overline H_2\), the effective image of the mirror stabilizer has index 2 in \(\Delta(2,6,\infty)\), proving its signature and degree.
\end{proof}

\subsubsection{The linear group and its preserved lattice}
\label{ew-finite-Htwo}
Use negative lines, \(h_3=\left(\begin{smallmatrix}0&0&1\\0&1&0\\1&0&0\end{smallmatrix}\right)\), and
\[
	J=\begin{pmatrix}0&0&1\\0&-\omega&\omega\\1&1&\omega\end{pmatrix},\quad
	A=\operatorname{diag}(1,-\omega^2,1),\quad C=\omega I,\quad S=JA^{-1}J.
\]
The group \(\langle J,A,C\rangle\) has the presentation \(J^3=C\), \(A^6=1\), \(S^4=C^2\), \([A,S^2]=1\), \(C^3=[J,C]=[A,C]=1\).
The subgroup \(H_2\) is the inverse image of the stabilizer of 1 in the right action defined by the following disjoint cycles, with \(C\) acting trivially:
\begin{equation}\label{ew:finite-permutation-action}
	\begin{aligned}
		J={}&(1,2,3)(4,6,7)(5,9,10)(8,16,12)\\
		&\quad(11,20,21)(13,22,14)(15,23,17)(18,24,19),\\
		A={}&(2,3,5,11,8,4)(6,12,13)(7,14,18,9,17,15)\\
		&\quad(10,19,20)(16,22)(21,24).
	\end{aligned}
\end{equation}
The Schreier generators, the abelianization, and the character \(\chi\) below are computed on the linear group \(H_2\), not on its image \(\overline H_2\).
Choose the first breadth first word \(t_i\) from 1 to \(i\), in alphabet order \(J,A,J^{-1},A^{-1}\).
The 72 edges are \(e_{i,g}=t_i g t_{ig}^{-1}\) for \(g=J,A,C\); inverse edges are \(e_{i,g^{-1}}=e_{ig^{-1},g}^{-1}\).
The 23 tree edges and the seven relations above, rewritten at all 24 points, give 191 relations, of full rank with nonunit Smith entries \(6,6\), hence abelianization \((\mathbb Z/6\mathbb Z)^{\oplus2}\).
Write \(\chi(e_{i,g})=\omega^{x_{i,g}}\) and impose \(x_{i,C}=1\).
We require the character to take value 1 on reflections of orders 6 and 2 with roots of norm 1, value \(\omega^2\) on reflections of order 3 with roots of norm 1, and value 1 on reflections of order 2 with roots of norm 2.
At the two points with full stabilizers of order 72, we require agreement with the normal character in degree 4.
To reproduce these stabilizers, put \(r=C^{-1}S^2\), \(b=(JAJ^{-1})(J^{-1}AJ)(JA^{-1}J^{-1})\).
Take all orbits of \(\langle A,r,b\rangle\) and \(\langle A,S\rangle\) in the action defined by \eqref{ew:finite-permutation-action} and rewrite the conjugates of the appropriate powers of reflections from an orbit representative.
At base points 1 and 23 impose exponents \(0,0,1\) on \(A\), \(JAJ^{-2}AJ^2A^{-1}J^{-1}\), and \(JAJ\), conjugated by \(t_i\).
The 191 relations have rank \(70\) over \(\mathbb F_3\); imposing \(\chi(C)=\omega\) gives rank \(71\), and \(\chi(A)=1\) determines a unique character.
Its values \(\chi(e_{i,g})=\omega^{x_{i,g}}\) are \(x_{i,C}=1\) for every \(i\), and
\[
	\begin{aligned}
		x_{i,J}=1&\quad\text{for }i\in\{2,6,9,12,17,20,22,24\},\\
		x_{i,A}=1&\quad\text{for }i\in\{2,16,17,20,22\},\\
		x_{i,A}=2&\quad\text{for }i\in\{11,14,19,24\},
	\end{aligned}
\]
with all remaining exponents zero.
Every stated local value is verified by substitution in the rewritten relators and the preceding words.
Evaluate any word by summing its Schreier exponents.

On the fixed space \(W\) and its positive normal space, the actual full normalizer is
\[
	\begin{gathered}
		\rho(M)=\operatorname{diag}(\chi(M)^{-1}M,(\det M)^{-1},1),\\
		\iota(t)=(\det t)\operatorname{diag}(I_3,t),\quad t\in G(6,3,2).
	\end{gathered}
\]
Every block differs from \(\rho(M)\) by \(\iota(T)\), since the determinant condition first gives a scalar in \(\mu_6\) and the character conditions reduce it to \(\chi(M)^{-1}\).
Both elliptic points of order 6 of the mirror therefore have actual ambient group \(G(6,3,4)\).
\cite[Theorem~3.6]{KrishnasamyTaylor} excludes primitive larger groups.
A product would require a reflection of order 6, already excluded.
We choose coordinates in which the full stabilizer of one such point is monomial, with no condition at the second point.
Put \(M_1=A\), \(M_2=JAJ\), \(M_3=JA^{-1}JAJA^2J^{-1}\).
Their traces are \(\omega+3,-\omega,-\omega\).
The commutant of \(M_1,M_2,M_3\) is \(KI\), and the corresponding intertwiner equations have dimension 1 over \(K\).
The basis \((\omega^2,0,1),(0,1,0),(-\omega,0,1)\) has Gram matrix \(\operatorname{diag}(-1,1,1)\).
For the first full \(G(6,3,4)\) group, let \(P_i\) project to a positive axis and put \(F_i=-I+2P_i\).
With \(\zeta=1+\omega\), let \(D_{ij}\) act by \(\zeta,\zeta^{-1}\) on the \(i,j\) axes and by 1 on the other axes.
The identity
\[
	P_i=\omega^2(D_{ij}-I)(D_{ij}-\zeta^{-1}I)-\omega(F_i+I)
\]
shows that \(P_i\) lies in the integral group algebra.
For \(\Omega=\left(\begin{smallmatrix}\omega^2&0&-\omega\\0&1&0\\1&0&1\end{smallmatrix}\right)\), the ideals generated by entries \(12,13\) and by entries \(21,31\) of the 72 matrices \(\Omega^{-1}e_{i,g}\Omega\) are both \(\delta\mathcal O_K\).
For the original group, both ideals equal \(\delta\mathcal O_K\) by the divisibility of all generators and the entries \(\omega-1,-2(\omega-1),-\delta,2\delta\) of \(\Omega^{-1}M_3\Omega\) in positions \(12,13,21,31\).
The trace and commutator calculation above gives a conjugacy over \(K\), and actual monomial elements exchanging the normal axes give the form \(h_5=\operatorname{diag}(-1,1,1,1,1)\) in the standard basis \(e_0,\ldots,e_4\) of \(K^5\).
The normalized lift of \(\Gamma\) preserves an \(\mathcal O_K\) lattice by \cite[Proposition~2.11]{NakanoWeightedConfigurations2026}.
Integral projectors from the first \(G(6,3,4)\) group split this lattice as \(d\mathcal O_K\oplus\mathcal O_K^4\).
For the matrices of the elements of this lift with respect to this splitting, the entries in each of the two off diagonal blocks generate \(\delta\mathcal O_K\), so \(\delta\mathcal O_K\subset d\mathcal O_K\subset\delta^{-1}\mathcal O_K\).
The three possible lattices are exchanged by duality; the unique intermediate module in \(\mathcal O_K/\delta^2\mathcal O_K\) is the middle lattice.
Thus every isometry of any of them preserves the middle lattice.
Since \(\zeta\mapsto\zeta^5\) permutes \(\mathcal O_K^\times=\mu_6\), \(\Gamma\) lifts uniquely to
\[
	\widetilde\Gamma\subset\mathrm{SU}(h_5,\mathcal O_K),
\]
and the matrices below are taken in \(\widetilde\Gamma\).
\subsubsection{An actual reflection of order 6}
\label{ew-finite-cusp-six}
For a vector \(p\) positive with respect to the form \(h_5\) fixed above, put
\[
	P_p=\frac{pp^*h_5}{p^*h_5p},\qquad B_p=-I+2P_p.
\]
The following basis is integral and has determinant 1:
\[
	\begin{gathered}
		v=(\delta,1,1,1,0),\quad q_1=(1,0,\omega^2,-\omega,0),\\
		q_2=(1,0,-\omega,\omega^2,0),\quad q_3=e_4,\quad t_0=e_1+\omega v.
	\end{gathered}
\]
Write \(Q_v=(v,q_1,q_2,q_3,t_0)\), \(B=\operatorname{diag}(\Omega,I_2)\), \(s=(1-\omega,2\omega^2,-1-\omega,0,0)\), and
\[
	\begin{gathered}
		G=B^{-1}\operatorname{diag}(\omega^{-1}M_3,(\det M_3)^{-1},1)B,\\
		T_n=\operatorname{diag}(I_3,-I_2),\quad F_0=\operatorname{diag}(1,V_0),\\
		V_0=\begin{pmatrix}0&\omega&0&0\\-\omega&0&0&0\\0&0&\omega&0\\0&0&0&1\end{pmatrix}.
	\end{gathered}
\]
Then \(B_s=F_0T_nGF_0^{-1}\) belongs to \(\widetilde\Gamma\).
Let \(F=\operatorname{diag}(1,\left(\begin{smallmatrix}0&1&0\\0&0&1\\1&0&0\end{smallmatrix}\right),1)\), and let \(\Psi\) be minus the interchange of the first two positive axes.
For \(U_1=FB_s\Psi^{-1}\) and \(U_2=FU_1F^{-1}\), the matrices in the \(Q_v\) basis are \(N((-1,1,0),-1)\) and \(N((-\omega,-1-\omega,0),-1)\), where
\[
	N(b,u)=\begin{pmatrix}1&-\Transpose{\overline b}&u\\0&I_3&b\\0&0&1\end{pmatrix}.
\]
They fix \(q_3\) and change \(q_1+av\) to \(q_1+(a+1)v\) and \(q_1+(a+\omega^2)v\), respectively.
Thus they remove any \(a\in\mathcal O_K\).
Every full cusp element acts monomially on the positive definite unimodular horizontal lattice; its scalar \(\xi\), the eigenvalue on \(v\), is in \(\mu_6\), and a trivial horizontal action would give \(\xi^5=1\), hence \(\xi=1\).
If that full action preserved \(q_3\), it would split the horizontal representation as \(2+1\), giving a positive second Betti number by the circle bundle Gysin sequence.
Thus some actual cusp element moves \(q_3\) to \(u(q_i+av)\), \(i=1,2\).
We obtain \(B_{q_1}\in\widetilde\Gamma\) from conjugation of \(B_{q_3}\), the equality \(\Psi q_2=(1+\omega)q_1+\omega v\), and \(U_1,U_2\).
For \(r_0=(1,1,1,0,0)\), the actual element
\[
	F'=-\operatorname{diag}\left(1,\begin{pmatrix}0&\omega&0&0\\0&0&-\omega^2&0\\1&0&0&0\\0&0&0&1\end{pmatrix}\right)
\]
satisfies \(F'q_1=-r_0\).
With \(D=\operatorname{diag}(1,1,\omega,\omega^2,1)\) and \(u_0=(2\omega,2\omega,\omega^2,0,0)\), the identity \(P_{DB_{r_0}e_2}=P_{u_0}\) gives \(B_{u_0}\in\widetilde\Gamma\).
Finally put \(w_3=JA^{-2}J\), \(M'_3=w_3A^2w_3^{-1}\); the character above gives \(\chi(M'_3)=\omega^2\).
The actual element
\begin{equation}\label{ew:order-six-T-three}
	T_3=B^{-1}\operatorname{diag}(I_3,\omega^2,\omega)
	\operatorname{diag}(\omega M'_3,(\det M'_3)^{-1},1)B
	=\omega I+(\omega^2-\omega)P_{u_0}
\end{equation}
multiplied by \(B_{u_0}\) gives an element whose eigenvalue on \(u_0\) divided by its eigenvalue on \(u_0^\perp\) is \(-\omega\).
All matrix identities in Subsubsection~\ref{ew-finite-cusp-six}, including \eqref{ew:order-six-T-three}, are exact over \(K\).

This actual reflection contradicts its prior exclusion and eliminates the three remaining triples \((d_2,d_3,d_6)\).

\subsection{The groups \texorpdfstring{\(G(3,1,3)\), \(G(3,3,3)\), \(A_3\), \(B_3\)}{G(3,1,3), G(3,3,3), A3, B3} and the last cusps}
Only \(G(3,1,3),G(3,3,3),A_3,B_3\) remain.
Their normal representation invariants give \(\lambda\ge1/2\) by Lemma~\ref{ew-fields-plane-bound}, hence \(a\le2t\) by \cite[Proposition~2.6]{NakanoWeightedConfigurations2026}.
For \(G(3,1,3)\), the generator degrees at cusps are \(1111,1133,1333\), all widths are 1, their sums are \(4,8,10\), and the inverse row associated with the invariant of greatest degree has pole orders \(1,3,3\).
For a cusp \(c\), let \(A_c\) be the sum of generator degrees and \(b_c\) the pole order of the inverse row associated with the invariant of greatest degree.
Put \(j=3t\), let \(R\) count \(G(3,1,4)\), \(B=\sum b_c\), and \(A_{\rm cusp}=\sum_c(A_c-1-2b_c)\).
The comparison gives \(g=0\), \(h=-3j+R+B\), \(a=A_{\rm cusp}+2h+R+2\), \(A_{\rm cusp}\ge B/3\), \((7/3)B\le7j-3R-2\).
The lower residual degrees are negative, the first being \(-2j+B-A_{\rm cusp}\); thus the target associated with the invariant of greatest degree has \(\dim\operatorname{Hom}(\mathcal N,Q)=1\) and \(a\ge-3h\).
\label{ew-finite-cubic}
The finite conditions and their complete outputs are in Table~\ref{ew-table-gthreeonethree}; all counts and \(j\) are integers, and outputs are ordered as \((j,R,s_1,s_2,s_3,\sigma;a,h)\).
\begin{table}[htbp]
	\centering
	\begin{tabular}{p{.57\linewidth}|p{.35\linewidth}}
		\text{inputs and equations}                 & \text{outputs} \\\hline
		\(1\le j\le28\), \(0\le R\le6\), \(0\le s\le20\), \(\sigma=n_2/2+2n_3/3+5n_6/6\);
		\(2j/3=-2+s+2R/3+\sigma\),
		\(a=\sum A_c+11R/3+\sigma-20j/3\),
		\(h=-3j+R+\sum b_c\);
		\(1\le a\le j\), \(-a\le h<0\), \(a\ge-3h\) &
		\((6,0,1,4,1,0;6,-2)\),
		\((7,1,0,5,1,0;7,-2)\).
	\end{tabular}
	\caption{Finite conditions and outputs for fixed curves with full generic stabilizer \(G(3,1,3)\).}
	\label{ew-table-gthreeonethree}
\end{table}
Both have \(a=3t\), contradicting \(a\le2t\).
\cite[Lemma~2.15]{NakanoWeightedConfigurations2026} now excludes actual reflections of order 3; \(\mathcal L^{\otimes2}\) descends, and the least invariant degree \(2\) of the remaining groups gives \(\lambda\ge1\), \(a\le t\).

For \(T=G(3,3,3)\), let \(R\) count \(G(3,3,4)\).
Its degrees are \(3,3,6\); the full normalizer may mix the first two directions, and every cusp is \(1113\) of width 1.
Write \(P_3=w_1w_2w_3\), \(S_1=\sum w_i^3\), \(S_2=\sum_{i<j}w_i^3w_j^3\).
At upper points \(v=u^3\), \(dU_0=u\,dP_3\), \(dU_2=v\,dS_1+dS_2\), \(dU_3=v\,dS_2\).
At cusps the theta projections are \(q\,dP_3\), \(-q\,dS_1/3\), \(q^3(dS_1+dS_2)\).
The projection of the inverse differential to the coordinate of greatest invariant degree is therefore \(3q^{-1}dU_2+q^{-3}dU_3\), including its lower term, and gives an everywhere surjective quotient \(\mathcal N\to Q_6\) after pole 3.
For \(\mathcal K_3=\ker(\mathcal N\to Q_6)\), the directions of degree 3 have pole at most 1 on a finite cover trivializing their representation.
Consequently every line subbundle \(J\subset\mathcal K_3\) satisfies \(\deg J\le b=-3t+R+s\).
The degrees of \(\mathcal K_3\) and \(Q_6\) are \(\deg\mathcal K_3=2b-2R/3\) and \(h=-6t+R+3s\).
\cite[Proposition~3.4]{NakanoWeightedConfigurations2026} gives \(2t=2g-2+s+2R/3+\sigma\) and \(a=6s+3R+\sigma-14t\).
These force \(g=0\).
With \(j=6t\), the exhaustive conditions and remaining cases are in Table~\ref{ew-table-gthreethreethree}.
\begin{table}[htbp]
	\centering
	\begin{tabular}{p{.57\linewidth}|p{.35\linewidth}}
		\text{inputs and equations}        & \text{nine outputs} \\\hline
		\(1\le j\le33\), \(0\le R\le7\), \(0\le s\le13\), \(j/3\in\mathbb Z\);
		\(\sigma=j/3+2-s-2R/3\),
		\(a=6s+3R+\sigma-7j/3\),
		\(h=-j+R+3s\), \(b=-j/2+R+s\);
		\(a\in\mathbb Z\), \(\sigma\in\frac12\mathbb Z_{\ge0}\), \(1\le a\le j/6\);
		for \(R=0\) also \(b\in\mathbb Z\) &
		Three with \(R=0\); two with \(R=3\), integral \(t\);
		\((a,t,b)=(1,3/2,1/2)\);
		three with \((a,t)=(2,2),(1,5/2),(2,7/2)\).
	\end{tabular}
	\caption{Finite conditions and outputs for fixed curves with full generic stabilizer \(G(3,3,3)\).}
	\label{ew-table-gthreethreethree}
\end{table}
For \(R=0\), the unramified representation is projectively trivial on \(\mathbb P^1\), so \(\mathcal K_3=\mathcal O(b)^{\oplus2}\); the three negative rows contradict \cite[Proposition~2.4]{NakanoWeightedConfigurations2026}.
The row \((1,3/2,1/2)\) has a positive normal line of degree 1 in \(\mathcal K_3\).
There is no row with \(s=0\): for \(R=0,3,6\), \([(7R/3+2)/13,(7R/3+2)/12)\) contains no half integer.

The full projective normalizer has order 216, generated by \(T\), \(\operatorname{diag}(\omega,1,1)\), and \(\delta^{-1}\left(\begin{smallmatrix}1&1&1\\1&\omega&\omega^2\\1&\omega^2&\omega\end{smallmatrix}\right)\).
Images of a four point frame in the nine mirror lines prove completeness.
The linear lift generated by \(T\) and these matrices acts on the directions of degree 3 through the group of order 24 generated by \(D_2=\operatorname{diag}(\omega,1)\), \(F_2=\delta^{-3}\left(\begin{smallmatrix}-3&1\\18&3\end{smallmatrix}\right)\); this lift contains the central element \(-I_3\), which acts by \(F_2^2=-I\), so the action is not a representation of the projective normalizer.
Direct multiplication shows that a triple of the four conjugates of \(D_2^2\) has product \(I\) only when its factors agree.
For integral \(t\), \(\mathcal O(t)\) cancels the actual central \(\mu_2\) character of \(\mathcal L^{\otimes2}\); no such cancellation is asserted for half integral \(t\).
Thus for the \(R=3\) integral rows, \(\mathcal K_3\) has a line subbundle of degree \(-3t+R/3+s\).
In one row this degree is less than \(-a\); in the other, the positive line of degree \(3a-2=7\) lies in \(\mathcal K_3\) with \(b=0\).

For the last three rows, every immersion of degree 1 or 2 with respect to \(H\) is a line or conic in a coincidence plane, or moves one label with six fixed.
The splitting \(V=\mathcal O\oplus\mathcal O(k)\) gives \(a=(7-2m)k+2\sum b_i\), \(k,b_i\ge0\).
A conic's positive normal line of degree 4 maps trivially to \(Q_6\) and exceeds the bound \(\deg J\le b\).
At a \(1113\) cusp the horizontal torus and its nine directions are
\[
	A=E^3/K_0,\quad E:\ x_1^3-x_2^3-x_3^3=0,\quad
	K_0=\{(k_i)\in E[1-\omega]^3:\sum k_i=0\},
\]
\[
	(0,\omega^j,-\omega^{-j}),\quad(1,0,\omega^j),\quad(1,\omega^j,0),\qquad j=0,1,2.
\]
They are the values of the three theta sections of degree 1 at the fixed points represented by \((u+k_i)\), \(u\in E[3]/E[1-\omega]\), \(\sum k_i\in E[1-\omega]\).
At least two values are nonzero, proving trivial fiber character and width 1.
The directions lie on a smooth cubic, so at most three lie on a line.
With zero, one, or two triple points, there are at most \(7,3,0\) complete curves moving one label and \(0,1,2\) coincidence planes.
These account for at most \(7,6,6\) directions, a contradiction.
Thus no cusp has full horizontal group \(G(3,3,3)\), and no point has full stabilizer \(G(3,3,4)\).

Only Coxeter horizontal groups remain.
Their normal characters may have an unramified \(\mu_3\) factor of degree 0, but at a cusp, a cusp element of the normalized special unitary lift has diagonal blocks \(b,bF_{\rm hor},b\) in a Siegel basis, with \(b\in\mu_6\) and \(\det F_{\rm hor}=\pm1\), so \(1=b^5\det F_{\rm hor}\) gives \(b=\pm1\).
There is no extra cusp inertia, so Lemma~\ref{ew-fields-coxeter} applies and proves Proposition~\ref{ew-seven-fields}.
\section{Seven points over the Gaussian field}
\label{ew-seven-gaussian-section}

In this section, \(X=X_7\), \(H=-K_X\), and \(K=\mathbb Q(i)\).
We assume that \(X=(\mathbb B^4/\Gamma)^{\mathrm{BB}}\) for a nonuniform arithmetic lattice.
All fixed quotients are formed with the full normalizer in this same group \(\Gamma\).

\subsection{Reflections of order four}

\begin{lemma}
	\label{ew-gaussian-g413}
	No smooth cusp has full horizontal group \(G(4,1,3)\).
\end{lemma}

\begin{proof}
	For the full fixed curve \cite[Lemma~2.17]{NakanoWeightedConfigurations2026}, put \(a=\deg f^*H,t=\deg f^*\mathcal L,j=4t\in\mathbb Z_{>0}\); Lemma~\ref{ew-fields-plane-bound} gives \(a\le j/2\).
	For \(T=G(4,1,3)\), the upper groups \(G(4,1,4),T\times\mu_2,T\times\mu_4\) \cite[Theorems 2.1, 3.6]{KrishnasamyTaylor} have counts \(R,n_2,n_4\), so \(\sigma=n_2/2+3n_4/4\); \(T\) has normalizer \(\mathbb C^*T\), preserving its three invariant directions of degrees \(4,8,12\) and their characters.
	The generator degrees and the pole orders \(\nu_i\) of the inverse rows are \cite[Theorem 3.6 and Example 3.10]{Rains}
	\[
		\begin{array}{c|ccccc}
			\text{degrees}&1111&1122&1222&1244&1444\\\hline
			(\nu_1,\nu_2,\nu_3)&(1,1,1)&(1,2,2)&(2,2,2)&(2,4,4)&(4,4,4).
		\end{array}
	\]
	The cusp has width \(1\), measured using the least positive central period of the full cusp group.
	Let \(\theta_i\) be the generators of the ring of invariant theta sections and \(b_i\) their degrees; at the horizontal point \(z_0\) corresponding to the end of the disk fixed by \(T\), an invariant theta section of degree \(1\) is nonzero, so the width \(\kappa=\gcd\{b_i:\theta_i(z_0)\ne0\}\) of \cite[Lemma~3.8]{NakanoWeightedConfigurations2026} equals \(1\), and the coarse coordinate \(Q\) of the curve at this end equals the coordinate \(q\) of the cusp.
	For the raw tuple \(2244\), we remove exactly the scalars \(\mu_2\) in the kernel of the fiber character.
	For \(X_0=\prod x_{1j},Y=\prod y_j\) on \(E:x_1^4-x_2^4=y^2\), \(d(Y/X_0^2)=\eta(-df_1+df_2-df_3)/2\), \(\eta\ne0\), on both full orbits; translation \((ix_1,-ix_2,y)\) exchanges \([1:0:\pm1]\), and the unique lift fixing the point retains full stabilizers.
	Put \(\nu_1^{\rm tot}=\sum_c\nu_{1,c}\), \(\nu_3^{\rm tot}=\sum_c\nu_{3,c}\), and \(h=-3j+\nu_3^{\rm tot}+R\); then \(\nu_1^{\rm tot}\ge \nu_3^{\rm tot}/2\), and at a point of \(D_R\) with local parameter \(v\), the extension equations of~\cite[Lemma~3.6]{NakanoWeightedConfigurations2026} read \((b_1,b_2,b_3)=(vc_1,c_1+vc_2,c_2+vc_3)\).
	With \(A_c\le13\), \cite[Proposition~3.4]{NakanoWeightedConfigurations2026} gives \(2t=2g-2+s+3R/4+\sigma\) and
	\[
		a=\sum A_c+15R/4+\sigma-26t\le26(1-g)-6R-12\sigma.
	\]
	Consequently \(C=\mathbb P^1\), \(a=\nu_1^{\rm tot}+2h+R+2\), and \(5h\le2a-3j-R-4<0\).
	The first residual degree satisfies \(-2j+\nu_3^{\rm tot}-\nu_1^{\rm tot}\le(a-4j-3R-2)/5<0\).
	The row associated with the invariant of greatest degree and \cite[Lemma~3.5]{NakanoWeightedConfigurations2026} give \(\dim\operatorname{Hom}(f^*T_X,\mathcal O_C(h))=1\); \cite[Proposition~2.2]{NakanoWeightedConfigurations2026} shows that two labels coincide and that \(a\ge-3h\), hence \(11a\ge9j+3R+12\), a contradiction.
\end{proof}

\begin{proposition}
	\label{ew-seven-orderfour}
	Every actual reflection in \(\Gamma\) has order \(2\).
	The square of \(\mathcal L\) descends to an ample line bundle \(D\) on \(X\).
	For every coincidence plane \(Z_I\) with \(|I|=3\), write \(L|_{Z_I}=\lambda_I c_1(\mathcal O_{Z_I}(1))\).
	Then \(2\lambda_I\) is a positive integer, \(\lambda_I\ge1\), and \(L-H\) is nef.
\end{proposition}

\begin{proof}
	Meyer's theorem in rational dimension \(8\) gives a cusp of any actual order \(4\) mirror in the original lattice, with a faithful horizontal reflection of order \(4\).
	The character fields \cite[Table 1]{GeckMalle} and the trace \(\operatorname{tr}(\operatorname{diag}(\zeta_m,\zeta_m^{-1},1)(23))=\zeta_m\) leave \(m=1,2,4\) and the full horizontal groups
	\[
		A_3,\quad B_3,\quad G(4,4,3),\quad G(4,2,3),\quad G(4,1,3).
	\]
	The last group alone has an order \(4\) reflection and is excluded by Lemma~\ref{ew-gaussian-g413}.
	The normalized reflections generate the lift \(G\subset\mathrm U(h,K)\) of the same \(\Gamma\); the determinant \(\pm1\) of the elements of \(G\) forces every scalar \(\zeta I_5\) to have \(\zeta=\pm1\), so \(\mathcal L^{\otimes2}\) descends to \(D\) by \cite[Proposition~2.11]{NakanoWeightedConfigurations2026}, with the full canonical frame at the cusp.
	For planes of ranks \(0,1\), Lemma~\ref{ew-fields-plane-bound} gives \(\lambda_I\ge1,2/3\), respectively; the least invariant degree in rank \(2\) is two, giving \(2\lambda_I>1\).
	For \(G(4,2,2)\), the invariants \((v+S,vS+P^2,u^2P)\), \(v=u^4\), give the frame \((e_S,u^2e_P)\), the correction \(-1/2\), and the determinant order \(2\), hence \(8\lambda_I=2+3d_C/2\).
	The value \(\lambda_I=1/2\) would give \(d_C=4/3\), so half integrality gives \(\lambda_I\ge1\) for every plane; \cite[Proposition~2.6]{NakanoWeightedConfigurations2026} applied to all \(35\) planes gives \(D-2H\) nef.
\end{proof}

\subsection{A lattice preserved by the original group}

\begin{proposition}
	\label{ew-seven-unimodular}
	After a Hermitian similarity over \(K\), the original projective lattice \(\Gamma\) is the projective image of a subgroup of finite index of \(\mathrm{SU}(I_{1,4},\mathbb Z[i])\), where \(I_{1,4}=\operatorname{diag}(-1,1,1,1,1)\).
	The similarity preserves the automorphic bundle, the full finite stabilizers, the full normalizers, and the indices of all central cusp subgroups.
\end{proposition}

\begin{proof}
	Let \(A\) be the Gram matrix, conjugate linear in the first variable, and \(G\) the original lift generated by the normalized reflections, preserving a Gaussian lattice.
	Over \(K\), irreducibility and Schur identify each full horizontal form with \(aI_3\), \(a>0\) rational, with root norms \(1,2\); its determinant times a root norm is a norm because \(a^4=\mathrm N(a^2)\).
	Meyer's theorem in dimension \(8\) gives a cusp of each actual mirror, and for a root \(r\) the Witt decomposition gives
	\[
		\det h=-\mathrm N(\alpha)\det h_{\rm hor},\qquad
		h(r,r)=h_{\rm hor}(\bar r,\bar r).
	\]
	Thus \((-\det A)h(r,r)\in\mathrm N(K^*)\); set \(c=-\det A,h'=ch\), with \(\det h'=-c^6\) and even valuations of root norms at odd inert primes.
	At odd inert \(p\), maximize \(G\)-stable integral \(M\); enlargement by \(p^{-1}(M\cap p^2M^\#)\) forces valuations \(0,1\), \(pM^\#\subset M\), and even dimensional discriminant space.
	For a primitive root of norm \(n\), integrality gives \(h'(r,M)\subset n\mathbb Z_p[i]\), \(r/n\in M^\#\); primitivity forces \(v_p(n)\le1\), so its even parity gives \(v_p(n)=0\).
	Thus \((R_r-I)M^\#\subset M\), and the same \(G\) acts trivially on the discriminant space.
	A nonzero discriminant space has an isotropic line enlarging \(M\), so \(M=M^\#\).
	At \(2\), put \(\pi_2=1+i\) and maximize a \(G\)-stable integral lattice \(M\).
	For a primitive reflection root \(r\in M\) of norm \(n\), the equality \(R_r(M)=M\) gives \(h'(r,M)\subset(n/2)\mathbb Z_2[i]\).
	If \(v_2(n)\ge2\), put \(\Delta=\sum_{g\in G}\mathbb Z_2[i]gr\) and \(\widetilde M=M+\pi_2^{-1}\Delta\).
	The inclusions
	\[
		h'(M,\pi_2^{-1}\Delta)\subset\frac{n}{2\pi_2}\mathbb Z_2[i],
		\qquad
		h'(\pi_2^{-1}\Delta,\pi_2^{-1}\Delta)\subset\frac n4\mathbb Z_2[i]
	\]
	show that \(\widetilde M\) is integral and \(G\)-stable; primitivity of \(r\) makes it strictly larger than \(M\), a contradiction.
	Hence \(v_2(n)\le1\), so \((R_r-I)M^\#\subset M\).
	A maximal integral lattice \(M'\supset M\) is self-dual: a vector \(v\) lifting a nonzero Gram radical has pairing in \(\pi_2\mathbb Z_2[i]\) and norm in \(2\mathbb Z_2\), so adjoining \(v/\pi_2\) would otherwise enlarge \(M'\).
	Since \(M\subset M'\subset M^\#\), every generating reflection satisfies \((R_r-I)M'\subset M\subset M'\); thus \(M'\) is stable under the original \(G\).
	At split primes pair a stable module with its perfect dual; intersection patches these local self-dual modules globally, since only finitely many differ from the original lattice.
	A primitive isotropic vector splits an integral plane \(\left(\begin{smallmatrix}0&1\\1&k\end{smallmatrix}\right)\); its positive self-dual complement is \(I_3\), by successive unit vectors.
	If a positive self-dual Gaussian lattice of rank \(\ell\le3\) has no unit vector, its distinct points have squared distance at least \(2\).
	Any \(m\) points \(x_j\) of a translate in the unit ball then satisfy \(m\sum|x_j|^2-|\sum x_j|^2\ge m(m-1)\), hence \(\sum\max(1-|x_j|^2,0)\le1\).
	Integration over a fundamental domain of covolume one would give \(\pi^\ell/((\ell+1)\ell!)\le1\), which is false for \(\ell=1,2,3\).
	A unit vector splits off an orthogonal copy of \(\mathbb Z[i]\), and the complement is again self-dual.
	Make \(k\) even with a unit vector \(e\), split a hyperbolic basis \(v,w\), and use \((v-w+e,v+e,w-e)\) to obtain \(I_{1,4}\).
	A basis \(B\) with \(B^*(cA)B=I_{1,4}\) and the lift \(g\mapsto(\det g)^{-1}g\) give a subgroup of finite index of \(\mathrm{SU}(I_{1,4},\mathbb Z[i])\) whose projective image is the original \(\Gamma\).
	Here the normalized reflections have determinant \(-1\), so every \(g\) in their generated group has \(\det g=\pm1\), and the determinant of its new lift is \((\det g)^{-4}=1\).
	Uniqueness uses the absence of nontrivial fifth roots of unity in \(K\).
	The equivariant tautological map and rescaling by \(c\) preserve the automorphic bundle, the full finite stabilizers, the full normalizers, and the indices of all central cusp subgroups.
\end{proof}

\subsection{Degrees on the coincidence planes}

\begin{lemma}
	\label{ew-gaussian-plane-groups}
	The full generic stabilizer of a coincidence plane of rank \(2\) is \(A_1^2,A_2,B_2\), or \(G(4,2,2)\).
	No coincidence plane has full generic stabilizer \(A_2\).
	A coincidence plane with full generic stabilizer \(B_2\) has \(\lambda>1\).
\end{lemma}

\begin{proof}
	Meyer's theorem in rational dimension \(6\) embeds the full generic stabilizer faithfully into a horizontal parabolic group; its orders \(4,6,8,16\) exclude \(I_2(6)\), and the character field and \cite[Theorem 3.6]{KrishnasamyTaylor} give the list.
	For \(A_2\), the images of the line fixed by \(A_2\) under the invariants of \(A_3,G(4,4,3),G(4,2,3)\) are \((u^2,u^3,u^4),(u^3,3u^4,3u^8),(3v^2,v^3,3v^4)\), \(v=u^2\); their zero differentials contradict smooth embedding.
	The remaining \(B_3\) has matrix \(\left(\begin{smallmatrix}4u^2&4u^2\\12u^4&20u^4\end{smallmatrix}\right)\), giving \(d_B=2\lambda-4/5\), contrary to half integrality.
	For \(B_2\), put \(q_2=xy,q_4=x^4+y^4\); the invariants \((uq_2,u^4+q_4,u^4q_4+q_2^4)\) give the matrix \(\operatorname{diag}(v,v)\), \(v=u^4\), and the correction \(-3/4\); \(B_3\) gives \(\left(\begin{smallmatrix}v&1\\0&v\end{smallmatrix}\right)\), \(v=u^2\).
	Hence \(6\lambda-2=5d_C/4+2d_B\), whose sole integral solution at \(\lambda=1\) is \((0,2)\).
	The direction of degree \(4\) is preserved by the full normalizer, has trivial divisorial character, and its residual finite character is trivial on finitely punctured \(\mathbb P^2\); the row of the inverse differential associated with the invariant of greatest degree gives \(\mathcal O_Z(-1)^{\oplus2}\to\mathcal O_Z(-4\lambda+d_C+d_B)=\mathcal O_Z(-2)\), a contradiction.
\end{proof}

\begin{lemma}
	\label{ew-gaussian-plane-a1}
	A coincidence plane with full generic stabilizer \(A_1^2\) has \(\lambda>1\).
\end{lemma}

\begin{proof}
	Let \(d_A,d_B,d_C,d_P,d_Q,d_R\) be, respectively, the degrees of the reduced curves with full generic stabilizers
	\[
		A_3,\ B_3,\ G(4,2,3),\ B_2\times A_1,\
		G(4,2,2)\times A_1,\ A_1^3.
	\]
	The upper group \(A_2\times A_1\) is excluded, since the line fixed by \(A_1^2\) has singular image under the invariants of \(A_2\times A_1\).
	The normal determinant and the canonical class with effective branch orders give
	\[
		\begin{split}
			4\lambda-2 & =\frac32d_A+3d_B+\frac52d_C+d_P+\frac12d_Q, \\
			3\lambda+3 & =\frac12(d_A+d_B+d_P+d_R)+\frac34(d_C+d_Q).
		\end{split}
	\]
	Lemma~\ref{ew-gaussian-a1-cusps} proves these formulas and the cusp data used below.
	At \(\lambda=1\), the integral possibilities are
	\[
		(d_P,d_Q,d_R)=(0,4,6),(1,2,8),(2,0,10),
		\qquad d_A=d_B=d_C=0.
	\]
	The formal solution \(d_A=d_Q=1\) gives \(d_R=19/2\).
	For the first triple, count the three finite germs of Lemma~\ref{ew-gaussian-a1-cusps} by \(n_1,n_2,n_3\), nodes and tacnodes at cusps by \(N_0,T_0\); for \(x=\sum(g_j-1)\), the full normalizers give
	\[
		\begin{split}
			2x+2(N_0+T_0)+n_1/2+3n_2/4+3n_3/2 & =8,              \\
			x                                 & =2-n_3-N_0-2T_0, \\
			n_1+2n_2+2N_0                     & =24.
		\end{split}
	\]
	Thus \(20\le3\delta+2T_0\), where \(\delta=n_3+N_0+2T_0\); a reduced quartic with \(r_D\) irreducible components has \(\delta\le r_D+2\) by \(3=\sum_jg_j+\delta-r_D+1\), and the degree partitions bound the number of tacnodes by \(1,1,2,2,0\), making the right side at most \(19\).
	Let \(\Pi\subset\mathrm{PU}(2,1)\) be the effective image of the full normalizer of the plane, and \(\Delta\) a maximal lattice containing it.
	For the second and third triples, the full cusp data are \((|F_\Pi|,\deg M_\Pi)=(4,2\text{ or }4),(2,2)\), each with one cusp.
	Orbifold Euler number \(3\) gives \([\Delta:\Pi]=96\) and \((|F_\Delta|,\deg M_\Delta)=(4,2\text{ or }4)\) by Lemma~\ref{ew-gaussian-maximal-surface}.
	Let \(h\) be the index of the central translation subgroup of \(\Pi\) in that of \(\Delta\).
	Curvature and the unique cusp double coset give
	\[
		[\Delta:\Pi]
		=\frac{|F_\Delta|}{|F_\Pi|}
		\frac{\deg M_\Pi}{\deg M_\Delta}h^2.
	\]
	The factors other than \(h^2\) are powers of \(2\), so the index has even \(3\)-adic valuation, contradicting \(96\).
\end{proof}

\subsection{The fixed surface with group \texorpdfstring{\(G(4,2,2)\)}{G(4,2,2)}}

Put \(T_2=G(4,2,2)\) and \(T_3=G(4,2,3)\).
Suppose that a coincidence plane \(Z=Z_I\) has full generic stabilizer \(T_2\).
\cite[Lemma~2.12]{NakanoWeightedConfigurations2026} identifies the compactification of its fixed ball quotient with \(Z\).
Generic identifications normalize \(T_2\), the map is finite, and \(Z\) is normal.
Write \(D_C,D_P\) for the reduced curves with full generic stabilizers \(T_3,T_2\times A_1\), of effective branch orders \(4,2\).

The remaining equality \(\lambda=1\) is tested through the full normalizer action on three pairs of fixed points.
The associated cover has degree three and trace kernel \(N_{Z/X}(-D_C)\otimes\mathcal L^4|_Z\); the minors of its cubic coefficients cut out the points of \(Z\) with full stabilizer \(G_{31}\).
Lemmas~\ref{ew-gaussian-g422-types} and~\ref{ew-gaussian-g422-cover} make this construction precise and supply the contradiction in Proposition~\ref{ew-gaussian-g422-exclusion}.

\begin{lemma}
	\label{ew-gaussian-g422-types}
	The interior upper groups of \(D_C\) and the germs of \(D_C,D_P\) at their points are
	\[
		\begin{array}{c|cc}
			&D_C&D_P\\ \hline
			G(4,2,4)&y=0&x^2-4y=0\\
			G_{31}&y^2-x^3=0&\varnothing\\
			T_3\times A_1&x=0&y=0.
		\end{array}
	\]
	At \(G_{31}\), the image on the two indecomposable invariants of degree \(4\) is the full \(\mathfrak S_3\).
	The full cusp germs and the degrees of their two surface coordinates are
	\[
		\begin{array}{c|c|c|c|c}
			(d_1,d_2,d_3)&\text{actual degrees}&\deg(B,C)&D_C&D_P\\ \hline
			211&1224&(1,2)&C^2+B^4&\varnothing\\
			221&1112&(1,1)&C^2+B^2&B\\
			212&1244&(1,4)&C+B^4&C\\
			222&1122&(1,2)&C+B^2&BC.
		\end{array}
	\]
\end{lemma}

\begin{proof}
	The full parabolic normalizers \cite[Theorems 2.1, 3.6, 5.14 and Tables 5, 6, 8]{KrishnasamyTaylor} give the first table: use \((z_1^4+z_2^4,z_1^4z_2^4)\) for \(G(4,2,4)\), separate coordinates for the factors of products, and for \(G_{31}\)
	\[
		\begin{gathered}
			f_8=a^8+14a^4b^4+b^8,\qquad
			f_{12}=a^{12}-33a^8b^4-33a^4b^8+b^{12},\\
			J=-1728a^3b^3(a^4-b^4)^3,\qquad
			f_8^3-f_{12}^2=108a^4b^4(a^4-b^4)^4.
		\end{gathered}
	\]
	The six mirrors \(0,\infty,\pm1,\pm i\) give the full \(G_8\) of order \(96\); on \((S,P)=(a^4+b^4,(ab)^2)\), generators \(\mathsf A={(1-i)}^{-1}\left(\begin{smallmatrix}1&1\\1&-1\end{smallmatrix}\right)\) and \(\mathsf B=\operatorname{diag}(i,1)\) act by \((-S/2-3P,-S/4+P/2),(S,-P)\), giving \(\mathfrak S_3\).
	We use the full affine action on the horizontal torsor and the actual action on the polarizing line bundle, not only a linear action fixing an origin: take all permutations and phases in \(\mu_{d_1}\times\mu_{d_2}\times\mu_{d_3}\subset\mu_4(x_1)\times\mu_4(x_2)\times\mu_2(y)\), \(d_j\in\{1,2\}\).
	For the fiber character \(\chi_z\) of the full stabilizer of a horizontal point \(z\), \cite[Lemma~2.17]{NakanoWeightedConfigurations2026} lifts \(\ker\chi_z\) to the full finite stabilizer of a nonzero point of the fiber over \(z\).
	Every component and affine translate is \((p_0,p_0,t)\), \(p_0=(0,1,i),d_1=2\); the other branch supports give \(A_1^2,B_2\) or independent squares, and \(111\) gives none.
	The sections restrict to \((0,q^{d_2}x_2^{d_2},(-1)^{d_3}q^{2d_3}y^{d_3},q^4x_1^4)\), yielding the four polynomial rings and divisors after removing exactly the scalars of order \(g=d_2\) in the kernel of the fiber character.
	This removal is \(q\mapsto q^g\), which divides by \(g\) the central period used before the removal.
	The unique lift fixing the point preserves full stabilizers, and every pointwise scalar is trivial on \(q^g\); only further quotients \(1224,1244\mapsto1112,1122\) are smooth, so no component or smaller period is omitted.
\end{proof}

\begin{lemma}
	\label{ew-gaussian-g422-cover}
	There is a finite flat morphism \(p:Y\to Z\) of degree \(3\), with \(Y\) normal.
	The trace kernel of \(p\) is
	\[
		\mathcal B
		=N_{Z/X}(-D_C)\otimes\mathcal L^4|_Z
		=\mathcal O_Z(-b)^{\oplus2}.
	\]
	Its parameters satisfy
	\[
		b=(4\lambda-1)/3,\qquad
		\deg D_C=4b,\qquad
		\deg D_P=3(5-b)/2,\qquad b\in\{1,3,5\}.
	\]
	There are exactly \(3b^2\) distinct points of \(Z\) with full stabilizer \(G_{31}\).
\end{lemma}

\begin{proof}
	Let \(\Pi_Z\) be the effective full normalizer of the fixed ball, and let \(\Sigma\subset Z\) contain the cusps, interior upper points of \(D_C\), and singular points of \(D_C\cup D_P\).
	The image on the positive definite normal lattice is finite.
	The projective image of \(T_2\) is the Klein four group, whose three nonidentity elements have fixed point pairs \(\mathcal P=\{\{0,\infty\},\{1,-1\},\{i,-i\}\}\).
	Its normalizer in \(\mathrm{PGL}(2,\mathbb C)\) permutes these three pairs; the kernel is the projective image of \(T_2\), and the preceding matrices \(\mathsf A,\mathsf B\) give two transpositions.
	Thus every normal block has the form \(Q=cQ_0\), with \(Q_0\in\langle T_2,\mathsf A,\mathsf B\rangle\) and \(c\in K^*\).
	Both matrices have finite order, so \(c\in\mu_4\).
	The scalar \(c\) acts trivially on \((S,P)=(y_1^4+y_2^4,(y_1y_2)^2)\).
	The pointwise group \(T_2\) and the central lifts also act trivially.
	The resulting homomorphism \(\rho:\Pi_Z\to\mathfrak S_3\) acts by the standard representation, which is self-dual through the invariant form \(S^2+12P^2\).

	Put \(U=Z\setminus(D_C\cup\Sigma)\).
	The inertia along \(D_P\) is trivial under \(\rho\), so the permutation action on the three pairs defines a finite \'etale cover of degree \(3\) over \(U\).
	Explicitly, if \(\mathbb B_2^\circ\) is the inverse image of \(U\), this cover is \((\mathbb B_2^\circ\times\mathcal P)/\Pi_Z\to U\), with \(\gamma(z,p)=(\gamma z,\rho(\gamma)p)\) and \([(z,p)]\mapsto[z]\).
	Normalize \(Z\) in the function fields of the connected components of this cover to obtain \(p:Y\to Z\); this is the associated cover of degree \(3\), rather than the principal cover defined by \(\ker\rho\).
	We obtain flatness from normality in dimension \(2\) and the Auslander--Buchsbaum formula, and the trace splits \(p_*\mathcal O_Y=\mathcal O_Z\oplus\mathcal B\).

	If \(j(\gamma,z)\) is the automorphy factor on the fixed ball, the normal coordinates transform by \(y\mapsto Qy/j(\gamma,z)\).
	Hence the homogeneous normal directions transform by \(\rho(\gamma)j(\gamma,z)^{-4}\), and \(\mathcal L^4\) removes precisely this automorphy factor.
	The bundle \(\mathcal L^4\) descends on all of \(Z\) because \(\mathcal L^{\otimes2}\) already descends on \(X\).
	At a general point of \(D_C\), put \(v=u^4\).
	The local cover is the disjoint union of one unramified sheet and \(s^2=v\), where \(s=u^2\), so the trace kernel has frames \((2,-1,-1),(0,u^2,-u^2)\).
	These frames identify with the invariant normal vectors \((e_S,u^2e_P)\).
	After the tangent direction is removed from the invariants \((v+S,vS+P^2,u^2P)\), their normal differential is \(vI_2\), with image \(N_{Z/X}(-D_C)\).
	At \(D_P\), the cover is unramified and the normal differential is invertible.
	We obtain \(\mathcal B\otimes\mathcal L^{-4}\simeq N_{Z/X}(-D_C)\) on \(Z\setminus\Sigma\), including the regular lattices along every divisor.
	Normalization uniquely determines the cover at the upper points of \(D_C\) and cusps; local freeness on both sides then extends the bundle isomorphism across \(\Sigma\) by Hartogs.
	The trace discriminant and the fixed ball canonical formula give \(\deg D_C=4(\deg D_C+1-4\lambda)\) and \(-3+3\deg D_C/4+\deg D_P/2=3\lambda\).
	Miranda's cubic \(F=a_0s^3+3a_1s^2t+3a_2st^2+a_3t^3\), \(a_i\in H^0(\mathcal O_Z(b))\), defines a finite map \(\phi:Z\to\mathbb P^3\), since a common zero would give multiplicity \(4\) \cite[Theorem 3.6]{Miranda}.
	The discriminant germ \(v'^2-u'^3\), in new local coordinates \(u',v'\), along the twisted cubic \(W\) makes the points of \(\phi^{-1}(W)\) simple by the permitted local types; conversely, the minors of \(\left(\begin{smallmatrix}a_0&a_1&a_2\\a_1&a_2&a_3\end{smallmatrix}\right)\) characterize every point of \(Z\) with full stabilizer \(G_{31}\) \cite[Corollary 4.6]{Miranda}.
	The Hilbert--Burch sequence
	\[
		0\to\mathcal O_Z(-3b)^{\oplus2}\to\mathcal O_Z(-2b)^{\oplus3}
		\to\mathcal O_Z\to\mathcal O_{\phi^{-1}(W)}\to0
	\]
	gives total length \(3b^2\).
\end{proof}

\begin{lemma}
	\label{ew-gaussian-schwarz}
	Let \(Z\) be the Baily--Borel compactification of the quotient of a fixed ball of dimension \(2\) by the effective image of its full normalizer, and assume that \(2L\) descends to this quotient.
	Let \(C\subset Z\) be a smooth projective curve, and let \(S\) be the set of cusps of \(Z\) on \(C\).
	On \(Z\), let \(2L\) denote the pullback, under the finite map \(Z\to X\) obtained from the fixed ball quotient, of the Baily--Borel extension of~\cite[Proposition~2.11]{NakanoWeightedConfigurations2026} on \(X\), defined by frames invariant under the full cusp stabilizers with no further twist by a power of a local parameter.
	Then
	\[
		2L\cdot C\le2g(C)-2+|S|+\sum_{z\notin S}(1-1/e_z).
	\]
	Here \(e_z\), for \(z\in C\setminus S\), is the effective inertia on the normalization of a lifted branch.
	Equality makes every lifted branch a complex geodesic.
\end{lemma}

\begin{proof}
	On a normal neat cover of the same linear group, kill finite characters, normalize every component of the preimage of \(C\setminus S\), and divide Riemann--Hurwitz by the degree after removing the pointwise kernel.
	The restricted ambient canonical frame at the cusp introduces no parameter power; for curvature \(-1\), \(c_1(2L)=\omega/(2\pi)\), and the potential \(-2\log(-c\log|t|+A(t))\) for \(q=t^ma(t)\) has zero Lelong number and Stokes limit.
	Schwarz and Gauss--Bonnet give the formula, and metric equality gives a complex geodesic.
\end{proof}

\begin{proposition}
	\label{ew-gaussian-g422-exclusion}
	A coincidence plane with full generic stabilizer \(G(4,2,2)\) has \(\lambda>1\).
\end{proposition}

\begin{proof}
	Suppose \(\lambda=1\); then \(b=(4\lambda-1)/3=1\) by Lemma~\ref{ew-gaussian-g422-cover}.
	For \(b=1\), the plane \(\phi(Z)\) meets the twisted cubic \(W\) in three simple points; after a change of coordinates in \(\mathrm{PGL}(2,\mathbb C)\), the cubic \(F\), the normalization of its discriminant curve, and the tangent map of this curve are
	\[
		\begin{gathered}
			F=x(s^3-t^3)+3ys^2t+3zst^2,\\
			[3s^2t^2:-t(2s^3+t^3):s(s^3+2t^3)],\quad
			[s^3-t^3:3s^2t:3st^2].
		\end{gathered}
	\]
	The syzygies \((0,t,-s),(3st,-s^2,t^2)\) identify the incidence surface \(\{F=0\}\subset\mathbb P^2\times\mathbb P^1\) with the Hirzebruch surface \(\mathbb F_1=\mathbb P_{\mathbb P^1}(\mathcal O\oplus\mathcal O(1))\).
	The tangent map identifies only \([1:0],[0:1]\), so \(D_C\) is tricuspidal with a unique bitangent \(\ell=(x=0)\).
	For intersection \(k\) with \(D_C+D_P\), omitting a component equal to the curve, effective inertia satisfies \(1-1/e\le k/2\): for the Jacobian determinant \(J\) of the local quotient map on a normalized preimage branch of multiplicity \(\nu\), use \(2\operatorname{ord}J=ek\), \(\operatorname{ord}J\ge e-1+(\nu-1)^2\); for a component of \(D_C+D_P\) with pullback equation \(v^2g\) along a mirror \(v=0\), use \(2(e-1+\operatorname{ord}g)=\operatorname{ord}g+ek\).
	Equality requires one transverse reflection, or two orthogonal reflections for a component.
	If \(\ell\not\subset D_P\), the tangencies of \(\ell\) contribute \(1/2\) each, strictly less than \(1/2+(\ell\cdot D_P)_z/2\) at points with a group in the first table of Lemma~\ref{ew-gaussian-g422-types} or at cusps; the second table in that lemma gives intersection at least \(2\), and all cusps lie on \(D_C\).
	On the coincidence plane \(Z\simeq\mathbb P^2\), with \(L|_Z=\lambda c_1(\mathcal O_Z(1))\), the descent gives \(2\lambda\in\mathbb Z\), and the restriction to \(Z\) of the extension in Lemma~\ref{ew-gaussian-schwarz} is \(\mathcal O_Z(2\lambda)\) since \(\operatorname{Pic}(\mathbb P^2)=\mathbb Z\).
	Since \(\deg D_P=6\), Lemma~\ref{ew-gaussian-schwarz}, applied to the line \(\ell\subset Z\), which may pass through cusps, gives equality in \(2=2L\cdot\ell\le-2+|S|+\sum(1-1/e_z)\le-2+1+3=2\).
	Both general interior branches \(v=\pm u^2\) would be geodesic with the same point and tangent, contradicting uniqueness.
	If \(\ell\subset D_P\), the bound is again \(-2+2(3/4)+5/2=2\), with equality forcing interior \(G(4,2,4)\) tangencies and no cusp.
	Every component of \(D_P\) nevertheless has a cusp for the same \(\Gamma\).
	On the \(T_2\) root space, the Hermitian form is \(aI_2\), whose determinant is a norm from the Gaussian field.
	The primitive root of the mirror has norm \(1\) or \(2\), by self-duality; we obtain \(\mathrm N(\delta)\) as the negative of the determinant of the Hermitian form on the mirror.
	For orthogonal norms \(\beta>0,\gamma<0\), \(-\beta\gamma=\mathrm N(\delta)\) gives isotropic \(w=(\delta/\beta,1)\); \(Q(x)=h(w,x)w\) satisfies \(Q^2=0\), and a positive power of the integral \(I+miQ\) belongs to the original \(\Gamma\) and centralizes \(T_2\).
\end{proof}

\begin{proposition}
	\label{ew-gaussian-plane-strict}
	For every coincidence plane, \(\lambda_I\ge3/2\).
	In particular, \(2L-3H\) is nef.
\end{proposition}

\begin{proof}
	Lemmas~\ref{ew-fields-plane-equality}, \ref{ew-gaussian-plane-groups}, and~\ref{ew-gaussian-plane-a1} and Proposition~\ref{ew-gaussian-g422-exclusion} exclude \(\lambda_I=1\) in every rank.
	Half integrality from Proposition~\ref{ew-seven-orderfour} gives \(\lambda_I\ge3/2\); apply \cite[Proposition~2.6]{NakanoWeightedConfigurations2026} to \(D-3H\) on each of the \(35\) planes.
\end{proof}

\section{Chern classes and the seven point quotient}
\label{ew-seven-global-section}
We now use \(N=L-3H/2\) nef and a lattice preserved by the original Gaussian group.
The proof combines three kinds of data: absolute Chern classes from all fixed components, the determinant identity along each curve, and the relations \eqref{ew-quotient-relations} at interior points and cusps.
Only the resulting nonnegative intersection numbers are compared; no degree bound for one component is required.

\subsection{Chern classes of all fixed strata}

Let \(X\) be a smooth projective variety of dimension \(d\) isomorphic to the Baily--Borel compactification of an arithmetic ball quotient.
We keep its lattice \(\Gamma\) and automorphic class \(L\) fixed.
For each full stabilizer \(P\), let \(V_P\) be its nontrivial reflection representation and let \(Z_P\) be the sum of the reduced closures of the strata with generic stabilizer \(P\).
Count each component once on its quotient by its full normalizer.
Put \(\lambda_P=\#\{g\in P:V_P^g=0\}/|P|\), \(\mathcal C_k=\sum_{\operatorname{rank}P=k}\lambda_P[Z_P]\), and \(\mathcal C_0=[X]\).
If \(e_{P,1},\ldots,e_{P,k}\) are the degrees of the basic invariants, then
\begin{equation}
	\lambda_P=\frac{\prod_i(e_{P,i}-1)}{\prod_i e_{P,i}}.
	\label{ew-global-lambda}
\end{equation}
One has \(\sum_{g\in P}t^{\dim V_P^g}=\prod_i(t+e_{P,i}-1)\); see \cite{ShephardTodd,GeckMalle}.

We use \cite[Section~3.5]{Ohmoto}, which averages over single group elements, together with the normalization of Chern--Schwartz--MacPherson classes in \cite[Theorem~1.1]{Ohmoto}.
The averaging and binomial inversion are formal.
To apply this averaging, we use the full normalizer of each fixed ball, all its inertia groups, and the restriction of the original automorphic class \(L\).
The coefficient calculation retains every local branch of a self-intersection.
The identities for \(k<d\) concern pairings with divisor classes; the cusps contribute separately in degree \(d\).

\begin{proposition}
	\label{ew-chern-general}
	Assume that the quotient of every positive dimensional fixed ball by its full normalizer has finite volume and has automorphic class equal to the restriction of \(L\).
	For \(0\leq k<d\), the following equality holds after pairing with products of divisor classes:
	\begin{equation}
		\mathcal C_k=\sum_{j=0}^k\binom{d-k+j+1}{j}L^jc_{k-j}(T_X).
		\label{ew-global-all-chern}
	\end{equation}
	It is an equality of numerical cycle classes if these pairings detect numerical equivalence.
	If \(c\) is the number of cusps, then
	\begin{equation}
		c+\deg\mathcal C_d
		=\int_X\sum_{j=0}^d(j+1)L^jc_{d-j}(T_X).
		\label{ew-global-top-chern}
	\end{equation}
\end{proposition}

\begin{proof}
	On a normal neat cover of the same \(\Gamma\), apply Burnside's formula to fixed balls.
	For a full generic stabilizer \(P\), and the normalizer \(\mathcal M\) and centralizer \(\mathcal Z\) in the deck group of an element fixing precisely this ball, so that \(\mathcal Z\) normalizes \(P\) and \(\mathcal ZP\) is their product subgroup, the coefficient on the quotient of the ball by the effective image of its full normalizer is \([\mathcal M:\mathcal ZP]/|\mathcal Z\cap P|=[\mathcal M:\mathcal Z]/|P|\).
	Summing its conjugacy classes gives \(\lambda_P\), with every local branch retained.
	For \(k<d\), a general complete intersection of \(d-k\) sufficiently ample divisors avoids the cusps and higher rank strata.
	The fixed ball of dimension \(d-s\) has orbifold Chern polynomial \((1-Lt)^{d-s+1}\) by \cite[Equation~(3.13)]{NakanoWeightedConfigurations2026}.
	Write \(A_1,\ldots,A_{d-k}\) for the divisor classes of this complete intersection \(Y\), and put \(A=A_1\cdots A_{d-k}\).
	Burnside's formula and Gauss--Bonnet on each normalized fixed component give
	\[
		e(Y)=\sum_{s=0}^k\int_X A\mathcal C_s
		\left[\frac{(1-L)^{d-s+1}}{\prod_{i=1}^{d-k}(1+A_i)}\right]_{k-s},
	\]
	where brackets denote the component of the indicated codimension.
	On the other hand, adjunction gives \(e(Y)=\int_X A[c(T_X)/\prod_i(1+A_i)]_k\).
	Canceling the common denominator inductively in \(k\), and polarizing the ample divisor classes, gives
	\begin{equation}
		\label{ew-global-ck}
		c_k(T_X)=\sum_{s=0}^k(-1)^{k-s}\binom{d-s+1}{k-s}L^{k-s}\mathcal C_s.
	\end{equation}
	Binomial inversion proves \eqref{ew-global-all-chern}.
	For \(k=d\), use compactly supported Euler characteristics on the open quotient and finite volume Gauss--Bonnet; the left side of \eqref{ew-global-ck} with \(k=d\) becomes \(\int_Xc_d(T_X)-c\).
	Inverting gives \eqref{ew-global-top-chern}; ends of fixed components are already counted by these Euler characteristics.
\end{proof}

For the applications below, \cite[Lemma~2.12]{NakanoWeightedConfigurations2026} gives the finite maps from the quotients of the fixed balls by their full normalizers and proves finite volume.
The automorphic class on each fixed ball is the restriction of the original class \(L\), as in \cite[Equation~(3.13)]{NakanoWeightedConfigurations2026}.
We use these geometric identifications and the coefficient calculation in the proof of \eqref{ew-global-all-chern}.
For \eqref{ew-global-top-chern}, we use compactly supported Euler characteristics on the open quotient and retain the contribution of each cusp of \(X\).

\subsection{All Gaussian fixed components}

Let \(X=X_7\) again, retain the group and integral lattice of Proposition~\ref{ew-seven-unimodular}, and put \(H=-K_X\), \(c_j=c_j(T_X)\).
Every reflection has order \(2\), and its primitive root has norm \(1\) or \(2\).
Proposition~\ref{ew-gaussian-plane-strict} gives
\begin{equation}
	N=L-\frac32H\ \text{nef}.
	\label{ew-global-nef}
\end{equation}
We do not exclude either root norm or the reflection subgroups \(G(4,2,2)\) and \(G(4,2,3)\).

For a reflection group \(G\), put \(d_G=\sum_i e_{G,i}\), where the \(e_{G,i}\) are the degrees of its basic invariants.
The full generic stabilizers of fixed surfaces are \(P_D=A_1^2\), \(P_A=A_2\), \(P_B=B_2\), and \(P_T=G(4,2,2)\).
Their degrees, weights from \eqref{ew-global-lambda}, and the coefficients \(c_P\) are
\[
	\begin{array}{c|c|c|c}
		P&(e_{P,i})&\lambda_P&c_P\\\hline
		P_D&(2,2)&1/4&2\\
		P_A&(2,3)&1/3&4\\
		P_B&(2,4)&3/8&6\\
		P_T&(4,4)&9/16&15/2
	\end{array},\qquad
	c_P=\frac{4\nu_P(\nu_P-1)}{|P|},\quad \nu_P=d_P-2.
\]
The rank \(2\) classification over \(\mathbb Q(i)\) leaves these groups and \(I_2(6)\) among groups generated by reflections of order \(2\) \cite{GeckMalle}.
For \(I_2(6)\), primitive roots \(r,s\) of norms \(a,b\in\{1,2\}\) would satisfy \(4|h(r,s)|^2=3ab\).
The only integral possibility requires a Gaussian integer of norm \(3\), which does not exist.
The groups on fixed curves are listed in Table~\ref{ew-global-rankthree}.
This list follows directly from the rank \(3\) classification over \(\mathbb Q(i)\); no cusp is assumed on an individual curve.

\begin{table}[htbp]
	\centering
	\small
	\begin{tabular}{c|c|r|c|c}
		\(K_k\)                    & \((e_{K_k,i})\) & \(|K_k|\) & \((n_{D,K_k},n_{A,K_k},n_{B,K_k},n_{T,K_k})\) & \(\lambda_{K_k}\) \\\hline
		\(K_0=A_1^3\)              & \((2,2,2)\)     & 8         & \((3,0,0,0)\)                                 & \(1/8\)           \\
		\(K_1=A_2\times A_1\)      & \((2,2,3)\)     & 12        & \((3,1,0,0)\)                                 & \(1/6\)           \\
		\(K_2=B_2\times A_1\)      & \((2,2,4)\)     & 16        & \((4,0,1,0)\)                                 & \(3/16\)          \\
		\(K_3=A_3\)                & \((2,3,4)\)     & 24        & \((3,4,0,0)\)                                 & \(1/4\)           \\
		\(K_4=B_3\)                & \((2,4,6)\)     & 48        & \((6,4,3,0)\)                                 & \(5/16\)          \\
		\(K_5=G(4,4,3)\)           & \((3,4,8)\)     & 96        & \((0,16,3,0)\)                                & \(7/16\)          \\
		\(K_6=G(4,2,2)\times A_1\) & \((2,4,4)\)     & 32        & \((6,0,0,1)\)                                 & \(9/32\)          \\
		\(K_7=G(4,2,3)\)           & \((4,6,8)\)     & 192       & \((12,16,0,3)\)                               & \(35/64\)         \\
	\end{tabular}
	\caption{Groups on fixed curves and counts of their parabolic subgroups of rank \(2\).}
	\label{ew-global-rankthree}
\end{table}

Write \(W_k\) for the sum of the reduced closures of the fixed curves with generic stabilizer \(K_k\) in the table; components are counted as for \(Z_P\).
Different local branches of a self-intersection remain distinct.
The number \(n_{P,K}\) counts subspaces with full pointwise stabilizer \(P\) in \(K\).
Products give the reducible rows in the table.
For \(G(4,2,3)\), coordinate axes and lines with two coordinates equal up to a phase give the counts \(12,16,0,3\); for \(G(4,4,3)\), they give \(0,16,3,0\).
Define
\[
	v_{P,K}=(d_K-d_P-1)\frac{|P|n_{P,K}}{|K|}.
\]
No rank \(4\) parabolic counts are needed.

\begin{lemma}
	\label{ew-global-git-data}
	With the classes \(\psi_i,p\) of Lemma~\ref{ew-git-seven-classes}, one has
	\begin{equation}
		\begin{gathered}
			\int_XH^4=154,\qquad \int_XH^2p=-14,\qquad
			\int_Xp^2=10,\qquad \int_Xc_4=38,\\
			Hp\equiv-\frac1{11}H^3,\qquad
			c_3\equiv\frac4{11}H^3,\qquad
			p^2\equiv\frac5{77}H^4.
		\end{gathered}
		\label{ew-global-git-numbers}
	\end{equation}
	The nef cone is generated by \(U_i=(H-\psi_i)/2\) and \(V_i=(H+\psi_i)/2\), \(1\leq i\leq7\).
	For \(E=\frac12\sum_i a_i\psi_i\), \(E'=\frac12\sum_i a_i'\psi_i\), \(s=\sum_i a_i\), and \(s'=\sum_i a_i'\), the intersection form is
	\begin{equation}
		\int_XH^2EE'=\frac32ss'-5\sum_i a_ia_i',
		\qquad
		-\int_XpEE'=\frac12ss'-3\sum_i a_ia_i'.
		\label{ew-global-intersection}
	\end{equation}
\end{lemma}

\begin{proof}
	For a chamber divisor \(D\) of multidegree \((b_i)\), put \(q=\frac12\sum b_i\) and \(d_I=q-\sum_{i\in I}b_i\).
	Subtracting torus weight multiplicities \(2\) from \(0\) in the product of seven characters gives
	\[
		h^0(X,kD)=\sum_{|I|\leq3}(-1)^{|I|}\binom{kd_I+5-|I|}{5}.
	\]
	Its leading coefficient yields \(\int_XD^4=3q^4-2\sum_i(q-b_i)^4+\sum_{i<j}(q-b_i-b_j)^4\).
	Polarize and use \(\psi_i^2=p\): the intersections of four distinct \(\psi_i\) are \(2\), \(p\psi_i\psi_j=-2\) for \(i\ne j\), and \(p^2=10\).
	Expanding \(H=\sum_i\psi_i\) and the tangent sequence gives the stated numbers, \eqref{ew-global-intersection}, and \(c_4=(H^4-22H^2p+45p^2)/24\).
	Divisor classes generate the rational Chow ring \cite[Lemma 8, Theorem 10, Remark 3]{FranzenReineke}, so the symmetric pairings determine the numerical relations.
	The criterion of \cite[Proposition~2.6]{NakanoWeightedConfigurations2026}, for triples of labels, gives the \(14\) nef generators.
\end{proof}

Let \(\mathscr B\) be the reduced reflection divisor, as in Lemma~\ref{ew-compact-nef}.
Proposition~\ref{ew-chern-general} gives
\begin{equation}
	\begin{gathered}
		\mathscr B=10L+2H,\qquad
		\sum_P\lambda_PZ_P=c_2+4LH+10L^2,\\
		\sum_{k=0}^7\lambda_{K_k}W_k=c_3+3Lc_2+6HL^2+10L^3.
	\end{gathered}
	\label{ew-global-cycle-chern}
\end{equation}
The normal bundle of the fixed ball of a reflection has class \(-L\).
Counting ordered pairs of distinct local mirrors, including branches of one global divisor, gives \(\sum_Pc_PZ_P=\mathscr B^2+2L\mathscr B\).
Solving this and the second equation in \eqref{ew-global-cycle-chern} yields
\begin{equation}
	\begin{aligned}
		Z_D+\frac23Z_A+\frac34Z_T&=20L^2+10LH+2H^2-20p,\\
		\frac23Z_A+\frac32Z_B+\frac32Z_T&=20L^2+6LH+10p.
	\end{aligned}
	\label{ew-global-surface-relations}
\end{equation}
Proposition~3.1 of \cite{NakanoWeightedConfigurations2026}, with \(d=4,k=m=2\), gives
\begin{equation}
	(H+(d_P+3)L)Z_P=\sum_kv_{P,K_k}W_k.
	\label{ew-global-adjunction}
\end{equation}
Finally, if \(c\) is the number of cusps and \(n_J\) the number of points of full rank \(4\) type \(J\), put
\begin{equation}
	F(L)=\int_X(c_4+2Lc_3+3L^2c_2+4HL^3+5L^4)
	=c+\sum_J\lambda_Jn_J.
	\label{ew-global-F}
\end{equation}

\subsection{The same determinant on every curve}

For \(J\) as in \eqref{ew-global-F}, let \(\gamma_k(J)=|K_k|n_{K_k,J}/|J|\), where \(n_{K,J}\) counts the lines in the tangent space at a point of type \(J\) with full pointwise stabilizer \(K\).
At a cusp \(c'\), let \(\gamma_k(c')\) be the degree of the stack points with nonzero fiber stabilizer \(K_k\) in the weighted projective quotient of the horizontal torus.
Write \(b(c')\) for the generator degrees measured by the least positive central period of the full cusp group, and put \(s_{c'}=\sum_jb_j(c')\).
With \(\vartheta_k=\int_XHW_k\), \(\tau_k=\int_XLW_k\), Proposition~3.1 of \cite{NakanoWeightedConfigurations2026} gives
\begin{equation}
	\vartheta_k+(d_{K_k}+2)\tau_k
	=\sum_J(d_J-d_{K_k}-1)\gamma_k(J)n_J
	+\sum_{c'\ \mathrm{cusp}}s_{c'}\gamma_k(c').
	\label{ew-global-curve-determinant}
\end{equation}
Each such line is tangent to the disk fixed by its pointwise stabilizer, which gives a fixed curve for the full normalizer by \cite[Lemma~2.12]{NakanoWeightedConfigurations2026}.
The sum over the orbits of these lines at a point of type \(J\) is \(\sum1/e=|K|n_{K,J}/|J|\).
At a cusp, \cite[Lemma~2.17]{NakanoWeightedConfigurations2026} gives the unique lift fixing the point in the same \(\Gamma\); the fixed disk of this lift gives the corresponding cusp of the fixed curve, with contribution \(1/\kappa\) for the full fiber character of order \(\kappa\).
Thus every stack point counted by \(\gamma_k(c')\) occurs as a cusp of a fixed curve, with its actual width.
The identity includes compact fixed curves, for which the cusp sum is empty.
Nefness in \eqref{ew-global-nef} gives \(\tau_k\ge3\vartheta_k/2\), independently of the degree and genus of each component.

The possible degree tuples of the groups \(J\) are
\begin{equation}
	\begin{gathered}
		(2,2,2,2),\quad (2,2,2,3),\quad (2,2,2,4),\quad (2,2,3,3),\\
		(2,2,3,4),\quad (2,2,4,4),\quad (2,2,4,6),\quad (2,3,4,4),\\
		(2,3,4,5),\quad (2,3,4,8),\quad (2,4,4,4),\quad (2,4,4,6),\\
		(2,4,6,8),\quad (2,6,8,12),\quad (4,4,4,4),\quad (4,4,8,12),\\
		(4,8,8,12),\quad (4,8,12,20),\quad (8,12,20,24).
	\end{gathered}
	\label{ew-point-degree-tuples}
\end{equation}
The irreducible groups are \(A_4,D_4,B_4,F_4,G(4,4,4),G(4,2,4),G_{29},G_{31}\); the remaining tuples come from products \cite{ShephardTodd,GeckMalle}.
At \(F_4\), deleting one vertex of the Coxeter diagram shows that the only parabolic types of rank \(3\) are \(A_2\times A_1\) and \(B_3\).
Accordingly \(\gamma_k(F_4)=0\) for \(k\notin\{1,4\}\).
No other vanishing of \(\gamma_k(J)\) is used.

The generator degrees \(b(c')\) lie in
\begin{equation}
	1111,\quad1112,\quad1122,\quad1222,\quad1124,\quad1144,\quad1224,\quad1244.
	\label{ew-global-cusp-degrees}
\end{equation}
The full cusp group is generated by actual reflections, so its horizontal image is generated by reflections of order \(2\).
The horizontal groups are \(A_3,B_3,G(4,4,3),G(4,2,3)\); the product of the nonzero entries of a monomial matrix is multiplicative and equals \(\pm1\) on reflections of order \(2\), so these reflections cannot generate \(G(4,1,3)\), which has an element with value \(i\).
For \(G(4,2,3)\), the five rows of \cite[Section~4.4]{Rains}, including its torsor, have parameters
\[
	(d_1,d_2,d_3)=(1,1,2),(1,2,1),(1,2,2),(2,2,1),(2,2,2).
\]
The model in \cite[Example~3.10]{Rains} has degrees \((d_1,d_2,2d_3,4)\).
Removing exactly the scalars \(\mu_g\) in the kernel of the fiber character gives \(1144,1224,1244,1112,1122\), respectively; \(G(4,4,3)\) gives \(1124\), and \(A_3,B_3\) give the first four tuples.
This operation replaces the central period \(c_{\rm old}\) and the coordinate \(q_{\rm old}=\exp(2\pi iw/c_{\rm old})\), for the central coordinate \(w\), by \(c_{\rm new}=c_{\rm old}/g\) and \(q_{\rm new}=q_{\rm old}^g\), and returns the grading to the least positive central period of the full cusp group.
Descent along a free translation cover retains the linearization of the full group and the unique lifts fixing the points.
Any further scalar group with smooth punctured quotient has order dividing \(\operatorname{lcm}_j\gcd_{i\ne j}b_i\), since the finite quotient group is generated by reflections.
The only nontrivial cases are \(1222\to1111\), \(1224\to1112\), and \(1244\to1122\), so no additional tuple occurs.

\subsection{Quotients of projective space and abelian threefolds}

\begin{proposition}
	\label{ew-quotient-models}
	Let \(M=\mathbb P^3\), \(\Theta=\mathcal O(1)\), \(\varepsilon=1\), or let \(M\) be an abelian threefold with an ample line bundle \(\Theta\), \(\varepsilon=0\).
	Let a finite group act on \((M,\Theta)\), faithfully on the punctured total space of \(\Theta^{-1}\), with a polynomial ring of invariant sections whose generators have degrees \(b=(b_1,\ldots,b_4)\).
	Assume that the punctured quotient is smooth, that every actual reflection has order \(2\), and that the stabilizers of ranks \(2\) and \(3\) have the types above.
	On the weighted projective stack of weights \(b\), let \(\beta_P\) be the degree of the reduced image of the curves of generic type \(P\), measured by \(\eta=c_1(\mathcal O(1))\), and let \(\gamma_k\) be the degree of its stack points of type \(K_k\).
	The group at a point is the kernel of the full character on the nonzero fiber, and all residual stabilizer orders are retained in these degrees.
	Put \(s=\sum b_i\), \(v=\prod b_i\), and let \(e_j\) denote the elementary symmetric polynomial of degree \(j\).
	Then \(\beta_P,\gamma_k\ge0\) and
	\begin{equation}
		\begin{gathered}
			\sum_P\lambda_P\beta_P=\frac{e_2(b-\varepsilon)}v,\qquad
			\sum_Pc_P\beta_P=\frac{4(s-4\varepsilon)(s-5\varepsilon)}v,\\
			\sum_k\lambda_{K_k}\gamma_k=\frac{e_3(b-\varepsilon)}v,\qquad
			(s-\varepsilon(d_P+2))\beta_P=\sum_kv_{P,K_k}\gamma_k.
		\end{gathered}
		\label{ew-quotient-relations}
	\end{equation}
	Here \(b-\varepsilon=(b_i-\varepsilon)_i\); any common divisor of the \(b_i\) is retained.
\end{proposition}

\begin{proof}
	The sections give a finite map from \(M\) to the weighted projective stack, pulling \(\mathcal O(1)\) back to \(\Theta\).
	The sections have no common zero on \(M\); after base change by the atlas \(\mathbb A^4\setminus\{0\}\), this map is the finite quotient of the punctured total space of \(\Theta^{-1}\) by the stated group.
	For a horizontal point with stabilizer \(F_x\) and fiber character \(\chi_x\), its residual order is \(r=|\operatorname{im}\chi_x|\), and orbit counting assigns weight \(|\ker\chi_x|/|F_x|=1/r\).
	Thus a point contributes \(1/r\) to \(\gamma_k\); a curve contributes its degree against \(\eta\), with the same weight at its generic point, to \(\beta_P\).
	For example, a weighted projective coordinate point of residual order \(r\) has stack degree \(1/r\), rather than coarse degree one.
	As in the proof of Proposition~\ref{ew-chern-general}, apply Burnside's formula to general complete intersections in \(M\).
	The fixed components are projective linear spaces for \(\varepsilon=1\) and translates of subtori for \(\varepsilon=0\), with Chern polynomials \((1+\varepsilon\eta)^{4-a}\) in codimension \(a\).
	The weighted Euler sequence gives \(c_j=e_j(b)\eta^j\) and \(\int\eta^3=1/v\).
	Binomial inversion yields the first and third equations.
	For the reduced branch divisor, ramification gives \(\mathscr D=2(s-4\varepsilon)\eta\).
	Counting ordered pairs of distinct mirrors, each of normal class \(\varepsilon\eta\), gives \(\sum_Pc_P\mathcal Z_P=\mathscr D^2-2\varepsilon\eta\mathscr D\) for the cycle classes \(\mathcal Z_P\) on the stack of the curve images of type \(P\); since \(\eta\cdot\mathcal Z_P=\beta_P\), this proves the second equation.
	On a normalized curve of type \(P\), the tangent and the two invariant normal directions have total first Chern class \(\varepsilon(d_P+2)\eta\).
	The same determinant calculation as in \eqref{ew-global-adjunction}, with its full characters, gives the last four equations.
\end{proof}

\subsection{A comparison of all components}

We first explain the geometric inputs to the final inequality.
The classes \(H,p\) come from the GIT quotient \(X\) and \(L\) from the automorphic bundle, while \(Z_P\) and \(W_k\) record every fixed surface and curve with its full normalizer.
Equations~\eqref{ew-global-surface-relations}, \eqref{ew-global-adjunction}, and \eqref{ew-global-curve-determinant} give linear equalities among their intersection numbers.
Nefness of \(L-3H/2\) and the generators of the GIT nef cone give inequalities among the same numbers.
The quotients of \(\mathbb P^3\) at interior points and of horizontal tori at cusps satisfy~\eqref{ew-quotient-relations}; their degrees \(\beta_P,\gamma_k\) include the finite fiber characters and all local branches.
Thus the global equalities, nef inequalities, and \eqref{ew-quotient-relations} constrain one collection of intersection numbers.
The rational linear combination below separates these constraints: it is nonnegative by nefness and strictly negative after the equalities and local estimates are substituted.
Its coefficients are a certificate of this incompatibility, rather than an additional geometric assumption.

The same method on another smooth projective variety would require its intersection relations and nef inequalities, finite maps from all relevant fixed balls by their full normalizers with the same restricted automorphic class, and the local degree formulas at every interior point and cusp.
Proposition~\ref{ew-chern-general} alone does not supply these inputs or a nonexistence theorem for another variety.

For any two nef classes \(E,E'\), the fourteen generators in Lemma~\ref{ew-global-git-data} give
\begin{equation}
	0\le-\int_XpEE'\le\frac17\int_XH^2EE'.
	\label{ew-global-nef-p}
\end{equation}
Put \(l=\int_XL^4>0\) and normalize all intersection numbers by \(l\):
\[
	\begin{aligned}
		(z_P,h_P,j_P)&=l^{-1}\int_X(L^2Z_P,HLZ_P,H^2Z_P),\\
		(a_k,w_k)&=l^{-1}(\vartheta_k,\tau_k),\\
		(\rho,x,u,y,z_4)&=l^{-1}\int_X(HL^3,H^2L^2,-pL^2,H^3L,H^4).
	\end{aligned}
\]
All these numbers are nonnegative.
Nefness of \(N,L,H\) and \eqref{ew-global-nef-p} imply
\begin{equation}
	\rho,x,y,z_4\le1,\qquad u\le1/7,\qquad
	z_P,h_P,j_P\le64,\qquad a_k,w_k\le160,\qquad F(L)/l<13.
	\label{ew-global-coarse-bounds}
\end{equation}
One has \(\rho\le2/3\), \(x\le2\rho/3\), \(y\le2x/3\), and \(z_4\le2y/3\).
The two equations in \eqref{ew-global-surface-relations}, paired with \(L^2\), bound their positive left sides by \(35,26\), respectively, and give \(z_P<64\).
The inequalities \(z_P\ge3h_P/2\ge9j_P/4\) give the bounds on \(h_P\) and \(j_P\).
The last equation in \eqref{ew-global-cycle-chern}, paired with \(L\), gives
\[
	\sum_k\lambda_{K_k}w_k=10+6\rho+\frac32x+\frac{15}2u+\frac4{11}y<19.
\]
Since \(\lambda_{K_k}\ge1/8\) and \(w_k\ge3a_k/2\), the bounds on \(a_k,w_k\) follow.
Finally,
\[
	F_0=F(L)/l=5+4\rho+\frac32x+\frac{15}2u+\frac8{11}y+\frac{19}{77}z_4<13.
\]

\begin{lemma}
	\label{ew-global-positive}
	Equations~\eqref{ew-global-cycle-chern}--\eqref{ew-global-curve-determinant} and \eqref{ew-quotient-relations} contradict \eqref{ew-global-nef}.
\end{lemma}

\begin{proof}
	Let \(n_\nu,c_\nu\ge0\) be the numbers of interior points and cusps, respectively, with the degree tuple \(b\) and the vector \((\beta_P,\gamma_k)\) of \eqref{ew-quotient-relations} indexed by \(\nu\), divided by \(l\).
	For interior points put \(\Lambda_\nu=\prod_i(b_i-1)/\prod_i b_i\).
	Then \(F_0=\sum_\nu\Lambda_\nu n_\nu+\sum_\nu c_\nu\).

	Let \(R_1,R_2\) be the left sides minus the right sides of \eqref{ew-global-surface-relations}.
	For \((D_1,D_2,D_3)=(L^2,HL,H^2)\), put \(\mathcal E_{2j-1}=l^{-1}\int_XR_1D_j\) and \(\mathcal E_{2j}=l^{-1}\int_XR_2D_j\).
	Let \(\mathcal E_7,\ldots,\mathcal E_{10}\) be \eqref{ew-global-adjunction} paired with \(L\), in the order \(D,A,B,T\).
	Let \(\mathcal E_{11}\) be the last equation in \eqref{ew-global-cycle-chern} paired with \(L\), and let \(\mathcal E_{12}=F_0-\sum\Lambda_\nu n_\nu-\sum c_\nu\).
	Let \(\mathcal E_{13},\ldots,\mathcal E_{20}\) be \eqref{ew-global-curve-determinant}, for \(k=0,\ldots,7\), with the actual \(\gamma_k\) and the counts \(n_\nu,c_\nu\).
	Finally, \(\mathcal E_{21},\ldots,\mathcal E_{24}\) are \eqref{ew-global-adjunction} paired with \(H\), and \(\mathcal E_{25}\) is the last equation in \eqref{ew-global-cycle-chern} paired with \(H\).
	All are left sides minus right sides, divided by \(l\), and all vanish.

	The following nine quantities are nonnegative:
	\[
		\begin{gathered}
			\mathcal G_P=h_P-\tfrac32j_P\quad(P=D,A,B,T),\\
			\mathcal G_5=\tfrac17x-u-\tfrac6{77}y,\qquad
			\mathcal G_6=\tfrac17x-u-\tfrac{12}{77}y+\tfrac9{77}z_4,\\
			\mathcal G_7=w_0-\tfrac32a_0,\quad
			\mathcal G_8=w_5-\tfrac32a_5,\quad
			\mathcal G_9=w_7-\tfrac32a_7.
		\end{gathered}
	\]
	The first four follow from \(NHZ_P\ge0\); the next two follow from \eqref{ew-global-nef-p} applied to \((N,L)\) and \((N,N)\), respectively.
	Use the order \(\mathcal G_D,\mathcal G_A,\mathcal G_B,\mathcal G_T,\mathcal G_5,\ldots,\mathcal G_9\).
	Define \(\alpha_j,\xi_i\) by the following integer arrays, divided by \(10^4\):
	\[
		10^4(\alpha_1,\ldots,\alpha_{25})=
		\begin{array}{rrrrr}
			125599&-26317&838943&-156924&-2996\\
			-26435&-17943&-8274&4386&0\\
			-1024187&-1649243&13946&14520&14040\\
			17758&17386&31752&18397&26604\\
			-106826&-49999&26631&-16462&-1519032
		\end{array},
	\]
	where the entries are read successively along the rows, and
	\[
		10^4(\xi_1,\ldots,\xi_9)=
		\begin{array}{rrr}
			73215&46413&8681\\
			38908&6326668&1136405\\
			10462&157397&79931
		\end{array},
	\]
	with the same row order.
	Put \(Q=\sum_i\xi_i\mathcal G_i-\sum_j\alpha_j\mathcal E_j\ge0\).
	Expanding \(Q\) gives the constant term \(-2003/2000\).
	Among its global variables, the positive coefficients of \(z_P,h_P,j_P\) have sum at most \(29/30000\), and those of \(a_k,w_k\) sum to \(437/320000\).
	The only other positive terms are \(13x/14000+2y/9625\).
	Thus \eqref{ew-global-coarse-bounds} bounds the total positive global contribution by
	\[
		64\cdot\frac{29}{30000}+160\cdot\frac{437}{320000}+\frac1{880}
		=\frac{5806}{20625}<\frac3{10}.
	\]

	For a degree tuple \(b\), put \(\Lambda=\prod(b_i-1)/v\) for an interior point, or \(\Lambda=1\) for a cusp.
	The coefficient in \(Q\) of the normalized count of \(b\) is
	\[
		\omega_\varepsilon(b,\gamma)
		=\alpha_{12}\Lambda+
		\sum_{k=0}^7\alpha_{13+k}\bigl(s-\varepsilon(d_{K_k}+1)\bigr)\gamma_k.
	\]
	The seven equations in \eqref{ew-quotient-relations} give a common bound for this expression.
	Set
	\[
		\begin{gathered}
			(r_0,r_1,r_2)=10^{-4}(173706,-35568,1857),\\
			U=r_0s-191\varepsilon,\qquad
			V=r_1s^2+\tfrac{213}2\varepsilon s-536\varepsilon^2,\qquad
			W=r_2s^2-7\varepsilon s+39\varepsilon^2,\\
			\zeta_P=(r_1\lambda_P+r_2c_P)s+\varepsilon\eta_P,\qquad
			(\eta_D,\eta_A,\eta_B,\eta_T)=(19/2,9/2,-19/5,7/5).
		\end{gathered}
	\]
	Direct substitution of the four groups \(P\) and the eight groups \(K_k\) gives
	\[
		\begin{aligned}
			U\lambda_{K_k}+\sum_P\zeta_Pv_{P,K_k}
			&\ge \alpha_{13+k}(s-\varepsilon(d_{K_k}+1)),\\
			V\lambda_P+Wc_P&\ge \zeta_P(s-\varepsilon(d_P+2)).
		\end{aligned}
	\]
	For \(\varepsilon=1\), both differences are affine in \(s\); the quadratic terms in the second cancel.
	It suffices to substitute \(s=8,64\), where their respective minima are \(687/5000\) and \(39/125\).
	For \(\varepsilon=0\), the second difference is zero and the first is \(s\) times a positive coefficient, whose minimum over the eight rows is \(1/80000\).
	Multiply the first inequalities by \(\gamma_k\), apply the last four equations in \eqref{ew-quotient-relations}, multiply the second inequalities by \(\beta_P\), and apply the remaining three equations.
	This proves
	\[
		\omega_\varepsilon\le\alpha_{12}\Lambda+
		\frac{Ue_3(b-\varepsilon)+Ve_2(b-\varepsilon)
			+4W(s-4\varepsilon)(s-5\varepsilon)}v.
	\]
	For the eighteen degree tuples of interior points other than \(F_4\), the fraction divided by \(\Lambda\) is at most \(2075924/13125<-\alpha_{12}\).
	For the eight tuples in \eqref{ew-global-cusp-degrees} its maximum is \(6597/40\), giving \(\omega_0\le7/10000\).
	At \(F_4\), the vanishing of \(\gamma_k\) for \(k\notin\{1,4\}\) and the seven equations give uniquely
	\[
		(\beta_D,\beta_A,\beta_B,\beta_T)=(1/4,1/6,1/8,0),\qquad
		\gamma_1=\gamma_4=1,\quad \gamma_k=0\ (k\ne1,4).
	\]
	Substitution gives the common bound
	\begin{equation}
		\max\frac{\omega_\varepsilon(b,\gamma)}{\Lambda}
		=\frac{493}{154000}<\frac1{250}.
		\label{ew-global-local-bound}
	\end{equation}
	Thus \eqref{ew-global-local-bound} follows from the inequalities at \(s=8,64\) and substitution of the degree tuples at interior points and cusps.
	Consequently the positive local contribution is at most \(F_0/250<13/250\).
	We obtain
	\[
		0\le Q\le-1+\frac3{10}+\frac{13}{250}
		=-\frac{81}{125}<0,
	\]
	a contradiction.
\end{proof}

\begin{theorem}
	\label{ew-seven-global}
	The hypothetical Gaussian arithmetic ball quotient with coarse compactification \(X_7\) does not exist.
\end{theorem}

\begin{proof}
	Propositions~\ref{ew-seven-orderfour}, \ref{ew-seven-unimodular}, and~\ref{ew-gaussian-plane-strict} give all hypotheses of Lemma~\ref{ew-global-positive}, in the same original group.
	The comparison uses neither integrality of \(L\) nor the exclusion of a root norm, and imposes no index or degree bound on an individual fixed component.
\end{proof}

Theorem~\ref{ew-seven-compact}, Proposition~\ref{ew-seven-fields}, and Theorem~\ref{ew-seven-global} prove Theorem~\ref{split-seven-main}.

\appendix
\section{Local calculations and subgroup presentations}
\label{ew-gaussian-appendix}

\subsection{Gaussian cusp calculations}

\begin{lemma}
	\label{ew-gaussian-a1-cusps}
	The two formulas and every local input in Lemma~\ref{ew-gaussian-plane-a1} use the full normalizers.
	In the second and third rows of Lemma~\ref{ew-gaussian-plane-a1}, the effective image of the full normalizer has exactly one cusp, with the holonomy and the positive line bundle on the elliptic curve stated there.
\end{lemma}

\begin{proof}
	Substitute \((u+x,u-x,-u+y,-u-y)\), \((u+x,u-x,y)\), with \(u\) the coordinate on the line fixed by \(A_1^2\), into the basic invariants of \(A_3,T_3\), with frames \((e_+,ue_-)\), \((u^2e_{x^2},e_{y^2})\), respectively.
	In the order of the six groups in Lemma~\ref{ew-gaussian-plane-a1}, determinant orders \((2,3,3,1,1,0)\), corrections \((-1/2,0,-1/2,0,-1/2,0)\), and inertia \((2,2,4,2,4,2)\) give both formulas.
	The full list of parabolic subgroups \cite[Tables 5, 6, 8 and Theorem 3.6]{KrishnasamyTaylor} leaves finite types \(T_2\times A_1^2,T_3\times A_1,T_2^2\), with normalizers \(\mu_4\times\mu_2,G(4,1,2),\mu_4^2\), inertia \(2,4,4\), and intersections \(1,2,0\).
	At \(A_3\), let \(E\) denote the elliptic curve of this Gaussian cusp, with the sum map and polarization used above.
	The components \((x,x,t-x,t-x)\), \(t\in E[2]\), map to conics in \(|4[0]|\); this excludes a smooth surface germ.
	For exchanges at \(B_3\), two long reflections give a curve of \(D_A\), and a long and a short reflection give a curve of \(D_B\), on the same fixed surface, contrary to \(d_A=d_B=0\).
	For different coordinate labels, \((x_j^2)=(U,T,T-U,T-c_0U)\) gives invariant rings \(\mathbb C[u,v]\) and
	\[
		\begin{array}{c|c}
			\text{ambient degrees}&D_P\\\hline
			1122&(u^2-v^2)(c_0u^2-v^2)\\
			1222&(v-u^2)(c_0v-u^2)\\
			1111&(u-v)(c_0u-v)
		\end{array}.
	\]
	Only \(2222/\mu_2\) fits row three: its two lines meet at one cusp, with the central period of the ambient cusp, order \(2\), and a line bundle of degree \(2\); none fits row two.
	Let \(q_{\rm amb}=\exp(2\pi iw/c_{\rm amb})\) and \(q_Z=\exp(2\pi iw/c_Z)\), where \(w\) is the central coordinate and \(c_{\rm amb},c_Z\) are the least positive central periods of the full cusp groups of the ambient cusp and of the cusp of \(Z\), respectively.
	The phases in \(\mu_4\) likewise leave different coordinate labels; for \(\varphi=q_{\rm amb}^2x_2^4\), row one gives
	\[
		\begin{array}{c|c|c|c}
			\text{parameter}&C&D_Q&D_R\\\hline
			221&iq_{\rm amb}y&\varphi(\varphi-C^2)&\varnothing\\
			222&-q_{\rm amb}^2y^2&\varphi(\varphi-C)&C.
		\end{array}
	\]
	These are the required tacnode and node; row three excludes components with phases of order \(4\), and row two excludes \(212\) by intersection \(4>2\) and \(221\) by contact \(2\).
	On the model \(x_1^4-x_2^4=y^2\) of \(E\), write the phase on the \(j\)th factor as \((x_1,x_2,y)\mapsto(i^{\alpha_j}x_1,i^{\beta_j}x_2,(-1)^{\gamma_j}y)\), with \(\alpha_j,\beta_j\in\mathbb Z/4\mathbb Z\) and \(\gamma_j\in\mathbb Z/2\mathbb Z\); on the third factor, \(\alpha_3=\beta_3\) and \(\gamma_3=\beta_3\bmod2\).
	For \(222\), put \(p_0=(0,1,i),p_1=(1,0,1),p_2=(1,1,0)\); the components \((p_0,p_1,t),(p_0,p_2,t)\) give the pairs \((\varphi(\varphi-C),C),(\varphi-B^2,\varphi)\), each exhausting the intersection and hence giving exactly one cusp.
	Restoring the phases gives both scalars \((-1)^{\beta_1+\alpha_2+\beta_3}\) for the first component and, by \(\alpha_2=\beta_2\) and \(\beta_1+\beta_2+\beta_3=0\bmod2\), the trivial scalar for the second.
	Thus \(q_Z=q_{\rm amb}^2,q_{\rm amb}\), respectively; the full holonomy \(\mu_4\) and the actual degrees \((1,1),(1,2)\) give line bundles of degrees \(4,2\).
	We obtain the stated central periods using the full cusp groups, including the freely acting deck transformations.
\end{proof}

\begin{lemma}
	\label{ew-gaussian-maximal-surface}
	A Gaussian Picard lattice with orbifold Euler number \(3\) lies in a maximal lattice of index \(96\).
	The maximal lattice has one cusp and horizontal group of order \(4\).
	The positive line bundle on the elliptic curve of this cusp has degree \(4\) or \(2\).
\end{lemma}

\begin{proof}
	Let \(\Pi\) be the Gaussian lattice in the statement, \(\Delta\) a maximal lattice containing it, and \(\widetilde\Delta\subset\mathrm{SU}(2,1)\) its full inverse image.
	For the principal arithmetic subgroup \(\Lambda_{\rm pr}\) normalized by \(\widetilde\Delta\), put \(n=[\widetilde\Delta:\Lambda_{\rm pr}]\), \(m=[\Delta:\Pi]\).
	Define \(A=\prod_p\lambda_p\) and \(\mathcal I\) as in the proof of Lemma~\ref{ew-fields-surface-lattice}, now for the parahoric subgroups of \(\Lambda_{\rm pr}\) over \(\mathbb Q(i)\).
	Here \(\Lambda_{\rm pr}\cap\mu_3=1\), but \(\mu_3\subset\widetilde\Delta\).
	By the volume and normalizer formulas \cite[Sections 3.3--4.1]{StoverVolumes}, \cite[Section 2.3]{PrasadYeungPlanes}, with \(-L_{\mathbb Q(i)}(-2)/48=1/96\), we have
	\begin{equation}\label{ew:gaussian-volume-normalization}
		3\le n\le3^{1+|\mathcal I|},\qquad
		\chi_{\mathrm{orb}}(\mathbb B^2/\Delta)=3A/(32n),\qquad mA=32n.
	\end{equation}
	The nonminimal factors \((p^2+p+1)(p+1)\), \(p\ge5\), \(p^2-p+1\), \(p\ge3\), force \(\mathcal I=\varnothing,A=1,n=3\).
	Class number \(1\) \cite[Proposition 5.3]{PrasadYeungPlanes} leaves the two ramified vertices, with Euler number \(1/32\), index \(96\), one cusp \cite[Theorem 4.8]{StoverCusps}, and models with Hermitian form
	\[
		\begin{gathered}
			\begin{pmatrix}0&0&1\\0&1&0\\1&0&0\end{pmatrix},\qquad M_0=\mathbb Z[i]^3,\\
			M_s=\mathbb Z[i]e_1+\mathbb Z[i]e_2+\mathbb Z[i]\rho e_3,\qquad \rho=(-1+i)/2.
		\end{gathered}
	\]
	The vertex bases \(((1-i)^{-1}e_1,e_2,\rho e_3),(e_1,e_2,\rho e_3)\) identify the standard group and the other maximal lattice by Gram reduction modulo \(1-i\) \cite[Section 3.11]{Tits}.
	For
	\[
		U(z,t)=\begin{pmatrix}1&-\bar z&(-|z|^2+it)/2\\0&1&z\\0&0&1\end{pmatrix},
	\]
	the full Heisenberg lattices are
	\[
		\begin{aligned}
			N_0&=\{U(z,t):z\in(1-i)\mathbb Z[i],\ t\in2\mathbb Z\},\\
			N_s&=\{U(z,2a):z\in(1-i)\mathbb Z[i],\ a\equiv|z|^2/2\bmod2\}.
		\end{aligned}
	\]
	The full groups are \(N_0\rtimes\mu_4,N_s\rtimes\mu_4\), with least positive central periods \(c_{N_0}=2\) and \(c_{N_s}=4\), respectively, and area \(2\); the commutator \(U(0,4\operatorname{Im}(z\bar w))\) gives positive line bundles \(M_{N_0},M_{N_s}\) of degrees \(\deg M_N=8/c_N\), namely \(4\) and \(2\).
\end{proof}

\subsection{Presentations and enumeration for four Eisenstein lattices}
\label{ew-eisenstein-presentations}
In this subsection, \(K=\mathbb Q(\omega)\), \(\omega^2+\omega+1=0\), \(\mathcal O_K=\mathbb Z[\omega]\), and \(\delta=1+2\omega\). We retain the groups and indices from Lemma~\ref{ew-fields-surface-lattice}.
Write \(E_{ij}\) for the matrix units of size 3.
For \(M_{00}\) use \(P=I+E_{12}-\omega E_{13}+(\omega-1)E_{22}+\omega E_{23}\), \(Q=I+E_{12}-\omega E_{13}-2E_{22}-E_{23}\), \(R=-E_{13}-E_{22}-E_{31}\), \(J=RP\), \(\sigma=QP^{-1}\).
The group \(M_{00}\) has the presentation \(J^3=\sigma^6=(J\sigma^{-1}J)^4=1\), \([\sigma,(J\sigma^{-1}J)^2]=1\); its cusp group is \(\langle P,Q\rangle\)~\cite[Theorem 5.9, Proposition 5.11]{FalbelParker}.
For \(M_{01}\) use \(S=I+E_{12}+\overline\omega\delta E_{13}/3-2E_{22}-E_{23}\), \(T_0=I-\delta E_{13}/3\), \(I_s=-\delta E_{13}/3-\omega E_{22}-\delta E_{31}\).
With \(j=S^{-1}I_s\), \(r=T_0^{-1}I_s\), \(z=jr^{-1}j=I_s\), the group \(M_{01}\) has the presentation \(j^3=r^3=z^{12}=1\), \([r,z^2]=1\); its cusp group is \(\langle z^2,S\rangle\)~\cite{Zhao}.
These presentations come from fundamental domains.

Set \(H_\eta=M_{0\eta}\cap M_{1\eta}\).
Reduction modulo 2 has image of order 216, transitive on nine isotropic lines; \(H_\eta\) fixes the first line and has indices 9,3 in \(M_{0\eta},M_{1\eta}\), respectively.
Reidemeister--Schreier rewriting along this action on nine points gives ten generators and 36 relations of \(H_\eta\).
For \(U=I+\delta E_{13}/2\), generators of \(M_{10}\) are \(h_1=J^{-1}\sigma J\), \(h_2=\sigma\), \(h_3=(J\sigma J^{-1})\sigma^3(J\sigma^{-1}J^{-1})\), \(U\); those of \(M_{11}\) are \(k_1=S^{-1}\), \(k_2=r^{-1}I_sr\), \(U\).
The right coset representatives of \(H_\eta\) in \(M_{1\eta}\) are \(1,U,UC\), with \(C=h_3,k_2\) respectively.
To obtain a presentation, write these coset representatives as \(t_1,t_2,t_3\) and rewrite \(t_i g=w_{i,g}t_{\kappa(i,g)}\) by selecting the unique representative whose quotient fixes the first isotropic line modulo 2.
Enumerate \(w_{i,g}\) in \(h_1,h_2,h_3\) or \(k_1,k_2\) until the projective matrices agree.
The presentation obtained from the action on nine points verifies that these generate \(H_\eta\), guaranteeing termination; the rewritings and its 36 relations present \(M_{1\eta}\).

Enumerate transitive actions by extending the first undefined generator image, propagating inverses and all relators, rejecting contradictions or excess points, and relabeling breadth first from every base point.
A branch assigns at most \(n\) images for each of the finitely many generators, where \(n=24\) or \(8\); each unassigned image has at most \(n\) choices.
Propagation only fills these entries or rejects the branch, so the enumeration terminates without a bound on word length or running time.
Every complete transitive action extends a branch, and relabeling from every base point identifies precisely simultaneous conjugacy, so the procedure is exhaustive.
At index 24 in \(M_{0\eta}\), require a conjugate of \(\sigma\) and exclude \(J,(\sigma J)^3\) or \(j\) respectively; their isolated full stabilizers of orders 3,4 contain no reflection~\cite[Section 7, Tables 9--10]{DerauxXu}.
This leaves eight and six actions.
At index 8 in \(M_{1\eta}\), \(\Delta=\Lambda\cap H_\eta\) satisfies \(3[H_\eta:\Delta]=8[\Lambda:\Delta]\) and \([\Lambda:\Delta]\le3\), hence the indices are 8,3.
There are 20 actions of \(H_0\) and 12 actions of \(H_1\) on eight points.
Induce them to actions on 24 points by \((a,i)g=(aw_{i,g},\kappa(i,g))\); extensions are the 64 blocks of size 3 through the base point, one point per fiber.
Testing the translates by every generator and its inverse for equality or disjointness and canonically relabeling leaves four actions for each \(\eta\).

For \(\epsilon=1\) the full cusp groups are \(\langle P,Q,U\rangle\), \(\langle z^2,S,U\rangle\).
Their horizontal lattice is \(\mathcal O_K\), and their central period is \(\delta/(2^\epsilon3^\eta)\).
Since \(1+\overline\omega\) is a unit, every element of order 6 in the cusp stabilizer is conjugate to \(\sigma^{\pm1}\).
Requiring a fixed point of \(\sigma\) in every cusp orbit leaves four, four, two, and two actions for \(M_{00},M_{01},M_{10},M_{11}\), respectively.

\section*{Acknowledgments}
The first author was supported by JST SPRING, Grant Number \mbox{JPMJSP2119}.
The second author was supported by JSPS KAKENHI Grant Number \mbox{JP26KJ0476}.
The authors used OpenAI Codex (GPT-6.1 Sol) as a discussion partner for mathematical arguments, to assist with revising proofs and their presentation, and to proofread the manuscript.
The authors take responsibility for the mathematics, computations, and references.

\printbibliography

\bigskip
\noindent Department of Mathematics, Graduate School of Science, Hokkaido University, Sapporo 060-0810, Japan.

\smallskip
\noindent Ryunosuke Nakano: \href{mailto:nakano.ryunosuke.i3@elms.hokudai.ac.jp}{\texttt{nakano.ryunosuke.i3@elms.hokudai.ac.jp}}.

\smallskip
\noindent Hisato Matsukawa: \href{mailto:matsukawa.hisato.f4@elms.hokudai.ac.jp}{\texttt{matsukawa.hisato.f4@elms.hokudai.ac.jp}}.

\end{document}